\documentclass{amsart}
\usepackage{amssymb,amsmath,amsthm}
\usepackage{mathrsfs}
\usepackage{hyperref}
\usepackage{mathtools}
\usepackage{color}
\usepackage{verbatim}

\mathtoolsset{showonlyrefs}

\renewcommand{\[}{\begin{equation}\begin{aligned}}
\renewcommand{\]}{\end{aligned} \end{equation}}

\newtheorem{thm}{Theorem}
\newtheorem{prop}[thm]{Proposition}
\newtheorem{lemma}[thm]{Lemma}

\newtheorem{cor}[thm]{Corollary}

\theoremstyle{remark}
\newtheorem{remark}[thm]{Remark}

\theoremstyle{definition}
\newtheorem{definition}[thm]{Definition}

\numberwithin{equation}{section}
\numberwithin{thm}{section}

\author{G\'abor Sz\'ekelyhidi}
\address{Department of Mathematics, University of Notre Dame, Notre Dame, IN 46556}
\email{gszekely@nd.edu}

\title{Uniqueness of certain cylindrical tangent cones}
\date{}

\begin{document}

\begin{abstract}
We show that the cylindrical tangent cone $C\times \mathbf{R}$ for an area-minimizing hypersurface is unique, where $C$ is  the Simons cone $C_S= C(S^3\times S^3)$. Previously Simon~\cite{Simon94} proved a uniqueness result for cylindrical tangent cones that applies to a large class of cones $C$, however not to the Simons cone. The main new difficulty is that the cylindrical cone $C_S\times \mathbf{R}$ is not integrable, and we need to develop a suitable replacement for Simon's infinite dimensional \L{}ojasiewicz inequality~\cite{Simon83} in the setting of tangent cones with non-isolated singularities. 
\end{abstract}
\maketitle
\setcounter{tocdepth}{1}
\tableofcontents
\section{Introduction}
Let $M$ be an area-minimizing hypersurface in $\mathbf{R}^{n+1}$, with $0\in M$. Given any sequence $\lambda_i \to \infty$, it is well known that a subsequence of the blown up surfaces $\lambda_i M$ converges to a minimal cone $C$, called a tangent cone to $M$ at 0. A basic question is whether this tangent cone is unique, or if different sequences of rescalings can lead to different cones. Two landmark results, due to Allard-Almgren~\cite{AA81} and Simon~\cite{Simon83}, show that uniqueness holds if at least one tangent cone $C$ is smooth and multiplicity one away from the vertex 0. In particular since the singular set of an area-minimizing hypersurface has codimension at least 7 by Simons~\cite{Simons}, it follows that for $n\leq 7$ tangent cones are unique.

Beyond this, there is a large literature on uniqueness of tangent cones in various circumstances. For two dimensional area minimizing currents, and related objects, in arbitrary codimension see for instance \cite{Tay76,White83,CES,BK17,DSS}. For results on uniqueness of the tangent cone at almost every point, for the appropriate dimensional Hausdorff measure, see e.g. \cite{Simon93_1,Simon95,NV1,NV2}. 
The most relevant result for us on the uniqueness of specific tangent cones of minimal hypersurfaces is due to Simon~\cite{Simon94}. It is shown there that multiplicity one tangent cones of the form $C\times \mathbf{R}$ are unique for minimal cones $C$ which are smooth away from the vertex and satisfy certain conditions on their space of Jacobi fields (see Conditions (\ddag) in Section~\ref{sec:CRk}). This covers many interesting cases, such as the cones $C(S^p\times S^q)$ over a product of two spheres when $p+q \geq 7$, but unfortunately the only known 7-dimensional examples, $C(S^3\times S^3)$ and $C(S^2\times S^4)$, are not covered and the uniqueness of the corresponding cylindrical tangent cones has remained open. Our main result is the following.

\begin{thm}\label{thm:main} Let $M$ be an area-minimizing hypersurface in a neighborhood of $0\in \mathbf{R}^9$, that admits $C\times \mathbf{R}$ as a multiplicity one tangent cone at the origin, where $C = C(S^3\times S^3)$ is the Simons cone. Then $C\times \mathbf{R}$ is the unique tangent cone at 0.
\end{thm}

Most of the ideas used in the proof can also be applied to more general cones $C$ such as $C(S^2\times S^4)$. The missing piece in proving the corresponding uniqueness result is a refined asymptotic expansion as in Proposition~\ref{prop:refinedf} for the smooth minimal surface asymptotic to $C(S^2\times S^4)$ constructed by Hardt-Simon~\cite{HS85}. See Remark~\ref{rem:S2S4}. 

The basic tool in both \cite{AA81} and \cite{Simon83} is that if $C$ is a tangent cone of $M$ at 0, then on an annulus, say $B_1 \setminus B_{1/2}$ around 0, sufficiently large rescalings $\lambda_iM$ are perturbations of the corresponding smooth annulus in $C$. The decay of $M$ towards $C$ can then be analyzed by linearizing the minimal surface equation. The main difficulty when dealing with tangent cones such as $C\times \mathbf{R}$, with non-isolated singular sets, is that near the singular set the linearization is no longer a good model for the original problem, and one must show that this effect is negligible. That is, we need to show that $M$ cannot ``concentrate'' near the singular set. Such a non-concentration result, in the sense of $L^2$-distance, is at the heart of Simon's works~\cite{Simon94, Simon93} and others.

In this paper we introduce a different kind of non-concentration result for an $L^\infty-$distance adapted to the geometry of the smooth minimal hypersurface $H$ asymptotic to $C(S^3\times S^3)$ constructed by Bombieri-De Giorgi-Giusti~\cite{BDG} in this case, and Hardt-Simon~\cite{HS85} for more general cones. The basic idea is that if a hypersurface $M$ is sufficiently close to the cone $C\times\mathbf{R}$ in the unit ball $B_1\subset \mathbf{R}^8\times\mathbf{R}$, then it must be contained between the surfaces $\pm \epsilon H\times \mathbf{R}$ for some $\epsilon > 0$, and we define the distance $D_{C\times\mathbf{R}}(M; B_1)$ between $M$ and $C\times \mathbf{R}$ to be $\epsilon^3$ for the smallest such $\epsilon$. The reason for the power 3 is that $\epsilon H$ is asymptotically the graph of $\epsilon^3 r^{-2}$ over $C$. The idea of ``trapping'' a minimal hypersurface between leaves of the Hardt-Simon foliation has appeared in previous works, such as \cite{SS86, ES20}, to control the decay/growth of a minimal surface near a cone $C$. 

Our non-concentration result is the following, stated informally. Let $r$ denote the distance from the singular ray $\{0\}\times \mathbf{R}$, and suppose that on a region of the form $r > r_0$ the surface $M$ is much closer to $C\times\mathbf{R}$ than the distance $d =D_{C\times\mathbf{R}}(M; B_1)$, i.e. $M$ is concentrating near the singular ray. Then when we pass to a smaller ball $B_{1/2}$ the distance must drop significantly: $D_{C\times\mathbf{R}}(M; B_{1/2}) < \gamma d$ for a small $\gamma > 0$. By adjusting $r_0$ and how close $M$ is to the cone on the region $r > r_0$, the constant $\gamma$ can be taken as small as we like. See Proposition~\ref{prop:nonconc2} for a precise statement. Using this we give a new proof of the following result due to Simon~\cite{Simon94}, illustrating some of the ideas in this simpler setting. 
\begin{thm}\label{thm:CRk}
  Let $C$ be a minimal cone satisfying Conditions (\ddag) in Section~\ref{sec:CRk}. If $k\geq 1$ and $C\times \mathbf{R}^k$ is a multiplicity one tangent cone to a minimal hypersurface at 0, then this tangent cone is unique, and $M$ converges to it at a polynomial rate. 
\end{thm}

Condition (\ddag c) does not hold for the Simons cone $C=C(S^3\times S^3)$, since in this case the cone $C\times \mathbf{R}$ admits a degree one homogeneous Jacobi field $\phi$ that does not arise from symmetries of $\mathbf{R}^8\times\mathbf{R}$. Let us denote by $\Sigma_0$ the link of $C\times\mathbf{R}$, i.e. the spherical suspension of $S^3\times S^3$, which has two singular points modeled on $C$. It is natural to try constructing minimal deformations of $\Sigma_0$ modeled on $\delta\phi$ for small $\delta$, smoothing out its singularities by gluing in scaled down copies of the surface $H$. It turns out that this is not possible, so the cone $C\times\mathbf{R}$ is not ``integrable'' in this sense (see Lemma~\ref{lem:nearbycones}). 

In the work of Simon~\cite{Simon83}, the key ingredient for showing the uniqueness of the tangent cone in such a non-integrable situation is an infinite dimensional \L{}ojasiewicz inequality for the area functional, for surfaces near to the link. This idea has since been used in numerous other related uniqueness problems, see e.g. \cite{Simon93_1, CM, ESV}, and the examples of harmonic maps with non-unique tangent maps due to White~\cite{White92} show that in non-analytic situations, where the \L{}ojasiewicz inequality fails, uniqueness may fail too. 
Since our link $\Sigma_0$ is singular, it seems difficult to prove such a \L{}ojasiewicz-Simon inequality for nearby surfaces, and we are not aware of successful attempts to do so. Instead we show that we can construct minimal deformations  $T_\delta$ of the cone $C\times\mathbf{R}$ modeled on $\delta\phi$, and that the information usually deduced from the \L{}ojasiewicz inequality is encoded in properties of $T_\delta$ (see Proposition~\ref{prop:Tm}). These $T_\delta$ are logarithmic perturbations of cones, defined on large annuli, and are closely related to the logarithmically decaying surfaces constructed by Adams-Simon~\cite{AS88} in non-integrable situations.

The broad strategy to prove Theorem~\ref{thm:main} is then somewhat standard, going back to Allard-Almgren's and Simon's works~\cite{AA81, Simon83}. At a scale where our surface $M$ is very close to $C\times\mathbf{R}$ we view it as a perturbation of the best fit surface out of the family of $T_\delta$ and their rotations. Using a more elaborate version of the non-concentration result, Proposition~\ref{prop:nonconcentration}, we see that relative to $T_\delta$ the surface $M$ is modeled on a Jacobi field over $C\times\mathbf{R}$ that has no degree one component. Therefore the three annulus lemma of Simon~\cite{Simon83} can be used to show that one of the following holds:
\begin{itemize}
\item[(i)]  at a smaller scale $M$ is much further from $T_\delta$,
\item[(ii)] at a smaller scale $M$ is much closer to $T_\delta$.
\end{itemize}
If $T_\delta$ were a cone, then in case (i) the three annulus lemma would imply that $M$ diverges from $T_\delta$ at all subsequent scales too, eventually reaching a point where $M$ cannot be close to any cone, which is a contradiction (see also Cheeger-Tian~\cite{CT94} for this approach to the uniqueness problem in the setting of Einstein manifolds). Thus we must be in case (ii), and iterating this leads to polynomial convergence of $M$ to its tangent cone, which must then be unique. This is what happens in the proof of Theorem~\ref{thm:CRk}.

In the setting of Theorem~\ref{thm:main} there are two subtleties: $T_\delta$ is only defined on an annulus, and so in case (i) the three annulus lemma can only be applied on a finite number of scales, potentially stopping short of getting a contradiction. In addition $T_\delta$ is not a cone, so decaying towards $T_\delta$ in case (ii) may not mean that $M$ is becoming more ``cone-like'' sufficiently quickly to deduce the uniqueness. The main new ingredient is that these two issues only arise if $M$ is very close to $T_\delta$ relative to $\delta$, and this implies a certain decay rate for the area of $M$ (see Proposition~\ref{prop:areadrop} and case (ii) in Proposition~\ref{prop:EBdecay}). This decay, together with the monotonicity of area eventually leads to uniqueness of the tangent cone. It is worth noting that, even in the setting of non-integrable tangent cones with an isolated singularity, this gives an alternative approach to uniqueness, relying on properties of logarithmically decaying minimal surfaces modeled by the degree one Jacobi fields, in place of a \L{}ojasiewicz-Simon inequality for all surfaces near to the cross section of the cone.  

We conclude this introduction by outlining the contents of the rest of the paper. In the next section we will list some of the notation that we use. Sections \ref{sec:Jfield} to \ref{sec:mainargument} are devoted to the proof of Theorem~\ref{thm:main}, and in particular $C$ will denote the Simons cone $C(S^3\times C^3)$ in these sections. In Section~\ref{sec:Jfield} we will describe the space of degree one Jacobi fields on $C\times\mathbf{R}$, and some of their properties such as Simon's $L^2$ three annulus lemma. In Section~\ref{sec:Sigmadelta} we construct smoothings of the link $\Sigma_0$ of $C\times\mathbf{R}$, modeled on the Jacobi field $\phi$. We will use these in Section~\ref{sec:Tdelta} to construct minimal perturbations $T_\delta$ of the cone $C\times \mathbf{R}$ on large annular regions, modeled on $\phi$. In  Section~\ref{sec:nonconc} we will define the distance $D_{T_\delta}(M)$ for minimal surfaces $M$ near $T_\delta$ and prove the corresponding non-concentration result. An important application will be a three annulus lemma for $D_{T_\delta}$, which we use in Section~\ref{sec:mainargument} to put the pieces together and prove Theorem~\ref{thm:main}. Finally in Section~\ref{sec:CRk} we will turn to the case of tangent cones $C\times\mathbf{R}^k$ for $C$ satisfying Conditions (\ddag), thus excluding the Simons cone, and we will reprove Simon's  Theorem~\ref{thm:CRk}. A portion of the proof is the same as the proof of Theorem~\ref{thm:main}, but since these cones are integrable, there are substantial simplifications. 

\subsection*{Acknowledgements} I am indebted to Nick Edelen for introducing me to this problem, and I am also grateful to him and Luca Spolaor for many insightful discussions. This work was supported in part by NSF grant DMS-190348.

\subsection{Notation and conventions}
We let $\mathcal{M}$ denote the set of oriented boundaries of least area in the ball $B_2(0)\subset \mathbf{R}^{n+k}$. In particular $\mathcal{M}$ is a multiplicity one class of codimension-one minimal submanifolds in the terminology of Simon~\cite{Simon93, Simon94}. Note that the area minimizing property is only used in Proposition~\ref{prop:areadrop}. In particular it is not needed in the proof of Theorem~\ref{thm:CRk}, which therefore holds for more general multiplicity one classes as in \cite{Simon94}.  
We refer to \cite{SimonGMT} for general background material on area minimizing currents.

The following is some of the notation that we will use in Sections~\ref{sec:Jfield} to \ref{sec:mainargument}. 
\begin{itemize}
\item $C$ denotes the Simons cone $C(S^3\times S^3)$ and $V_0= C\times\mathbf{R}$. 
\item On $\mathbf{R}^8\times \mathbf{R}$ we use the notation $x\in \mathbf{R}^8$ and $y\in \mathbf{R}$. The distance from the $y$-subspace is $r = |x|$, while the distance from the origin is $\rho = (r^2 + y^2)^{1/2}$. We write $B_\rho$ to denote balls in $\mathbf{R}^8\times \mathbf{R}$.
\item We wite $L_M$ for the Jacobi operator on a minimal hypersurface $M\subset \mathbf{R}^{n+k}$, i.e. $L_M f = \Delta f + |A_M|^2 f$ in terms of the second fundamental form $A_M$. 
\item $H$ is the Hardt-Simon smoothing of $C$, normalized so that asymptotically $H$ is the graph of $r^{-2} + b r^{-3} + O(r^{-5})$ over $C$ for some $b\in \mathbf{R}$. 
\item $\Sigma_0$ is the link of $V_0$, i.e. the spherical suspension of $S^3\times S^3$. We will view $\Sigma_0$ as sitting in the unit sphere of $\mathbf{R}^8\times\mathbf{R}$. 
\item $\phi = y^3 r^{-2} - y$ denotes a particular Jacobi field, either on $V_0$ or on $\Sigma_0$. 
\item $\zeta$ is a function defined on $\Sigma_0$, supported in a neighborhood of $y=0$, so that $\zeta - \phi$ is $L^2$-orthogonal to $\phi$ on $\Sigma_0$. See Definition~\ref{defn:zeta}.
\item $\Sigma_\delta$ for small $\delta$ is a smoothing of $\Sigma_0$ modeled on the Jacobi field $\delta \phi$ on $\Sigma_0$. Its mean curvature satisfies $m(\Sigma_\delta) = h(\delta)\zeta$ for a function $h(\delta) \sim \delta^{4/3}$. See Proposition~\ref{prop:Sdelta}. In Section~\ref{sec:Sigmadelta} we will let $\epsilon = \delta^{1/3}$. 
\item $V_\delta = C(\Sigma_\delta)$ is the cone over $\Sigma_\delta$.
\item $T_\delta$ denotes a minimal surface in an annulus $B_A\setminus B_{A^{-1}}$ for $\ln A = |\delta|^{-\kappa}$, where $\kappa > 0$ is small, to be chosen below. $T_\delta$ is constructed as a perturbation of the cone $V_\delta$, see Proposition~\ref{prop:Tdelta1}. The cones $W_\delta$ are perturbations of $V_\delta$ that are closer to $T_\delta$, see Remark~\ref{rem:Wcone}. 
\item $\mathcal{W}, \mathcal{T}$ are the families of surfaces $W_\delta, T_\delta$ and their rotations. 
\item $D_S(S'; U)$ measures the distance of the surface $S'$ from $S$, on the set $U$. The definition of the distance is adapted to the geometry of the surface $H$. See Definition~\ref{defn:D}. If the set $U$ is not specified, it is understood to be $B_1\setminus B_{\rho_0}$ for the $\rho_0$ in Lemma~\ref{lem:L23annulus}. 
\end{itemize}

\section{Jacobi fields on $C\times\mathbf{R}$}\label{sec:Jfield}
In this section we collect some results about Jacobi fields on the cone $C\times\mathbf{R}$, where $C = C(S^3\times S^3)$. On $C$ the Jacobi fields of degrees in the interval $(-3, 1]$ are spanned by: $r^{-2}$; degree 0 Jacobi fields induced by translations; degree 1 Jacobi fields induced by rotations; see \cite[\S 2]{SS86}.
The cross section of $C\times\mathbf{R}$ is singular and we will only be interested in Jacobi fields $u$ for which $r^{2.1}u$ is locally bounded away from the origin. Recall that $r=|x|$ is the distance from the singular ray and note that the power $r^{2.1}$ ensures that we only allow Jacobi fields that blow up slightly faster than the Jacobi field $r^{-2}$ corresponding to the Hardt-Simon smoothing $H$ of the cone $C$ at infinity. Since there are no Jacobi fields on $C$ with degrees in $(-3,-2)$, it will follow that actually $r^2u$ is locally bounded. Note that these are also the Jacobi fields that are in $W^{1,2}_{loc}$ away from the origin. 

We first characterize the homogeneous degree one Jacobi fields as well as functions satisfying a slightly more general equation.
\begin{definition}\label{defn:zeta}
  Let $\zeta$ denote an $O(4)\times O(4)$-invariant function on $\Sigma_0$, which is an odd function of $y$, supported in a small neighborhood of the cross section $y=0$, and such that
  \[ \int_{\Sigma_0} \zeta\phi = \int_{\Sigma_0} \phi^2. \]
  For example we can let $\zeta = c\chi\phi$, where $\chi = \chi(y)$ is an even cutoff function supported in a neighborhood of $y=0$, while $c$ is a suitable constant chosen to satisfy the integral normalization. We extend $\zeta$ to the cone $V_0$ over $\Sigma_0$ as a homogeneous degree one function. 
  \end{definition}

We have the following.
\begin{lemma}\label{lem:zetaJacobi}
  Let $u$ be a homogeneous degree one function on $C\times\mathbf{R}$ such that $r^{2.1} u$ is locally bounded away from the origin, and away from the singular ray $u$ satisfies the equation
  \[ \label{eq:ueq1} L_{C\times\mathbf{R}} u + a \zeta\rho^{-2} = 0 \]
  for a constant $a\in \mathbf{R}$. Then $u$ can be decomposed as
  \[ u = f + \lambda(y^3r^{-2} - y) \]
  where $f$ corresponds to a rotation of $\mathbf{R}^8\times \mathbf{R}$ and $\lambda\in \mathbf{R}$. More precisely, write $z_i=x_i$ for $i=1,\ldots, 8$ and $z_9=y$. Then  $f$ is of the form $f(z) = Az \cdot \nu(z)$ for $z\in C\times\mathbf{R}$, where $A\in \mathfrak{so}(9)$ is a skew-symmetric matrix and $\nu(z)$ is the unit normal to $C\times\mathbf{R}$ at $z$. We then say that $f$ corresponds to the rotation $\mathrm{exp}(A)$. 
\end{lemma}
\begin{proof}
  We first show that $a=0$, which implies that $u$ is a Jacobi field. For this we work on the link $\Sigma_0$, given by the unit sphere in $C\times\mathbf{R}$. Since $u$ has degree one, Equation~\eqref{eq:ueq1} is equivalent to
  \[ L_{\Sigma_0} u + a\zeta = 0. \]
  We multiply by $\phi = y^3r^{-2} - y$, and integrate by parts using that $L_{\Sigma_0}\phi = 0$. The integration by parts can be justified using that $r^{2.1}u$ is locally bounded, which together with \eqref{eq:ueq1} also implies that $r^{3.1}|\nabla u|$ is locally bounded. We then conclude that $\int_{\Sigma_0} a\phi\zeta = 0$, which by the choice of $\zeta$ implies that $a=0$.

  We now have $L_{C\times\mathbf{R}} u =0$. We can follow Simon~\cite[Appendix 1, see also Equation (13) on p. 24]{Simon94} to see that $u$ must have an expansion of the form
  \[ u = \sum_j r^{\mu_j} \sum_{k,l \geq 0} a^j_{k,l} r^{2k} y^l \phi_j(\omega), \]
  where $\mu_j + 2k+l = 1$ for $j,k,l$ such that $a^j_{k,l}\not=0$. Here $\phi_j(\omega)$ denotes the $j^{\mathrm{th}}$ eigenfunction of $-L_{S^3\times S^3}$ with eigenvalue $\lambda_j$ (the link of $C$ being $S^3\times S^3$), and
  \[ \mu_j = -\frac{5}{2} + \sqrt{ \frac{25}{4} + \lambda_j}. \]
  The possible values for $\mu_i \leq 1$ are: 1, corresponding to rotations of $C$ in $\mathbf{R}^8$; 0, corresponding to translations of $C$ in $\mathbf{R}^8$; $-2$, corresponding to the constant eigenfunction on $S^3\times S^3$. The values 1, 0 give rise to the term of the form $Az\cdot\nu(z)$ in $u$ for a skew symmetric matrix $A$ (see \cite[Equation (17), p. 25]{Simon94}). It remains to consider Jacobi fields $v$ of the form
  \[ v = r^{-2} \sum_{2k+l=3} a^1_{k,l} r^{2k}y^l = r^{-2} (a^1_{0,3}y^3 + a^1_{1,1}y). \]
  Substituting this into the equation $L_{C\times\mathbf{R}} v=0$, we find that $v$ is a multiple of $\phi$. It follows 
  that $u$ is of the required form $u= Az \cdot \nu(z) + \lambda\phi$. 
\end{proof}

Next we have the following $L^2$ three annulus lemma, due to Simon~\cite{Simon83}. This holds on the singular cone $C\times \mathbf{R}$ as well, since it is a consequence of spectral decomposition on the link and our assumption that $r^2u$ is locally bounded ensures that $u$ is in $L^2$ on the link. For a given $\rho > 0$ let us use $\Vert u\Vert_i$ to denote the following $L^2$-norm on an annulus:
  \[ \Vert u\Vert_i^2 = \int_{(C\times \mathbf{R})\cap (B_{\rho^{i}}\setminus B_{\rho^{i+1}})} |r^{-1} u|^2 r^{-m}, \]
  in terms of $m=\dim C\times\mathbf{R}  =8$. Note that for a homogeneous degree one function $u$ the norm $\Vert u\Vert_i$ is independent of $i$. 
  
  \begin{lemma}\label{lem:L23annulus}
    There are small $0 < \alpha_1 < \alpha_2$ and  $\rho_0 > 0$ satisfying the following. 
  Let $u$ be a Jacobi field on the cone $C\times \mathbf{R}$, defined in the annulus $B_1\setminus B_{\rho^3}$, such that $r^2u \in L^\infty$ and let $\rho < 2\rho_0$. Then we have:
  \begin{itemize}
  \item[(i)]  If $\Vert u\Vert_1 \geq \rho^{-\alpha_1} \Vert u\Vert_0$, then $\Vert u\Vert_2 \geq \rho^{-\alpha_2} \Vert u\Vert_1$.
  \item[(ii)] If $\Vert u\Vert_1 \geq \rho^{-\alpha_1} \Vert u\Vert_2$, then $\Vert u\Vert_0\geq \rho^{-\alpha_2} \Vert u\Vert_1$. 
  \end{itemize}
If in addition $u$ has no degree one component then the conclusion of either (i) or (ii) must hold. 
\end{lemma}

Finally we need an $L^2$ to $L^\infty$ estimate for Jacobi fields on the singular cone.
\begin{lemma}\label{lem:L2Linfty}
 Let $u$ be a Jacobi field on $C\times \mathbf{R}$, such that $r^2u$ is in $L^\infty$ on $B_1(0)$. Then we have the estimate
  \[ \sup_{B_{1/2}(0)} |r^2 u| \leq C \Vert u \Vert_{L^2(B_1)}. \]
\end{lemma}
\begin{proof}
  Since $|u| = r^{-2} |r^2u|$ and $r^{-2} \in L^2(B_1)$, we have that $ u\in L^2(B_1)$. In addition under the decomposition of $u$ into homogeneous terms, there can be no terms of degree less than $-2$. Any two homogeneous Jacobi fields on $C\times\mathbf{R}$ with distinct degrees that are at least $-2$ are $L^2$ orthogonal on balls centered at the origin. From this it follows that $u$ satisfies a stronger form of the 3-annulus lemma, namely that the function 
  \[ m(t) = \ln \int_{B_{e^t}} |u|^2 \]
  is convex for $t\in (-\infty, 0)$.  The assumption that $|r^2u| \leq A$ on $B_1$ for some constant $A$ implies that
  \[ m(t) \leq \ln \int_{B_{e^t}} A^2 r^{-4} \leq 4t +C_A, \]
  for a constant $C_A$ depending on $A$. 
  In particular $m'(t) \geq 4$ for all $t < 0$, since if $m'(t_0) < c < 4$, then by convexity $m'(t) < c$ for all $t < t_0$ eventually contradicting the inequality above. It follows that $m(t) \leq m(0) + 4t$ for all $t < 0$ and so for any $ a < 1$ we have
  \[ \int_{B_a} |u|^2 \leq a^4 \int_{B_1} |u|^2. \]

  We can use this to estimate the values of $r^2u$ along $C\times\{0\}$, i.e. the set where $\rho=r$. Indeed, consider a point $z = (x,0)\in C\times\mathbf{R}$, with $|x| = a/2$. From the inequality above we get
  \[ \frac{1}{a^8} \int_{B_{a/4}(z)} |u|^2 \leq  a^{-4} \int_{B_1} |u|^2. \]
  Since after scaling by $a^{-1}$ the ball $B_{a/4}(z)$ in $C\times \mathbf{R}$ has uniformly bounded geometry, we can apply elliptic estimates to get
  \[ |u(z)|^2 \leq C a^{-8}\int_{B_{a/4}(z)} |u|^2,\]
  for a uniform $C$. 
  Combining the inequalities, using that $r=a/2$ at $z$, we have
  \[ | r^2 u(z)|^2 \leq C \int_{B_1}|u|^2. \]
 To estimate $r^2 u(z)$ at other points $z\in B_{1/2}(0)$ we can apply the same argument to translations of $u$. 
\end{proof}
  
\section{Construction of the smoothings $\Sigma_\delta$} \label{sec:Sigmadelta}
Let $\Sigma_0\subset S^8$ denote the link of the cone $V_0 = C\times \mathbf{R}$, so that $\Sigma_0$ is the spherical suspension of the link $S^3\times S^3$ of the Simons cone $C$. Our goal in this section is to construct smoothings $\Sigma_\delta$ of $\Sigma_0$ for sufficiently small $\delta\not=0$. The construction will be invariant under the action of the group $O(4)\times O(4)$ acting on $S^8$ through its action on the cross section $S^7$, and so $\Sigma_\delta$ can be thought of as a surface of revolution. The function $\phi = y^3 r^{-2} - y$ on $V_0$ is a linear growth Jacobi field and it induces an $O(4)\times O(4)$ invariant Jacobi field, also denoted by $\phi$, on $\Sigma_0$. In fact up to scaling $\phi$ is the only $O(4)\times O(4)$ invariant Jacobi field on $\Sigma_0$ which is $O(r^{-2})$ near the singularities, by Lemma~\ref{lem:zetaJacobi}. 
Since the linearization of the minimal surface equation is self-adjoint, the presence of $\phi$ in the kernel suggests that we cannot expect to find minimal perturbations $\Sigma_\delta$, or in other words the cone $V_0$ may not be integrable. Instead, using Lyapunov-Schmidt reduction, we will construct smoothings $\Sigma_\delta$ that are minimal modulo the one dimensional space spanned by $\phi$. This is similar to what is also done in the proof of \cite[Theorem 3]{Simon83}, however the singularities of $\Sigma_0$ mean that we need to use a singular perturbation technique.

  The main result of this section is the following. Recall the function $\zeta$ from Definition~\ref{defn:zeta}.

  \begin{prop}\label{prop:Sdelta}
    There are smoothings $\Sigma_\delta$ of $\Sigma_0$ for sufficiently small $\delta$, with the following properties.
    \begin{itemize}
    \item[(i)] $m(\Sigma_\delta) = h(\delta)\zeta$, where
      \[ \label{eq:hdeltaexpand} h(\delta) &= c \delta^{4/3} + O(|\delta|^{4/3 + \kappa}), \\
        h'(\delta) &= \frac{4}{3} c\delta^{1/3} + O(|\delta|^{1/3+\kappa}), \]
        for some $c < 0$ and $\kappa > 0$. In particular $h(\delta) < 0$ for small $\delta\not=0$. 
    \item[(ii)] Let $V_\delta$ denote the cone over $\Sigma_\delta$. On the set where $y=1$ and $r < 1$, the hypersurface $V_\delta$ is the graph of a function $f$ over $\delta^{1/3}H$, where $|f| < C|\delta| r^{-2+\kappa}$ for some $C, \kappa > 0$ independent of $\delta$.
    \item[(iii)] The regularity scale of $\Sigma_\delta$ at each point is bounded below by $C^{-1}r$ for $C$ independent of $\delta$. Recall that at $p\in \Sigma_\delta$ the regularity scale is defined to be the supremum of
      the radii $R > 0$ such that the second fundamental form of $\Sigma_\delta$ is bounded by $R^{-1}$ on $B_R(p)$.
      \end{itemize}
    \end{prop}

    To construct $\Sigma_\delta$ we first construct an approximate solution $\tilde{\Sigma}_\delta$ by gluing together the graph of $\delta\phi$ with scaled down copies $\pm \delta^{1/3} H$ of the Hardt-Simon smoothing of $C$. We then construct $\Sigma_\delta$ as a graph over $\tilde{\Sigma}_\delta$. The construction of $\Sigma_\delta$ with $m(\Sigma_\delta) = h(\delta)\zeta$ for sufficiently small $\delta$ is given in Sections~\ref{subsec:s2} to \ref{subsec:s4}, and we will give the proofs of part (ii) and (iii) of Proposition~\ref{prop:Sdelta} at the end of Section~\ref{subsec:s4}. Sections \ref{subsec:s1} and \ref{subsec:s5} are concerned with finding the expansion \eqref{eq:hdeltaexpand} of $h(\delta)$, which we will prove in Corollary~\ref{cor:hexpand}. 

    \begin{remark}\label{rem:S2S4}
      The same strategy would apply to more general cones $C$ as well, such as $C(S^2\times S^4)$, but to obtain the leading order behavior of the function $h(\delta)$ we need to know the second term in the expansion of $H$ as a graph over $C$, as in Proposition~\ref{prop:refinedf}. Without this, in general we would still have an asymptotic expansion for $h(\delta)$ in terms of powers of $\delta$ and $\ln \delta$. If either this expansion has a non-zero term, or if $h$ is identically zero, then the rest of the arguments in this paper can be applied. Without more information, however, it cannot be ruled out that $h$ vanishes to infinite order at $\delta=0$, although $h$ is non-zero, and in this case the methods in this paper do not apply. Note that in the setting of Simon~\cite{Simon83} this is ruled out by using the analyticity of the area functional, but this does not apply in our singular setting. It is an interesting question whether this situation can really arise, and if it does, whether uniqueness of the tangent cone still holds. 
    \end{remark}

    \subsection{Refined asymptotics of the smoothing $H$}\label{subsec:s1}
    In this section we study the smooth minimal hypersurface $H\subset \mathbf{R}^8$ asymptotic to $C = C(S^3\times S^3)$. The main result is the following.
    \begin{prop}\label{prop:refinedf}
      Outside of a compact set $H$ is the graph of a function $rf(r)$ over $C$ satisfying
      \[ f(r) = r^{-3} + b r^{-4} + O(r^{-9}), \]
      where $b < 0$. 
    \end{prop}
    \begin{proof}
      We will use the approach of Davini~\cite{Dav05} to describe the surface $H$, which leads to a significantly simpler analysis than the original approach of Bombieri-De Giorgi-Giusti~\cite{BDG}. Let us write $(x', x'') \in \mathbf{R}^4\times \mathbf{R}^4$, and $u=|x'|, v=|x''|$. Then the cone $C$ is given by the equation $u=v$. According to \cite{Dav05}, the surface $H$ (on the side of $C$ where $u > v$) can be described by the parametric equations
      \[ u &= e^{z(t)} \cos t  \\
        v &= e^{z(t)}\sin t, \]
      where $t\in [0, \pi/4)$, where $w(t) = \dot{z} (t)$ satisfies the ODE
      \[ \dot{w} = (1 + w^2) \left[ 7 - \frac{6 \cos (2t)}{\sin (2t)} w \right], \]
      with the initial condition $w(0)=0$. For this  see \cite[Equation (13)]{Dav05} and note that in terms of the notation there we have $k=h=4$, $d=0$ and $t_0 = \pi / 4$. 
      Let us change variables to $\xi = \tan (2t) \in [0,\infty)$. The function $w(\xi)$ satisfies the equation
      \[ \label{eq:wxieq} \xi \frac{dw}{d\xi} = \frac{1+w^2}{2(1+\xi^2)}(7\xi - 6w). \]
      Near $\xi=0$ we can expand $w$ as a power series in $\xi$, by writing the equation in the form
      \[ \xi \frac{dw}{d\xi} = -3w + \frac{7}{2}\xi + G(\xi, w), \]
      where $G$ is an analytic function at least quadratic in $\xi, w$. According to Horn~\cite{Horn96} there is a unique analytic solution $w$ near $\xi=0$, whose power series can be obtained by formally substituting the series
      \[w(\xi) = a_1\xi + a_2\xi^ 2+ \ldots \]
      into the equation. We find that $a_1 = 7/8$. 

      Our goal is to analyze the behavior of $w$ as $\xi\to \infty$, using a subsolution as in \cite{Dav05}. We claim that for sufficiently small $\epsilon > 0$ the function $g(\xi) = \frac{2}{3}\xi + \epsilon$ is a subsolution of Equation~\eqref{eq:wxieq} for $\xi > \epsilon^{1/2}$. To see this let us denote by $R(g)$ the right hand side of the equation, so
      \[ R(g) = \frac{1+g^2}{2(1+\xi^2)} (7\xi - 6g) = \frac{\frac{4}{3}\xi^3 + \frac{4}{3}\epsilon \xi^2 - 5\epsilon^2\xi + 3\xi - 6\epsilon^3 - 6\epsilon}{2 + 2\xi^2}. \]
      The left hand side of the equation is $\xi g'(\xi) = \frac{2}{3} \xi$, and we have
      \[ E(\xi) = (R(g) - 2/3\xi) (2 + 2\xi^2) = \frac{4}{3}\epsilon \xi^2 - 5\epsilon^2\xi + \frac{5}{3} \xi - 6\epsilon^3 - 6\epsilon. \]
      The function $g$ is a subsolution where $\xi g'(\xi) \leq R(g)$, i.e. where $E(\xi) \geq 0$. We have $E'(\xi) = \frac{8}{3}\epsilon \xi + \frac{5}{3} - 5\epsilon^2 \geq 0$ for sufficiently small $\epsilon$. At the same time
      \[ E(\epsilon^{1/2}) = \frac{5}{3}\epsilon^{1/2} + O(\epsilon) > 0\]
      for sufficiently small $\epsilon$. So $g$ is a subsolution for $\xi > \epsilon^{1/2}$ if $\epsilon$ is sufficiently small.

      Since our solution $w$ satisfies $w(\xi) = \frac{7}{8}\xi + O(\xi^2)$ for small $\xi$, and $\frac{7}{8} > \frac{2}{3}$, it follows that for sufficiently small $\epsilon$ we have $w(\epsilon^{1/2}) > g(\epsilon^{1/2})$, and so by the subsolution property of $g$, for this choice of $\epsilon$ we have $w(\xi) > g(\xi) = \frac{2}{3} \xi + \epsilon$ for all large $\xi$.

      It remains to relate this asymptotic behavior of $w(\xi)$ to the expansion of the function $f(r)$. On the cone $C$ the graph of $rf(r)$ is minimal if
      \[  r^2f''(r) + 8rf'(r) + 12 f(r) + Q(f, rf', r^2 f'') = 0, \]
      were $Q$ collects the quadratic and higher order terms. Because of the symmetry between the regions $u > v$ and $u < v$ on the two sides of the cone $C$, the mean curvature of the graph of $f$ is an odd function of $f$, and so $Q$ is at least cubic. The solutions of the linearized operator are $f = r^{-3}$ and $f=r^{-4}$. Since $C$ is a strictly minimizing cone (see Hardt-Simon~\cite{HS85}), we know that $H$ is asymptotically the graph of $rf(r)$ over $C$ with $f(r) = r^{-3} + O(r^{-4})$ as $r\to\infty$. Since the nonlinear terms are cubic, it follows that we have $f(r) = r^{-3} + br^{-4} + O(r^{-9})$ for a constant $b\in \mathbf{R}$, and it remains to determine $b$.

      Let us first see that we cannot have $b=0$.  For this we suppose that
      \[ f(r) = r^{-3} + O(r^{-9}), \]
      and compute the expansion of $w(\xi)$ for large $\xi$ under this assumption. The unit normal vector of $C$ at the point $(x',x'')$, pointing into the region where $u > v$, is
      \[ \mathbf{n} = r^{-1} (x', -x''), \]
      where $r = (|x'|^2 +|x''|^2)^{1/2}$, so $|x'| = |x''| = r / \sqrt{2}$. It follows that the graph of $rf(r)$ is given by the set of points of the form $(x', x'') + f(r) (x', -x'')$ for $(x', x'') \in C$. At these points we have
      \[ u &= e^{z(t)}\cos t = (1 + f(r)) |x'| = r(1 + f(r)) / \sqrt{2}, \\
        v &= e^{z(t)}\sin t = (1-f(r)) |x''| = r(1-f(r)) / \sqrt{2}. \]
      It follows that
      \[ z(t) &= \frac{1}{2} \ln ( r^2 (1 + f(r)^2) ) \\
        &= \ln r + \frac{1}{2}\ln (1 + f(r)^2) \\
        &= \ln r + O(r^{-6}), \]
      and $\tan t = \frac{1-f(r)}{1+f(r)}$. 
      We have $\xi = \tan (2t) = \frac{ 2\tan t}{1 - \tan^2 t}$, from which we get
      \[ \xi = \frac{1-f^2}{2f} = \frac{1}{2} r^3 + O(r^{-3}), \]
      and so
      \[ r &= 2^{1/3} \xi^{1/3} + O( \xi^{-5/3}), \\
         z &= \frac{1}{3}\ln \xi + \frac{1}{3} \ln 2 + O(\xi^{-2}). \]
      Finally we have
      \[ w &= \frac{dz}{dt} = \frac{dz}{d\xi} \frac{d\xi}{dt} = \frac{dz}{d\xi} 2(1 + \xi^2) \\
        &= (\frac{1}{3} \xi^{-1} + O(\xi^{-3}) )(2 + 2\xi^2) \\
        &= \frac{2}{3} \xi + O(\xi^{-1}). \]
      This contradicts our earlier result that for some $\epsilon > 0$ we have $w(\xi) > \frac{2}{3}\xi + \epsilon$ for all sufficiently large $\xi$.

      It follows that we have
      \[ f(r) = r^{-3} + br^{-4} + O(r^{-9}) \]
      for some $b\not =0$. To determine the sign of $b$ we can proceed as above, and compute the asymptotics of $w(\xi)$ under this assumption. As above we have $z(t) = \ln r + O(r^{-6})$ and
      \[ \xi = \frac{1}{2}r^3 - \frac{b}{2} r^2 + O(r^{-3}). \]
      It follows that
      \[ r = 2^{1/3} \xi^{1/3} + \frac{b}{3} + O(\xi^{-1/3}), \]
      and so
      \[ z = \frac{1}{3}\ln \xi + \frac{1}{3}\ln 2 + \frac{b}{3\cdot 2^{1/3}}\xi^{-1/3} + O(\xi^{-2/3}). \]
      As above this implies that
      \[ w = \frac{dz}{d\xi} (2+2\xi^2) = \frac{2}{3}\xi - \frac{2^{2/3}}{9} b \xi^{2/3} + O(\xi^{1/3}). \]
      Since we know that $w(\xi) > \frac{2}{3}\xi + \epsilon$ for sufficiently large $\xi$, we must have $b < 0$ as claimed. 
    \end{proof}

    \subsection{The approximate solutions}\label{subsec:s2}
    In this section we build approximately minimal smoothings $\tilde{\Sigma}_\delta$ of $\Sigma_0$. In the next section we show that these $\tilde{\Sigma}_\delta$ can be perturbed to the $\Sigma_\delta$ that we are trying to find.

    By symmetry we can assume that $\delta > 0$, and to simplify the notation we will let $\epsilon = \delta^{1/3}$. We construct $\tilde{\Sigma}_\delta$ by gluing the scaled down surfaces $\pm \epsilon H$ near the two singular points of $\Sigma_0$ to the graph of $\epsilon^3 \phi$. Using the symmetry that maps $y \mapsto -y$ and interchanges $x', x''$, it will be enough to focus on one of the singular points.

    The surface $\Sigma_0\subset \mathbf{S}^8$ is given by the set
    \[ \{ (x', x'', y)\,:\, |x'| = |x''|, \quad |x'|^2 + |x''|^2 + y^2 = 1\} \subset S^8 \subset \mathbf{R}^8\times \mathbf{R}. \]
    We will work near the singular point where $y=1$ and $x', x'' = 0$, and we use the unit normal that points into the region where $|x'| > |x''|$. Near this singular point we use the chart on $S^8$ given by 
    \[ \label{eq:Fchart} F : (z', z'') \mapsto (z', z'', \sqrt{1 - |z'|^2 - |z''|^2}), \]
    so that $F^*\Sigma_0$ is the cone given by $|z'| = |z''|$. At the same time the pullback of the spherical metric is
    \[ F^*g_{S^8} = g_{Euc} + O(r^2), \]
    for small $r$, where $r^2 = |z'|^2 + |z''|^2$. We let $r_\epsilon = \epsilon^\alpha$, where $\alpha < 1$ will be chosen very close to 1. We define the surface $\tilde{\Sigma}_\delta$ to be the graph of $\epsilon^3\phi$ over $\Sigma_0$ on the region where $r > 2r_\epsilon$, and to be the surface $\epsilon H$ in the region $r < r_\epsilon$ using our chart to identify a neighborhood of the singular point with a ball in $\mathbf{R}^8$. In the gluing region $r\in (r_\epsilon, 2r_\epsilon)$ we use a cutoff function to interpolate between the two pieces (recall that the construction will be symmetric across the equator $y=0$, so that near the other singular point we will glue in $-\epsilon H$).

    To understand how to interpolate between the two pieces on the gluing region, we need to view $\epsilon H$ as a graph over $\Sigma_0$ in the region $r\in (r_\epsilon, 2r_\epsilon)$. It will be convenient to scale this up by $r_\epsilon^{-1}$, so we are considering the surface $r_\epsilon^{-1}\epsilon H$ in the region $\tilde{r}\in (1,2)$ in terms of $\tilde{r} = r_\epsilon^{-1}r$. In $\mathbf{R}^8$ the surface $H$ is the graph of $r^{-2} + br^{-3} + O(r^{-9})$ over the cone $C$ outside of a compact set, and so $r_\epsilon^{-1}\epsilon H$ is the graph of
    \[ r_\epsilon^{-3}\epsilon^3 \tilde{r}^{-2} + b r_\epsilon^{-4} \epsilon^4 \tilde{r}^{-3} + O(r_\epsilon^{-9}\epsilon^9) \]
    over $C$. We are considering graphs in terms of the Euclidean normal vector to $C$, whereas we would like to use the normal vector with respect to the metric $r_\epsilon^{-2} F^*g_{S^8}$, which satisfies  $r_\epsilon^{-2} F^*g_{S^8}= g_{Euc} + O(r_\epsilon^2)$ on the region $\tilde{r} \in (1,2)$. It follows that with respect to the spherical normal vector $r_\epsilon^{-1}\epsilon H$ is the graph of
    \[ \label{eq:piece1} r_\epsilon^{-3}\epsilon^3 \tilde{r}^{-2} + b r_\epsilon^{-4}\epsilon^4 \tilde{r}^{-3} + O(r_\epsilon^{-1} \epsilon^3 + r_\epsilon^{-9}\epsilon^9). \]
    We are gluing this to the graph of $\epsilon^3 \phi$  (scaled up by a factor of $r_\epsilon^{-1}$), where
    \[ \label{eq:phi10}
      \phi &= \frac{y^3}{r^2} - y = \frac{(1-r^2)^{3/2}}{r^2} - (1-r^2)^{1/2} \\
      &= r^{-2} + O(1). \]
    Scaling up by $r_\epsilon^{-1}$ we find that the other piece of our surface is the graph of
    \[ \label{eq:piece2}  r_\epsilon^{-3}\epsilon^3 \tilde{r}^{-2} + O(r_\epsilon^{-1} \epsilon^3) \]
    over $\Sigma_0$ on the region $\tilde{r}\in (1,2)$. 
    Let $\chi$ denote a cutoff function such that $\chi(s)=1$ for $s\leq 1$ and $\chi(s)=0$ for $s \geq 2$. Using $\chi$ to interpolate between the expressions \eqref{eq:piece1} and \eqref{eq:piece2} we define $r_\epsilon^{-1}\tilde{\Sigma}_\delta$ on the region $\tilde{r}\in (1,2)$ to be the graph of a function of the form
    \[ \label{eq:gluedpiece}
      r_\epsilon^{-3}\epsilon^3 \tilde{r}^{-2} + \chi(\tilde{r}) b r_\epsilon^{-4} \epsilon^4 \tilde{r}^{-3} + O(r_\epsilon^{-1}\epsilon^3 + r_\epsilon^{-9} \epsilon^9)
    \]
    over $\Sigma_0$. 
    We next estimate the mean curvature of the surface $\tilde{\Sigma}_\delta$ in a suitable weighted space. For a function $u$ on $\tilde{\Sigma}_\delta$ we define the weighted norm $\Vert u\Vert_{C^k_\tau}$ to be the smallest constant $c$ such that
    \[ r^{-\tau} |u| + r^{1-\tau} |\nabla u| + \ldots + r^{k-\tau} |\nabla^k u| \leq c. \]
    See \eqref{eq:wH1} below for a definition of corresponding weighted H\"older spaces. In terms of the weighted spaces we have the following estimate.
    \begin{prop}\label{prop:mcest1}
      For a small $\kappa > 0$, if $\alpha$ is chosen sufficiently close to 1, the mean curvature $m(\tilde{\Sigma}_\delta)$ satisfies
      \[ \label{eq:mest2} \Vert m(\tilde{\Sigma}_\delta) \Vert_{C^1_{-4}} \leq \epsilon^{3+\kappa} \]
      for sufficiently small $\epsilon$. 
    \end{prop}
    \begin{proof}
      We will work in regions $r\in (R, 2R)$ for various $R$, scaling up by a factor of $R^{-1}$. The required estimate is equivalent to saying that the scaled up surface has mean curvature of order $R\cdot \epsilon^{3+\kappa} R^{-4}$. We examine three cases separately: 
      \begin{itemize}
      \item {\bf $R > 2r_\epsilon$}. On this region $\tilde{\Sigma}_\delta$ is the graph of $\epsilon^3 \phi$. Scaled up by $R^{-1}$, in terms of the rescaled variable $\tilde{r} = R^{-1} r$ it is the graph of $R^{-1} \epsilon^3\phi( R\tilde{r})$, which is of order $O(R^{-3}\epsilon^3)$ if $\tilde{r}\in (1,2)$.  Note that $\phi$ is a Jacobi field on $\Sigma_0$ and $R^{-1}\Sigma_0$ has bounded geometry in our region, so the mean curvature of the graph is at least quadratic, i.e. $O(R^{-6}\epsilon^6)$. We need
        \[ R^{-6} \epsilon^6 \leq R^{-5} \epsilon^{3+\kappa}, \]
        i.e. $R \geq \epsilon^{3-\kappa}$. Since $R > 2r_\epsilon = 2\epsilon^\alpha$ for $\alpha$ close to 1, this estimate holds. 
      \item {\bf $r_\epsilon < R < 2r_\epsilon$}.
        This is the gluing region, and after scaling by $r_\epsilon^{-1}$ (which is essentially equivalent to scaling by $R^{-1}$), the surface $r_\epsilon^{-1}\tilde{\Sigma}_\delta$ is the graph of
        \[ r_\epsilon^{-3}\epsilon^3 \tilde{r}^{-2} + O( r_\epsilon^{-1} \epsilon^3 + r_\epsilon^{-4} \epsilon^4) \]
        according to \eqref{eq:gluedpiece}, where we absorbed the $\tilde{r}^{-3}$ term into the error. The function $\tilde{r}^{-2}$ is a Jacobi field with respect to the Euclidean metric, which after our scaling differs from the spherical metric by $O(r_\epsilon^2)$. It follows that the contribution of the first term to the mean curvature is $O(r_\epsilon^{-1}\epsilon^3 + r_\epsilon^{-6}\epsilon^6)$. The contribution of the remaining terms is $O( r_\epsilon^{-1} \epsilon^3 + r_\epsilon^{-4} \epsilon^4)$, and so in sum the mean curvature is $O(r_\epsilon^{-1} \epsilon^3 + r_\epsilon^{-4} \epsilon^4)$. For the estimate \eqref{eq:mest2} we need
        \[ r_\epsilon^{-1} \epsilon^3 + r_\epsilon^{-4} \epsilon^4 \leq r_\epsilon^{-3} \epsilon^{3+\kappa}. \]
        This follows if $r_\epsilon^2 \leq \epsilon^\kappa$ and $\epsilon^{1-\kappa} \leq r_\epsilon$. These hold for small $\kappa > 0$ when $\alpha$ is sufficiently close to 1. 
      \item {\bf $R < r_\epsilon/2$}. Here the surface is $\epsilon H$. Scaled up by $R^{-1}$ we are considering $R^{-1}\epsilon H$ in the region $\tilde{r} \in (1,2)$ in terms of $\tilde{r} = R^{-1} r$. These surfaces have uniformly bounded geometry, and zero mean curvature with respect to the Euclidean metric. The spherical metric (scaled up) differs by $O(R^2)$ from the Euclidean metric, so with respect to $R^{-2}F^* g_{S^8}$ the mean curvature of $R^{-1}\epsilon H$ is $O(R^2)$. To satisfy \eqref{eq:mest2} we need $R^2 \leq R^{-3}\epsilon^{3+\kappa}$, i.e. $r_\epsilon^5 \leq \epsilon^{3+\kappa}$. This follows if $\kappa < 2$ and $\alpha$ is sufficiently close to 1. 
      \end{itemize}
     \end{proof}

     \subsection{Inverting the linearized operator}\label{subsec:s3}
     We will construct the desired surface $\Sigma_\delta$ as the graph of a function $u$ over $\tilde{\Sigma}_\delta$, for $u$ in a suitable weighted space. We can define the weighted $C^{k,\alpha}_\tau$ norms for functions on $\tilde{\Sigma}_\delta$ as follows. Let us denote by $g$ the induced metric on $\tilde{\Sigma}_\delta$, and for any $R > 0$ let $A_R$ denote the region where $r\in (R, 2R)$. We define the weighted norm of $u$ as follows:
     \[ \label{eq:wH1} \Vert u\Vert_{C^{k,\alpha}_\tau} = \sup_R R^{-\tau}\Vert u\Vert_{C^{k,\alpha}(\tilde{\Sigma}_\delta\cap A_R, R^{-2}g)}. \]
     We take the supremum over all $R$ such that the level set $r = 3R/2$ intersects $\tilde{\Sigma}_\delta$. Note that for such $R$ the surface $\tilde{\Sigma}_\delta \cap A_R$, equipped with the scaled up metric $R^{-2} g$ has uniformly bounded geometry and this definition of the norm is uniformly equivalent to the one we used in Proposition~\ref{prop:mcest1}.

 We can consider the graph of $u$ over $\tilde{\Sigma}_\delta$ if $\Vert u\Vert_{C^{2,\alpha}_1}$ is sufficiently small (independently of $\delta$), but we will also need to work with weights other than $1$. Since $\tilde{\Sigma}_\delta$ is given by $\pm \epsilon H$ near the singular points, we have $r > c_1 \epsilon$ on $\tilde{\Sigma}_\delta$ for a uniform $c_1 > 0$. This allows us to compare the norms for different weights: for $\tau < \tau'$ we have
 \[\label{eq:normcomp1}
   \Vert u\Vert_{C^{k,\alpha}_\tau} \lesssim \Vert u\Vert_{C^{k,\alpha}_{\tau'}} \lesssim \epsilon^{\tau - \tau'} \Vert u\Vert_{C^{k,\alpha}_\tau}. \]

     The Jacobi operator on $\tilde{\Sigma}_\delta$ defines a bounded linear operator
     \[ L_\delta : C^{2,\alpha}_\tau \to C^{0,\alpha}_{\tau-2}, \]
     with bound independent of $\delta$. We only work with functions invariant under $G = O(4)\times O(4)$, and write $C^{k,\alpha, G}$ for the corresponding spaces. In addition for technical reasons we fix a small $r_0 > 0$, and denote by $C^{k,\alpha, G, 0}$ the $G$-invariant functions with also vanish along the set $\{ r = r_0, y > 0\}$. The main result in this section is the following.
     \begin{prop}\label{prop:invertLdelta}
       Let $\tilde{L}_\delta$ denote the linear operator
       \[ \tilde{L}_\delta : C^{2,\alpha, G, 0}_\tau \times \mathbf{R} \to C^{0,\alpha, G}_{\tau-2}, \]
       defined by
       \[ \label{eq:Ltildedef} \tilde{L}_\delta(u, \lambda) = L_\delta(u) + \lambda \zeta. \]
       Then for $\tau\in (-3,-2)$ the operator $\tilde{L}_\delta$ is invertible for sufficiently small $\delta$, with inverse bounded independently of $\delta$. We view $\zeta$ as a function on $\tilde{\Sigma}_\delta$ using that $\tilde{\Sigma}_\delta$ is a graph over $\Sigma_0$ on the support of $\zeta$. 
     \end{prop}
     \begin{proof}
       The proof is by a standard argument by contradiction. Our goal is to prove that we have the estimate
       \[ \Vert u\Vert_{C^{2,\alpha}_\tau} + |\lambda| \leq C \Vert L_\delta u + \lambda\zeta \Vert_{C^{0,\alpha}_{\tau-2}} \]
       for all $u\in C^{2,\alpha, G, 0}$, with $C$ independent of $\delta$, for sufficiently small $\delta$. Suppose that this estimate does not hold, so that for a sequence $\delta_i\to 0$ we have corresponding functions $u_i$ and $\lambda_i\in \mathbf{R}$ such that
       \[ \label{eq:ui1} \Vert u_i\Vert_{C^{2,\alpha}_\tau} + |\lambda_i| = 1, \]
       but
       \[ \label{eq:ui2} \Vert L_\delta u_i + \lambda_i\zeta \Vert_{C^{0,\alpha}_{\tau-2}} \to 0. \]
       The uniform $C^{2,\alpha}_\tau$ bounds imply that up to choosing a subsequence we can extract a limit $u_i \to u$ in $C^{2,\alpha/2}_{loc}$, where $u$ is defined on $\Sigma_0$. We can also assume $\lambda_i\to \lambda$, and we find that $L_{\Sigma_0} u + \lambda\zeta = 0$ on $\Sigma_0$. From Lemma~\ref{lem:zetaJacobi}, using that $u$ is $G$-invariant, and that $u$ vanishes on the set $\{r=r_0, y > 0\}$, we must have $u = 0, \lambda=0$. From \eqref{eq:ui2} it then follows that
       \[ \label{eq:ui3}  \Vert L_\delta u_i \Vert_{C^{0,\alpha}_{\tau-2}} \to 0. \]

       By the definition of the weighted spaces, using that the surfaces $A_R\cap \tilde{\Sigma}_\delta$ scaled up by $R^{-1}$ have uniformly bounded geometry, we have the weighted Schauder estimates
       \[ \Vert u_i\Vert_{C^{2,\alpha}_\tau} \leq C( \Vert L_\delta u_i\Vert_{C^{0,\alpha}_{\tau-2}} + \Vert u_i\Vert_{C^0_\tau}), \]
       so in order to contradict \eqref{eq:ui1} it is enough to show $\Vert u_i\Vert_{C^0_\tau}\to 0$. Note that $\Vert u_i \Vert_{C^0_\tau}$ is uniformly equivalent to $\sup r^{-\tau} |u_i|$. Arguing by contradiction suppose that we have $\sup r^{-\tau} |u_i| > c_2 > 0$ for a constant $c_2$ independent of $i$, and let $x_i\in \tilde{\Sigma}_{\delta_i}$ be the points where this supremum is realized. We already know that $r(x_i)\to 0$, since the $u_i$ converge to zero on $\Sigma_0$.

       There are two possibilities: either $r(x_i) \delta_i^{-1/3} \to \infty$, or $r(x_i)\delta_i^{-1/3}$ remains bounded. In the first case we can extract a non-zero limit $u$ on the cone $C$ satisfying $L_C u =0$, and $|u| \leq r^\tau$. Since $\tau\in (-3,-2)$ and there are no non-zero Jacobi fields on $C$ with growth rate in this range, this is a contradiction. In the second case we can extract a non-zero limit $u$ on the surface $H$, with $L_H u = 0$ and $|u| \leq r^\tau$. Since $H$ is strictly stable, the Jacobi field $v$ on $H$ generated by homothetic scalings is asymptotically $r^{-2}$. Replacing $u$ with $-u$ if necessary, we can assume that $u$ is positive somewhere. Since $v$ is positive and $u$ decays faster than $v$ at infinity, we can find a constant $c > 0$ such that $cu \leq v$, but $cu(q) = v(q)$ for some $q\in H$. By the strong maximum principle $cu=v$ which contradicts that $u$ decays faster than $v$.
     \end{proof}

     \subsection{The nonlinear problem}\label{subsec:s4}
     In this section we construct the surface $\Sigma_\delta$ as a graph of $u$ over $\tilde{\Sigma}_\delta$, for $u$ sufficiently small in $C^{2,\alpha}_\tau$. The main result it the following.
     \begin{prop}\label{prop:Sigma1}
       There is a small $\kappa > 0$ such that if $\tau\in (-3,-2)$ is chosen sufficiently close to $-2$ then we have the following. 
       For sufficiently small $\delta = \epsilon^3$ there is a unique $G$-invariant function $u$ on $\tilde{\Sigma}_\delta$ such that $u=0$ on the set $\{r=r_0, y > 0\}$ for a fixed small $r_0 > 0$, 
       \[ \Vert u \Vert_{C^{2,\alpha}_\tau} \leq \epsilon^{3+\kappa}, \]
       and the surface $\Sigma_\delta$ given by the graph of $u$ over $\tilde{\Sigma}_\delta$ has mean curvature $m(\Sigma_\delta) = h(\delta)\zeta$. We view $\zeta$ as a function on $\Sigma_\delta$ using that $\Sigma_\delta$ is a graph over $\tilde{\Sigma}_\delta$. 
     \end{prop}
     Note that the only purpose of introducing the condition that $u$ vanishes on $\{r=r_0, y > 0\}$ is to make the solution unique, and $r_0$ just needs to satisfy that $\phi(r_0)\not=0$. This leads to the linearized operator in Proposition~\ref{prop:invertLdelta} being invertible, since by Lemma~\ref{lem:zetaJacobi} the only $G$-invariant elements in the kernel of the linearized operator on $\Sigma_0$ are multiples of $\phi$.
     
     For a function $u\in C^{2,\alpha}_\tau$ on $\tilde{\Sigma}_\delta$ let us denote by $m(u)$ the mean curvature of the graph of $u$. We have
     \[ m(u) = m(\tilde{\Sigma}_\delta) + L_\delta(u) + Q_\delta(u), \]
     where $Q_\delta$ collects the higher order terms in the mean curvature operator. We can analyze the behavior of this operator in weighted spaces using that the surface $A_R \cap \tilde{\Sigma}_\delta$ scaled up by $R^{-1}$ has uniformly bounded geometry for all $R, \delta$ (recall that $A_R$ is the region $r\in (R, 2R)$). The following can be shown by rescaling, and using that $Q$ only has quadratic and higher order terms. 
     \begin{prop}\label{prop:Qprop1}
       Let $\tau\in \mathbf{R}$. There is a small $c_3 > 0$ such that if $\Vert u_i\Vert_{C^{2,\alpha}_1} \leq c_3$ for $i=1,2$, then we have
       \[ \Vert Q_\delta(u_1) - Q_\delta(u_2) \Vert_{C^{0,\alpha}_{\tau-2}} \leq C (\Vert u_1\Vert_{C^{2,\alpha}_1} + \Vert u_2\Vert_{C^{2,\alpha}_1}) \Vert u_1-u_2\Vert_{C^{2,\alpha}_\tau}, \]
       for $C$ independent of $\delta$. 
     \end{prop}
     
     Our goal is to find $(u, \lambda)$ for sufficiently small $\delta$ satisfying
     $m(u) + \lambda \zeta = 0$, i.e.
     \[ L_\delta (u) + \lambda \zeta = -m(\tilde{\Sigma}_\delta) - Q_\delta(u). \]
     We work with $G$-invariant functions $u$, so we can use the inverse given by Proposition~\ref{prop:invertLdelta} to write the equation in the equivalent form
     \[ (u, \lambda) = \tilde{L}_\delta^{-1} (-m(\tilde{\Sigma}_\delta) - Q_\delta(u)). \]
     In other words we are trying to find a fixed point of the operator
     \[ \mathcal{N} : C^{2,\alpha, G, 0}_\tau \times \mathbf{R} \to
       C^{2,\alpha, G, 0}_\tau \times \mathbf{R} \]
     defined by 
     \[ \mathcal{N}(u, \lambda) = -\tilde{L}_\delta^{-1}(m(\tilde{\Sigma}_\delta) + Q_\delta(u)). \]
     The existence and uniqueness follows from the following.
     \begin{prop}
       Define the set
       \[ E = \{ (u, \lambda) \,:\, \Vert u\Vert_{C^{2,\alpha}_\tau} + |\lambda| \leq \epsilon^{3+\kappa'} \} \subset C^{2,\alpha, G, 0}_\tau \times \mathbf{R}. \]
       There is a $\kappa' > 0$ such that for $\tau\in (-3,-2)$ sufficiently close to $-2$, and $\delta = \epsilon^3$ sufficiently small, the map $\mathcal{N}$ is a contraction on $E$, and so has a fixed point. 
     \end{prop}
     \begin{proof}
       Using the uniform bound for $\tilde{L}_\delta^{-1}$, the estimate \eqref{eq:mest2} and \eqref{eq:normcomp1}, we have
       \[ \Vert \mathcal{N}(0,0)\Vert &\leq C \Vert m(\tilde{\Sigma}_\delta)\Vert_{C^{0,\alpha}_{\tau-2}} \\
         &\leq C \Vert m(\tilde{\Sigma}_\delta)\Vert_{C^{0,\alpha}_{-4}} \\
         &\leq C \epsilon^{3 + \kappa}. \]

       If $(u,\lambda) \in E$, then we have
       \[ \Vert u\Vert_{C^{2,\alpha}_1} \leq C\epsilon^{\tau-1} \Vert u\Vert_{C^{2,\alpha}_\tau} \leq C\epsilon^{\tau+2+\kappa'}. \]
       Given $\kappa' > 0$, as long as $\tau$ is sufficiently close to $-2$, we have that $\Vert u\Vert_{C^{2,\alpha}_1} \leq c_3$ for the $c_3$ in Proposition~\ref{prop:Qprop1}, for sufficiently small $\epsilon$. It follows that
       \[ \Vert Q_\delta(u)\Vert_{C^{0,\alpha}_{\tau-2}} \leq C\epsilon^{3+\kappa'} \epsilon^{2+\tau + \kappa'}, \]
       and so using the bound on $\tilde{L}_\delta^{-1}$ we have
      \[ \Vert \mathcal{N}(u, \lambda)\Vert \leq C(\epsilon^{3+\kappa} + \epsilon^{3+\kappa'+2 + \tau + \kappa'}) \leq \epsilon^{3 + \kappa'} \]
      for sufficiently small $\epsilon$, as long as we choose $\kappa' < \kappa$, and $\tau$ is sufficiently close to $-2$.  It follows that $\mathcal{N}$ maps $E$ to itself.

      To see that $\mathcal{N}$ is a contraction, we have that for $(u_i, \lambda_i)\in E$
      \[ \Vert \mathcal{N}(u_1, \lambda_1) - \mathcal{N}(u_2,\lambda_2) \Vert &\leq C \Vert Q_\delta(u_1) - Q_\delta(u_2)\Vert_{C^{0,\alpha}_{\tau-2}} \\
        &\leq C \epsilon^{\tau+2+\kappa'} \Vert u_1 - u_2\Vert_{C^{2,\alpha}_\tau} \\
        &\leq C \epsilon^{\tau + 2+\kappa'} \Vert (u_1-u_2, \lambda_1-\lambda_2)\Vert \\
        &\leq \frac{1}{2} \Vert (u_1-u_2, \lambda_1-\lambda_2)\Vert \]
      for sufficiently small $\epsilon$, if $\kappa' > 0$ and $\tau$ is sufficiently close to $-2$. 
  \end{proof}

  \begin{proof}[Proof of parts (ii) and (iii) of Proposition~\ref{prop:Sdelta}]
    The previous Proposition says that for sufficiently small $\delta$ the graph of $u$ over $\tilde{\Sigma}_\delta$ satisfies $m(u) =h(\delta) \zeta$, for a suitable function $h(\delta)$. In addition there is a $\kappa' > 0$ such that $u$ satisfies the estimate $\Vert u\Vert_{C^{2,\alpha}_\tau} \leq \epsilon^{3+\kappa'}$, given $\tau\in (-3,-2)$ and $\epsilon = \delta^{1/3}$ sufficiently small. If we choose $\tau$ sufficiently close to $-2$, then this implies $|u| \leq C|\delta| r^{-2+\kappa}$ for some $\kappa > 0$. In particular it is enough to show that the claim (ii) holds for the cone $\tilde{V}_\delta$ over the surface $\tilde{\Sigma}_\delta$. 

    Recall that to construct $\tilde{\Sigma}_\delta$, we used the chart $F$ in \eqref{eq:Fchart} to identify a neighborhood of the singular point (the one with $y=1$) in $\Sigma_0$ with a neighborhood of the origin in $\mathbf{R}^8$. Under this identification on the set $r < r_\epsilon$ the surface $\tilde{\Sigma}_\delta$ is $\epsilon H$, while on $\{r > r_\epsilon\}$ it is the graph of a function $f_1$ over $C$ satisfying $f_1 = \epsilon^3 r^{-2} + O(\epsilon^3)$ on $r > 2r_\epsilon$ and $f_1 = \epsilon^3 r^{-2} + O(\epsilon^4 r^{-3})$ on $r\in (r_\epsilon, 2r_\epsilon)$. Since $\epsilon H$ itself is the graph of $\epsilon^3r^{-2} + O(\epsilon^4 r^{-3})$ over $C$ on the set $r > r_\epsilon$, it follows that $\tilde{\Sigma}_\delta$ is the graph of $f_2$ over $\epsilon H$ with $|f_2| < C(\epsilon^3 + \epsilon^4 r^{-3})$ on the set $r > r_\epsilon$. Since $r_\epsilon = \epsilon^\alpha$ with $\alpha < 1$, it follows that $|f_2| < C\epsilon^3 r^{-2+\kappa}$ for some $\kappa > 0$.

    Consider the cone $\tilde{V}_\delta$ over $\tilde{\Sigma}_\delta$, where we view $\tilde{\Sigma}_\delta$ as a subset of the unit sphere in $\mathbf{R}^8\times \mathbf{R}$. Then $\tilde{V}_\delta \cap\{y=1\}$ consists of the set of points $\lambda(p) p$ where $p$ runs over $\tilde{\Sigma}_\delta$, and the scaling factors $\lambda(p)$ are given by
    \[ \lambda(p) = \frac{1}{\sqrt{1 - r(p)^2}} = 1 + \frac{1}{2} r(p)^2 + O( r(p)^4). \]
    In particular we have $r(\lambda(p) p ) \sim r(p)$, as long as $r(p)$ is small. 

    We saw above that $p\in \tilde{\Sigma}_\delta$ sits in the graph of a function $f_2$ over $\epsilon H$ in our chart, with $|f_2| < C\epsilon^3 r(p)^{-2+\kappa}$. If $q\in \epsilon H$ in our chart, then $\lambda(q) q \in \lambda(q) \epsilon H$, and from Lemma~\ref{lem:Phidefn1} below we know that $\lambda(q)\epsilon H$ is the graph of an $O(\epsilon^3 (\lambda(q)-1) r^{-2})$ function over $\epsilon H$. Since $\lambda(q) - 1 \sim r(q)^2$, it follows that for $p\in \tilde{\Sigma}_\delta$, the corresponding point $\lambda(p) p\in \tilde{V}_\delta$ is in the graph of a function $f_3$ over $\epsilon H$ with $|f_3| < C(\epsilon^3 + \epsilon^3 r(p)^{-2 + \kappa})$, i.e. $|f_3| < C |\delta| r^{-2+\kappa}$ as claimed in property (ii). 

From the construction of $\tilde{\Sigma}_\delta$, and the proof of Proposition~\ref{prop:mcest1}, it is clear that the regularity scale of $\tilde{\Sigma}_\delta$ at each point is comparable to $r$. From Proposition~\ref{prop:Sigma1} it follows that the same holds for $\Sigma_\delta$ once $\delta$ is sufficiently small. This shows claim (iii) in Proposition~\ref{prop:Sdelta}. 
  \end{proof}

     \subsection{Computation of $h(\delta)$}\label{subsec:s5}
     So far we have constructed a surface $\Sigma_\delta$ satisfying $m(\Sigma_\delta) = h(\delta) \zeta$. In this section we show that $h(\delta) = c_4b \delta^{4/3} + O(\delta^{4/3+\kappa''})$ for some $\kappa'' > 0$, where $c_4 > 0$ and $b <0$ is the coefficient of $r^{-3}$ in the asymptotics of $H$ as a graph over $C$ from Proposition~\ref{prop:refinedf}. The first step is to compute the integral of $m(\tilde{\Sigma}_\delta)\phi$ over $\tilde{\Sigma}_\delta$. Here we are considering $\phi$ as the function $y^3r^{-2} - y$ on $S^8$ restricted to $\tilde{\Sigma}_\delta$. We have the following.

     \begin{prop}\label{prop:inttildeest}
       We have the estimate
       \[ \label{eq:inttildeest} \int_{\tilde{\Sigma}_\delta} m(\tilde{\Sigma}_\delta) \phi\, dA = c_5 b \delta^{4/3} + O(\delta^{4/3 + \kappa''}) \]
         for some $c_5, \kappa'' > 0$. 
       \end{prop}
       \begin{proof}
         We first show that the leading order contribution in the integral comes from the gluing region where $r\in (r_\epsilon, 2r_\epsilon)$. To see this let us consider a region of the form $r\in (R, 2R)$, where $R \geq 2r_\epsilon$. As in the proof of Proposition~\ref{prop:mcest1} the mean curvature on this region is of order $R^{-7}\epsilon^6$ (we have to scale down the result in that proof by a factor of $R$). At the same time on this region $\phi = O(R^{-2})$, while the volume is $O(R^7)$. The contribution of this region to the integral is therefore $O(R^{-2} \epsilon^6)$. We can sum this up for $R= 2r_\epsilon, 4r_\epsilon, \ldots$, and find that the total contribution is $O(r_\epsilon^{-2} \epsilon^6)$. Since $r_\epsilon = \epsilon^\alpha$ with $\alpha$ close to 1, we have $r_\epsilon^{-2}\epsilon^6 \ll \epsilon^4 = \delta^{4/3}$ as $\epsilon \to 0$.

         Similarly, on the region $r\in (R, 2R)$ for $R < r_\epsilon / 2$ the mean curvature is $O(R)$, so the contribution to the integral is $O(R^6)$. We sum this up for $R= r_\epsilon/2, r_\epsilon/4, \ldots$ to get a total contribution of $O(r_\epsilon^6)$, and we have $r_\epsilon^6 \ll \epsilon^4$.

         It remains to study the integral on the region where $r\in (r_\epsilon, 2r_\epsilon)$. Here we need a more careful analysis of the mean curvature than what we used in the proof of Proposition~\ref{prop:mcest1}. Recall that, according to Equation~\eqref{eq:gluedpiece}, after scaling up by $r_\epsilon^{-1}$,  $\tilde{\Sigma}_\delta$ is the graph of
         \[ v = r_\epsilon^{-3}\epsilon^3 \tilde{r}^{-2} + \chi(\tilde{r}) b r_\epsilon^{-4} \epsilon^4 \tilde{r}^{-3} + O(r_\epsilon^{-1}\epsilon^3 + r_\epsilon^{-9}\epsilon^9) \]
         over the cone $C$ on our region, in terms of $\tilde{r} = r_\epsilon^{-1}r$. The cutoff function $\chi$ satisfies $\chi(s) = 1$ for $s \leq 1$ and $\chi(s)=0$ for $s \geq 2$. If we write $g_\epsilon$ for the metric induced on $C$ by the scaled spherical metric $r_\epsilon^{-2} F^*g_{S^8}$, then we have $g_\epsilon = g_0 + O(r_\epsilon^2)$, where $g_0$ is the Euclidean metric restricted to $C$. The mean curvature of the graph of $v$ over $C$ is
         \[ \label{eq:mest10} m &= L_{g_\epsilon}(v) + O(r_\epsilon^{-6} \epsilon^6) \\
           &= L_{g_0}(v) + O( r_\epsilon^{-1}\epsilon^3+  r_\epsilon^{-6} \epsilon^6) \\
           &= b r_\epsilon^{-4}\epsilon^4 L_{g_0} (\chi(\tilde{r}) \tilde{r}^{-3}) +
           O( r_\epsilon^{-1}\epsilon^3+  r_\epsilon^{-6} \epsilon^6), \]
         since $L_{g_0} \tilde{r}^{-2}=0$. Note that for a function of $\tilde{r}$ we have
         \[ L_{g_0}(w) = w'' + \frac{6}{\tilde{r}} w' + \frac{6}{\tilde{r}^2}, \]
         and we also have $L_{g_0}(\tilde{r}^{-3}) = 0$. It follows that
         \[ L_{g_0} (\chi(\tilde{r}) \tilde{r}^{-3}) 
           &= \chi''(\tilde{r}) \tilde{r}^{-3}. \]
         The main contribution to the integral in \eqref{eq:inttildeest} comes from the integral of $\chi''(\tilde{r})\tilde{r}^{-3}\cdot \tilde{r}^{-2}$ over the annulus $\tilde{r}\in (1,2)$ in $C$, using the Euclidean metric:
         \[ \label{eq:intC1} \int_{C \cap \{ 1 < \tilde{r} < 2 \}} \chi''(\tilde{r}) \tilde{r}^{-5} \, dA &= c_5  \int_1^2 \chi''(r) r^{-5}\, r^6\, dr \\
           &= - c_5 \int_1^2 \chi'(r)\,dr = -c_5 (\chi(2)-\chi(1)) = c_5, \]
         where $c_5 > 0$ is the volume of the link of $C$. After scaling, this gives the leading term in \eqref{eq:inttildeest}.

         It remains to account for all the errors relating \eqref{eq:inttildeest} to \eqref{eq:intC1}. Note first that in \eqref{eq:inttildeest}, on the region $r\in (r_\epsilon, 2r_\epsilon)$ we have $m(\tilde{\Sigma}_\delta) = O(r_\epsilon^{-2}\epsilon^3 + r_\epsilon^{-5}\epsilon^4)$ from the proof of Proposition~\ref{prop:mcest1} (we must scale the result there down by $r_\epsilon$); the volume of the region is $O(r_\epsilon^7)$; and $\phi = O(r_\epsilon^{-2})$.  We can now consider the errors from different sources.
         \begin{itemize}
         \item The error in $m$ given in \eqref{eq:mest10}. After scaling down this is $O(r_\epsilon^{-2}\epsilon^3 + r_\epsilon^{-7}\epsilon^6)$. After integrating, this leads to an error of order
           \[ (r_\epsilon^{-2}\epsilon^3 + r_\epsilon^{-7}\epsilon^6) r_\epsilon^{-2} r_\epsilon^7 = r_\epsilon^3\epsilon^3 + r_\epsilon^{-2}\epsilon^6.  \]
         \item The error in replacing $\phi$ by $r_\epsilon^{-2}$. Note that by \eqref{eq:phi10} we have $\phi = r^{-2} + O(1)$, and in addition $\nabla\phi = O(r^{-3})$. Also $\tilde{\Sigma}_\delta$ is the graph of an $O(r_\epsilon^{-2}\epsilon^3)$ function over $C$ if $r\sim r_\epsilon$. It follows that the error in \eqref{eq:mest10} resulting from replacing $\phi$ with $r^{-2}$ is of order
           \[ (1 + r_\epsilon^{-3} r_\epsilon^{-2}\epsilon^3) (r_\epsilon^{-2}\epsilon^3 + r_\epsilon^{-5}\epsilon^4) r_\epsilon^7 = O(\epsilon^6 + r_\epsilon^{-3}\epsilon^7). \]
         \item The error in using the area form of $C$ instead of that of $\tilde{\Sigma}_\delta$. We again use that $\tilde{\Sigma}_\delta$ is the graph of an $O(r_\epsilon^{-2}\epsilon^3)$ function over $C$. Scaling up by a factor of $r_\epsilon^{-1}$ so that the surfaces have bounded geometry we see that the area forms are related by $\frac{dA_{\tilde{\Sigma}_\delta}}{dA_C} = 1 + O(r_\epsilon^{-3}\epsilon^3)$ on our annulus (in fact since $C$ is minimal we have a better estimate but we do not need it). Thus, using the area form of $C$ in the integral \eqref{eq:mest10} leads to an error of order
           \[ (r_\epsilon^{-2}\epsilon^3 + r_\epsilon^{-5}\epsilon^4)r_\epsilon^{-2} (r_\epsilon^{-3}\epsilon^3) r_\epsilon^7 = \epsilon^6 + r_\epsilon^{-3}\epsilon^7. \]
           \end{itemize}
       If $r_\epsilon = \epsilon^\alpha$ with $\alpha < 1$ sufficiently close to 1, then all of the errors are of lower order than $\epsilon^4$, i.e. they are $O(\delta^{4/3 + \kappa''})$ for some $\kappa'' > 0$ as required. 
       \end{proof}

       Next we compute the integral of $m(\Sigma_\delta)\phi$. We perform this integral on $\tilde{\Sigma}_\delta$, viewing  $m(\Sigma_\delta)$ as a function on $\tilde{\Sigma}_\delta$, using that $\Sigma_\delta$ is a graph over $\tilde{\Sigma}_\delta$. 
       \begin{prop}\label{prop:intest2}
         We have
         \[ \int_{\tilde{\Sigma}_\delta} m(\Sigma_\delta) 
           \phi\, dA = c_5 b \delta^{4/3} + O(\delta^{4/3+\kappa''}),\]
         for the same $c_5, b$ as in Proposition~\ref{prop:inttildeest}. 
       \end{prop}
       \begin{proof}
         We know that $\Sigma_\delta$ is the graph of a function $u$ over $\tilde{\Sigma}_\delta$, where $\Vert u\Vert_{C^{2,\alpha}_\tau} \leq \epsilon^{3+\kappa}$ for sufficiently small $\epsilon$. Here $\kappa > 0$, and we can take $\tau < -2$ as close to $-2$ as we like (if $\tau$ is chosen closer to $-2$, $\epsilon$ will need to also be smaller for the estimate to hold). We have
         \[ m(\Sigma_\delta) = m(\tilde{\Sigma}_\delta) + L_\delta (u) + Q_\delta(u), \]
         and so in view of \eqref{eq:inttildeest} our goal is to show that
         \[ \int_{\tilde{\Sigma}_\delta} (L_\delta u + Q_\delta (u)) \phi\, dA = O(\epsilon^{4+\kappa''}) \]
         for some $\kappa'' > 0$. The main point is that on $\Sigma_0$ the image of $L_0$ is $L^2$-orthogonal to $\phi$ since $L_0\phi =0$, and we can show that all the remaining error terms are $O(\epsilon^{4+\kappa''})$.

         Let us first consider the integral of $Q_\delta(u) \phi$. In the annular region $A_R$ where $r\in (R, 2R)$ we have
         \[ \Vert u\Vert_{C^{2,\alpha}_1} \leq C R^{\tau-1} \Vert u\Vert_{C^{2,\alpha}_\tau} \leq C R^{\tau-1}\epsilon^{3+\kappa}. \]
         Applying Proposition~\ref{prop:Qprop1} in this annular region we have
         \[ \Vert Q_\delta(u)\Vert_{C^{2,\alpha}_{-1}} \leq C R^{2\tau-2}\epsilon^{6+2\kappa}, \]
         i.e. on all of $\tilde{\Sigma}_\delta$ we have the estimate
         \[ |Q_\delta(u)| \leq C r^{2\tau-2} \epsilon^{6+2\kappa} r^{-1}. \]
         Using also that $\phi = O(r^{-2})$ it follows that
         \[ \left| \int_{\tilde{\Sigma}_\delta} Q_\delta(u) \phi\, dA\right| \leq C \int_\epsilon^1
         r^{2\tau-3}\epsilon^{6+2\kappa} r^{-2} r^6\,dr \leq C \epsilon^{2\tau + 8 + 2\kappa}.\]
       If $\tau$ is sufficiently close to $-2$, then this is $O(\epsilon^{4+\kappa})$.

       To deal with the $L_\delta(u)$ term let us write $u=u_1+u_2$ where $u_1 = \chi(r_\epsilon^{-1} r) u$ for our cutoff function $\chi$. So $u_1$ is supported on the region $r < 2r_\epsilon$, while $u_2$ is supported on $r > r_\epsilon$. Note that $\Vert u_i\Vert_{C^{2,\alpha}_\tau} \leq C\epsilon^{3+\kappa}$.

       We have $|L_\delta u_1| \leq C \epsilon^{3+\kappa}r^{\tau-2}$, so we can estimate
       \[ \left| \int_{\tilde{\Sigma}_\delta} (L_\delta u_1)\phi\,dA\right| \leq C\int_\epsilon^{r_\epsilon} \epsilon^{3+\kappa}r^{\tau-2}r^{-2}r^6\,dr \leq C \epsilon^{3+\kappa} r_\epsilon^{\tau + 3}. \]
       If $\tau$ is sufficiently close to $-2$ and $r_\epsilon = \epsilon^\alpha$ for $\alpha$ sufficiently close to 1, then this error is $O(\epsilon^{4+\kappa''})$ for some $\kappa'' > 0$.

       It remains to consider the $L_\delta u_2$ term. The function $u_2$ is supported on the region where $r > r_\epsilon$, and here $\tilde{\Sigma}_\delta$ is the graph of a function $f$ over $\Sigma_0$ satisfying the estimates $|f|\lesssim \epsilon^3 r^{-2}$ and $|\nabla f| \lesssim \epsilon^3 r^{-3}$. We use this to identify $\Sigma_0$ with $\tilde{\Sigma}_\delta$ on this region. Let us write $dA_0, dA_\delta$ for the corresponding volume forms on $\Sigma_0$, and $\underline\phi_\delta$ for the function $\phi$ pulled back to $\Sigma_0$ under this identification. Note that
       \[ \label{eq:intL10} \int_{\Sigma_0} L_0(u_2) \phi\, dA_0 = 0, \]
       since $L_0\phi = 0$. We need to compare this integral to
       \[ \label{eq:intL11} \int_{\Sigma_0}  L_\delta(u_2) \underline\phi_\delta\, dA_\delta.  \]
       Note that we have $|L_0 u_2| \leq C \epsilon^{3+\kappa} r^{\tau-2}$, $|\phi| \leq Cr^{-2}$ and $dA_0 \leq Cr^6\,dr$ for $O(4)\times O(4)$ invariant functions. We have the following estimates for the difference between \eqref{eq:intL10} and \eqref{eq:intL11}:
       \begin{itemize}
       \item Comparing $L_0$ to $L_\delta$ we have $|(L_0 - L_\delta) u_2| \leq C \epsilon^3r^{-3} \epsilon^{3+\kappa} r^{\tau-2}$. This leads to a difference of order
         \[ \int_{r_\epsilon}^1 \epsilon^3r^{-3} \epsilon^{3+\kappa} r^{\tau-2} r^{-2} r^6\,dr =O( \epsilon^{6+\kappa} r_\epsilon^\tau)\]
         between the two integrals.
       \item Using that $\nabla\phi = O(r^{-3})$, we find that on $\tilde{\Sigma}_\delta$ we have $|\phi - \underline\phi_\delta| \leq C\epsilon^3 r^{-2} r^{-3}$. This leads to an error of
         \[ \int_{r_\epsilon}^1 \epsilon^{3+\kappa} r^{\tau-2} \epsilon^3 r^{-5} r^6\,dr = O(\epsilon^{6+\kappa} r_\epsilon^\tau) \]
         between the integrals.
       \item The area forms satisfy $\frac{dA_\delta}{dA_0} = 1 + O(\epsilon^3 r^{-3})$ (we are not using that  $\Sigma_0$ is minimal). This leads to the error
         \[ \int_{r_\epsilon}^1 \epsilon^{3+\kappa} r^{\tau-2} r^{-2} \epsilon^3 r^{-3} r^6\,dr = O(\epsilon^{6+\kappa} r_\epsilon^\tau). \]
       \end{itemize}
       As long as $r_\epsilon = \epsilon^\alpha$ with $\alpha$ sufficiently close to 1 and $\tau$ is chosen close to $-2$, this error is less than $\epsilon^{4 + \kappa''}$ for some $\kappa'' > 0$, as required. 
       \end{proof}

       We can now compute $h(\delta)$, where $m(\Sigma_\delta) = h(\delta)\zeta$.
       \begin{cor}\label{cor:hexpand}
         The function $h(\delta)$ satisfies
         \[ h(\delta) = c_6 b \delta^{4/3} + O(\delta^{4/3 + \kappa''}) \]
         for some $c_6, \kappa'' > 0$, where $b < 0$ is the $r^{-3}$ coefficient of the asymptotics of $H$. 
       \end{cor}
       \begin{proof}
         Note that in constructing the surface $\Sigma_\delta$ we treated $m(\Sigma_\delta) = h(\delta)\zeta$ as an equation on $\tilde{\Sigma}_\delta$. The mean curvature $m(\Sigma_\delta)$ is pulled back using that $\Sigma_\delta$ is a graph over $\tilde{\Sigma}_\delta$, while $\zeta$ is viewed as a function on $\tilde{\Sigma}_\delta$ using that near the support of $\zeta$ the surface $\tilde{\Sigma}_\delta$ is itself a graph over $\Sigma_0$. By definition $\zeta$ satisfies
         \[ \int_{\Sigma_0} \zeta\phi\, dA_0 = \int_{\Sigma_0} \phi^2\, dA_0. \]
         Pulling back $m(\Sigma_\delta)$ to $\Sigma_0$, and denoting the pullback of $\phi$ by $\underline{\phi}_\delta$ and pulled back area form by $dA_\delta$, by the previous Proposition we have
         \[ \int_{\Sigma_0} h(\delta) \zeta \underline{\phi}_\delta\, dA_\delta = c_5b \delta^{4/3} + O(\delta^{4/3 + \kappa''}). \]
         We have $|\phi - \underline{\phi}_\delta| = O(\delta)$ and $\frac{dA_\delta}{dA_0} = 1 + O(\delta)$. This implies
         \[ h(\delta) \left(\int_{\Sigma_0} \phi^2\, dA + O(\delta)\right) =  c_5b \delta^{4/3} + O(\delta^{4/3 + \kappa''}), \]
         from which the claim follows. 
       \end{proof}

       \subsection{The functions $\phi_\delta$ and $\xi_\delta$} \label{subsec:s6}
       In this section we construct functions $\phi_\delta$ and $\xi_\delta$ on the surfaces $\Sigma_\delta$ that we will need later. The $\phi_\delta$ are the functions generating the family $\Sigma_\delta$, while $\xi_\delta$ satisfy the equation $L_{\Sigma_\delta}\xi_\delta = \zeta - c_\delta\phi_\delta$ for suitable constants $c_\delta$ with the additional constraint that $\langle \phi_\delta, \xi_\delta\rangle_{L^2} = 0$ on $\Sigma_\delta$. Since $m(\Sigma_\delta) = h(\delta) \zeta$, we expect that $\phi_\delta$ satisfies $L_{\Sigma_\delta} \phi_\delta = h'(\delta)\zeta$. In addition recall that both $\tilde{\Sigma}_\delta$ and $\Sigma_\delta$ are the graph of $\delta\phi$ on the set $\{r = r_0, y > 0\}$ over $\Sigma_0$. Using this we see that $\phi_\delta = \phi$ along $\{r=r_0, y > 0\}$. These properties uniquely characterize the function $\phi_\delta$ on $\Sigma_\delta$: 

       \begin{prop}\label{prop:phideltadefn}
         For sufficiently small $\delta$ there is a unique function $\phi_\delta$ on $\Sigma_\delta$ with uniformly bounded $C^{2,\alpha}_{-2}$ norm, satisfying $L_{\Sigma_\delta} \phi_\delta = h'(\delta) \zeta$ and $\phi_\delta = \phi$ along $\{r=r_0, y > 0\}$. In addition we have
         \begin{itemize}
         \item  There is an $r_1 > 0$ and $C > 0$ such that on the set where $r < r_1$ and $y > 0$ we have $C^{-1} r^{-2} < \phi_\delta < C r^{-2}$.
         \item On the support of $\zeta$ we have $\phi_\delta = \phi + O(\delta^\kappa)$ for $\kappa > 0$. 
         \end{itemize}
       \end{prop}
       \begin{proof}
         On the surface $H$ there is a unique positive Jacobi field $\Phi = r^{-2} + O(r^{-3})$, arising from homothetic scalings. As a first approximation we construct $\phi_\delta$ by gluing together $\epsilon^3\phi$ and the function $\epsilon \Phi(\epsilon^{-1}\cdot )$ on $\epsilon H$ (recall that $\delta = \epsilon^3$). More precisely 
         consider the function $\tilde{\phi}_\delta$ on $\tilde{\Sigma}_\delta$ constructed as follows: on the region where $r  > 2r_\epsilon$ we let $\tilde{\phi}_\delta = \delta\phi$; on the region $r < r_\epsilon$ in our chart $\tilde{\Sigma}_\delta$ is the surface $\epsilon H$, and we let $\tilde{\phi}_\delta(r) = \epsilon \Phi(\epsilon^{-1} r)$ here. On the gluing region $r\in (r_\epsilon, 2r_\epsilon)$ the surface $\tilde{\Sigma}_\delta$ is a graph over $\Sigma_0$, or over $C$ in our chart, and using this we define
         \[ \tilde{\phi}_\delta(r) = (1-\chi(r_\epsilon^{-1}r)) \epsilon^3 \phi(r) + \chi(r_\epsilon^{-1}r) \epsilon \Phi(\epsilon^{-1} r) \]
         in terms of our cutoff function $\chi$. Using the rescaled variable $\tilde{r} = r_\epsilon^{-1} r$ we have
         \[ r_\epsilon^{2} \tilde{\phi}_\delta(r_\epsilon \tilde{r}) = \epsilon^3 \tilde{r}^{-2} + O(\epsilon^{3+\kappa}). \]
         It follows, using a similar analysis in different annular regions to that in Proposition~\ref{prop:mcest1}, that $\Vert L_{\tilde{\Sigma}_\delta} \tilde{\phi}_\delta\Vert_{C^{2,\alpha}_{-4}} \leq C \epsilon^{3+\kappa}$, and so if $\tau$ is sufficiently close to $-2$, we have $\Vert L_{\tilde{\Sigma}_\delta} \tilde{\phi}_\delta\Vert_{C^{2,\alpha}_{\tau-2}} \leq C \epsilon^{3+\kappa'}$ for some $\kappa' > 0$.

         Recall that $\Sigma_\delta$ is the graph of a function $u$ over $\tilde{\Sigma}_\delta$ with $\Vert u\Vert_{C^{2,\alpha}_\tau} \leq \epsilon^{3+\kappa}$, and so $\Vert u\Vert_{C^{2,\alpha}_1} \leq \epsilon^{\kappa'}$ if $\tau$ is close to $-2$. This allows us to compare the operators $L_{\tilde{\Sigma}_\delta}$ to $L_{\Sigma_\delta}$ as maps $ C^{2,\alpha}_\tau \to C^{0,\alpha}_{\tau-2}$: 
         \[ \label{eq:compareL} \Vert L_{\Sigma_\delta} - L_{\tilde{\Sigma}_\delta} \Vert_{C^{2,\alpha}_\tau \to C^{0,\alpha}_{\tau-2}}\leq C \epsilon^{\kappa'}. \]
         It follows that $\Vert L_{\Sigma_\delta} \tilde{\phi}_\delta\Vert_{C^{0,\alpha}_{\tau-2}} \leq C\epsilon^{\kappa'}$. Using \eqref{eq:compareL} and the invertibility result Proposition~\ref{prop:invertLdelta} it follows that the corresponding ``augmented'' operator is also invertible on $\Sigma_\delta$, i.e. we can find functions $v_\delta$ and $\lambda_\delta\in\mathbf{R}$ such that
         \[ L_{\Sigma_\delta}(v_\delta) + \lambda_\delta \zeta = L_{\Sigma_\delta}\tilde{\phi}_\delta, \]
         the function $v_\delta$ vanishes along $\{r=r_0, y >0\}$ and 
         \[ \label{eq:vdeltaest} \Vert v_\delta\Vert_{C^{2,\alpha}_\tau} + |\lambda_\delta| \leq C \epsilon^{\kappa'}. \]
         We can finally define $\phi_\delta = \tilde{\phi}_\delta - v_\delta$.
         The estimates $C^{-1} r^{-2} < \phi_\delta < C r^{-2}$, and $\phi_\delta = \phi + O(\delta^\kappa)$ follow from the construction of $\tilde{\phi}_\delta$, together with the bound $\Vert v_\delta\Vert_{C^{2,\alpha}_{-2}} \leq C\epsilon^{\kappa''}$ for some $\kappa''$. This in turn follows from \eqref{eq:vdeltaest} if $\tau$ is sufficiently close to $-2$.
         
         By construction we have $L_{\Sigma_\delta} \phi_\delta = \lambda_\delta\zeta$, and we need to show that $\lambda_\delta = h'(\delta)$. To see this, fix a small $\delta > 0$. For sufficiently small $a\not= 0$ the surface $\Sigma_{\delta +a}$ is the graph of a function $u_a$ over $\Sigma_\delta$. Using the invertibility of the operator $\tilde{L}_\delta$ on $\Sigma_\delta$ defined in the same way as \eqref{eq:Ltildedef} on $\tilde{\Sigma}_\delta$, we can find $u_a\in C^{2,\alpha}_\tau$ for sufficiently small $a$ using the implicit function theorem. We find that there is some $\kappa > 0$ such that for any $\tau < -2$ sufficiently close to $-2$ we have $\Vert u_a - a\phi_\delta \Vert_{C^{2,\alpha}_\tau}\leq a^{1+\kappa}$ once $\delta$ is sufficiently small (as $\tau$ approaches $-2$ we will need to take smaller $\delta$, since the norm of the inverse of the linearized operator may blow up). In particular this shows that $\phi_\delta$ generates the family of surfaces $\Sigma_\delta$, which also implies that $L_{\Sigma_\delta} \phi_\delta = h'(\delta)\zeta$ as required. 
       \end{proof}

       An immediate consequence is the behavior of the areas of the surfaces $\Sigma_\delta$. Namely, for small $\delta > 0$ we have
       \[ \frac{d}{d\delta} \mathrm{Area}(\Sigma_\delta) = -\int_{\Sigma_\delta} m(\Sigma_\delta)\phi_\delta\, dA = -h(\delta) \int_{\Sigma_\delta} \zeta \phi_\delta\,dA = -c_5 b\delta^{4/3} + O(\delta^{4/3+\kappa}). \]
       In particular since $c_5 b < 0$, the surfaces $\Sigma_\delta$ have strictly larger area than $\Sigma_0$, for sufficiently small $\delta$. We record the following consequence that we will use:
       \[ \label{eq:ASbound} \mathrm{Area}(\Sigma_\delta) \leq \mathrm{Area}(\Sigma_0) + C |\delta| |h(\delta)|, \]
       for a constant $C > 0$. 

       We next construct the functions $\xi_\delta$.
       \begin{prop}\label{prop:xidefn}
         For sufficiently small $\delta$ there are functions $\xi_\delta$ on $\Sigma_\delta$ satisfying
         \begin{itemize}
         \item $\Vert \xi_\delta\Vert_{C^{2,\alpha}_\tau} \leq C_\tau$ for any $\tau\in (-3,-2)$, with $C_\tau$ depending on $\tau$ but not on $\delta$.
         \item $L_{\Sigma_\delta} \xi_\delta = \zeta - c_\delta\phi_\delta$, with $c_\delta = 1 + O(\delta^\kappa)$ for some $\kappa > 0$.
         \item $\langle \xi_\delta, \phi_\delta \rangle_{L^2(\Sigma_\delta)} = 0$.
           \end{itemize}
         \end{prop}
         \begin{proof}
           The functions $\phi_\delta$ are uniformly bounded in $C^{2,\alpha}_\tau$, therefore also in $C^{2,\alpha}_{\tau-2}$ by \eqref{eq:normcomp1}. Let us define $u_\delta$ to be the unique solution of
           \[ L_{\Sigma_\delta} u_\delta + \lambda_\delta\zeta = \phi_\delta, \]
           with $u_\delta$ vanishing along $\{r=r_0, y > 0\}$, for some constant $\lambda_\delta$. The $u_\delta$ are uniformly bounded in $C^{2,\alpha}_\tau$. It follows that we can find constants $a_\delta$, bounded independently of $\delta$ such that
           \[ \tilde{u}_\delta = u_\delta + a_\delta\phi_\delta \]
           is $L^2$ orthogonal to $\phi_\delta$. The $\tilde{u}_\delta$ are also uniformly bounded in $C^{2,\alpha}_\tau$, and satisfy
           \[ L_{\Sigma_\delta} \tilde{u}_\delta + \lambda_\delta \zeta = a_\delta h'(\delta)\zeta + \phi_\delta. \]
           Multiplying by $\phi_\delta$, and then integrating by parts we get
           \[ \int_{\Sigma_\delta} h'(\delta) \zeta \tilde{u}_\delta + \lambda_\delta \phi_\delta\zeta\, dA= \int_{\Sigma_\delta} a_\delta h'(\delta)\zeta \phi_\delta + \phi_\delta^2\,dA. \]
           Since $h'(\delta) = O(\delta^{\kappa})$, it follows from this that $\lambda_\delta = 1 + O(\delta^\kappa)$.
           We now define $\xi_\delta = (-\lambda_\delta + a_\delta h'(\delta))^{-1} \tilde{u}_\delta$. 
         \end{proof}

  \section{The comparison surfaces $T_\delta$}\label{sec:Tdelta}
  Let $V_\delta = C(\Sigma_\delta) \subset \mathbf{R}^8\times\mathbf{R}$ denote the cone over the surface $\Sigma_\delta$
  provided by Proposition~\ref{prop:Sdelta}.   The cones $V_\delta$ have mean curvature $m(V_\delta) = h(\delta)\zeta \rho^{-2}$, where we extend $\zeta$ from $\Sigma_\delta$ to $V_\delta$ as a homogeneous degree one function. Since for small $\delta\not=0$ we have $h(\delta)\not=0$, these cones are not minimal unless $\delta=0$. In this section we construct small minimal perturbations of them on annuli of the form $|\ln \rho| < |\delta|^{-\kappa}$ for sufficiently small $\kappa > 0$. Here $\rho = ( |x|^2 + |y|^2)^{1/2}$ is the distance from the origin in $\mathbf{R}^8\times\mathbf{R}$. 

  On the linearized level we are trying to find a function $u$ over $V_\delta$ so that $L_{V_\delta} u = -h(\delta)\zeta \rho^{-2}$ since then the graph of $u$ over $V_\delta$ will be minimal to leading order. In the limit $\delta\to 0$ this is roughly equivalent to solving the equation $L_{V_0} u = \phi \rho^{-2}$. Since $\phi$ is in the cokernel of the Jacobi operator $L_{\Sigma_0}$ on the link, this equation has no homogeneous degree one solution $u$, but we do have
  \[ L_{V_0} (c \phi \ln \rho) = \phi \rho^{-2} \]
for a suitable constant $c$. This suggests that we can try to find minimal perturbations  $T_\delta$ of $V_\delta$ given to leading order by the graph of the function $u = -c h(\delta) \phi_\delta\ln\rho$. This is only well defined as long as $|h(\delta) \ln \rho|$ does not get too large, which leads to us considering annuli where $|\ln \rho| < |\delta|^{-\kappa}$. This construction is closely related to that of Adams-Simon~\cite{AS88}, where minimal surfaces were constructed with logarithmic decay to their tangent cones. With additional work we expect that one could construct a minimal surface with an isolated singularity at the origin, having tangent cone $V_0 = C\times \mathbf{R}$ there, such that the surface converges at a logarithmic rate to this tangent cone. Such a surface would be modeled on different $T_\delta$ with $\delta\to 0$ as $\rho\to 0$. For our purposes here, however, the simpler construction of the $T_\delta$ defined on annuli suffices.

We construct $T_\delta$ as the graph of a function $u$ over $V_\delta$, where $u$ is in a suitable doubly weighted space $C^{2,\alpha}_{\gamma, \tau}$. The weight $\tau$ accounts for the singular ray of $V_0$, just like in the weighted spaces $C^{2,\alpha}_\tau$ on $\Sigma_\delta$ used in the previous section. The new weight $\gamma$ is related to the cone structure of $V_\delta$. 
  \begin{definition}\label{defn:weighted}
    Let $\gamma,\tau\in\mathbf{R}$ and let $f$ be a function over a subset $U\subset V_\delta$, locally in $C^{k,\alpha}$. We define the weighted norm $\Vert f\Vert_{C^{k,\alpha}_{\gamma,\tau}}$ as follows. For $Q, R > 0$ let us define the annular region $A_{Q,R}$ to be the set where $Q/2 < \rho < 2Q$ and $R/2 < r < 2R$. We then define
    \[ \Vert f \Vert_{C^{k,\alpha}_{\gamma,\tau}(U)} = \sup_{Q, R > 0} Q^{\tau-\gamma} R^{-\tau} \Vert f\Vert_{C^{k,\alpha}_{R^{-2}g}(A_{Q,R}\cap U)}.\]
    Here $g$ denotes the metric on $V_\delta$ and on the right hand side we are measuring the $C^{k,\alpha}$-norm with respect to the rescaled metric $R^{-2}g$.
  \end{definition}
  Note that by construction, the rescaled metric $R^{-2}g$ has bounded geometry on the annulus $A_{R,Q}$. To see this, note that the regularity scale of $\Sigma_\delta$ at each point (viewed as sitting in the unit sphere of $\mathbf{R}^9$) is uniformly equivalent to $r$, and so by scaling the regularity scale of $V_\delta$ is proportional to $\rho (r/\rho) = r$, since this is a degree one homogeneous function equal to $r$ on the unit sphere. As in the previous section, our construction will be invariant under the symmetry group $G:=O(4)\times O(4)$ and we denote the corresponding function spaces by $C^{k,\alpha,G}_{\gamma, \tau}$. We will have $\tau\in (-3,-2)$ as before. 

  We have the following basic comparison result between norms with different weights.
  \begin{lemma}\label{prop:normcomp1}
    There is a constant $C > 0$ such that given $f$ in $C^{2,\alpha}_{loc}$ we have
    \[ \Vert f\Vert_{C^{2,\alpha}_{1,1}} \leq C |\delta|^{(\tau-1)/3}\Vert f\Vert_{C^{2,\alpha}_{1,\tau}}. \]
    \end{lemma}
    \begin{proof}
      This follows directly from the definition of the norms using that on $V_\delta$ we have $r / \rho > C^{-1} |\delta|^{1/3}$. To see this recall that near the singular points of $\Sigma_0$ the surface $\Sigma_\delta$ is given as an $O(|\delta| r^{-2+\kappa})$ graph  over $\pm \delta^{1/3}H$. 
    \end{proof}

  The following result, analogous to Proposition~\ref{prop:invertLdelta}, is the main ingredient in constructing the minimal perturbations $T_\delta$ of $V_\delta$. Note that in contrast with the linearized operator on $\Sigma_\delta$, here we do not compensate for the function $\phi$ in the cokernel of $L_{\Sigma_0}$ using the function $\zeta$. The price of this is that we only invert the operator on an annular region, and the norm of the inverse blows up as the size of the annulus goes to infinity. 
  \begin{prop}\label{prop:Ldinv}
    Suppose that $\tau\in (-3,-2)$ is sufficiently close to $-2$, and $\kappa > 0$ is sufficiently small. Let  $U_\delta \subset V_\delta$ denote the subset where $|\ln \rho| \leq |\delta|^{-\kappa}$. There is a $C > 0$ such that for all $\delta$ sufficiently small the Jacobi operator 
    \[ \label{eq:LVdeltaUdelta}
      L_{V_\delta} : C^{2,\alpha,G}_{1, \tau}(U_\delta) \to C^{0,\alpha,G}_{-1,\tau-2}(U_\delta) \]
    on the cone $V_\delta$ has a right inverse $P_\delta$ with norm bounded by $C|\delta|^{-\kappa}$.
  \end{prop}
  \begin{proof}
    This result can be proven similarly to \cite[Proposition 22]{Sz17}. We break the argument into several steps. 
    \bigskip
    
    \noindent {\em Step 1.} We first need a result analogous to \cite[Proposition 23]{Sz17}, on invertibility of $L_{V_0}$ between suitable weighted spaces. The main difference is that we want to work with a weight which does not avoid the indicial roots, and because of this we need to work orthogonal to the corresponding Jacobi field. More precisely, let us denote by $C^{k,\alpha, G, \perp}_{\gamma, \tau}(V_\delta) \subset C^{k,\alpha}_{\gamma, \tau}(V_\delta)$ the subspace of $G$-invariant functions that are $L^2$-orthogonal to $\phi$ on every level set of $\rho$ (i.e. all the cross sections of the cone). We then claim that
    \[ \label{eq:LV0inv} L_{V_0} : C^{2,\alpha, G, \perp}_{1, \tau}(V_0) \to C^{0,\alpha, G, \perp}_{-1, \tau-2}(V_0) \]
    is invertible for $\tau\in (-3,-2)$. This can be shown exactly as \cite[Proposition 24]{Sz17}. The main ingredient is that, analogously to \cite[Proposition 13]{Sz17}, the map
    \[ L_{\Sigma_0} : C^{2,\alpha, G, \perp}_{\tau}(\Sigma_0) \to C^{0,\alpha, G, \perp}_{\tau-2}(\Sigma_0) \]
    is invertible for $\tau\in (-3,-2)$. As the notation suggests, we are restricting ourselves to $G$-invariant functions on $\Sigma_0$ that are $L^2$-orthogonal to $\phi$. In turn the invertibility of this map follows from standard Fredholm theory in weighted spaces on manifolds with conical singularities using that $L_{\Sigma_0}$ is self-adjoint and the only $G$-invariant elements in the kernel are multiples of $\phi$ (see Lockhart-McOwen~\cite{LM85} or Marshall~\cite{Marsh02}). The rest of the proof of the invertibility of \eqref{eq:LV0inv} is exactly as in \cite{Sz17}. 

    \bigskip
    \noindent{\em Step 2.}
    We need a result analogous to \cite[Proposition 21]{Sz17}, namely that the operator
    \[ \label{eq:LHR} L_{H \times \mathbf{R}}: C^{2,\alpha}_\tau(H\times \mathbf{R}) \to C^{0,\alpha}_{\tau-2}(H\times \mathbf{R}) \]
    is invertible for $\tau\in (-3,-2)$, with the weighted spaces defined analogously to \eqref{eq:wH1}. The key ingredient for this is the following, analogous to \cite[Proposition 18]{Sz17}: for $\tau\in (-3,-2)$, $\lambda \geq 0$, and $u$ having compact support on $H$, we can find $f$ such that
    \[ L_H f - \lambda f = u, \]
    and $r^{-\tau}|f| \leq C \Vert u\Vert_{C^0_{\tau-2}}$ for a constant $C$ depending on $\tau$. Note first that the existence of the positive Jacobi field $\Phi$ on $H$ with $\Phi= r^{-2} + O(r^{-3})$ as $r\to\infty$ implies that $L_H$ has trivial kernel in $C^{2,\alpha}_\tau$ for $\tau\in (-3,-2)$. The Fredholm theory in weighted spaces then implies that $L_H : C^{2,\alpha}_\tau \to C^{0,\alpha}_{\tau-2}$ is invertible, which settles the $\lambda=0$ case of our claim. When $\lambda > 0$ then we can argue as in the proof of \cite[Proposition 18]{Sz17}.

      \bigskip
      \noindent {\em Step 3.}
      We can now prove the invertibility of \eqref{eq:LVdeltaUdelta} by constructing an approximate inverse for sufficiently small $\delta$, similarly to the proof of \cite[Proposition 22]{Sz17}. As a preliminary step, given a function $\tilde{u}\in C^{0,\alpha, G}_{-1,\tau-2}(U_\delta)$ with $\Vert\tilde{u}\Vert \leq 1$, we write
      \[ \label{eq:utilde1} \tilde{u} = u + u_0, \]
      where $u\in C^{0,\alpha, G, \perp}_{-1,\tau-2}(V_\delta)$ and $u_0 = u_0(\rho) \phi_\delta \rho^{-1}$ for a function $u_0 \in C^{0,\alpha}_{-1}(0,\infty)$. Here we are extending $\phi_\delta$ to $V_\delta$ as a degree one homogeneous function and $C^{0,\alpha, G, \perp}_{-1,\tau-2}$ denotes those functions that are orthogonal to $\phi_\delta$ on each cross section of $V_\delta$. We can construct the approximate inverse on the function $u$ following \cite[Proposition 22]{Sz17} closely, and then we will deal with the piece $u_0$. 

     To construct the approximate inverse on $u$ we first write $u=u_1 + u_2$, where $u_2 = \chi(\Lambda^{-1} r\rho^{-1} |\delta|^{-1/3})$. Here $\chi$ is the cutoff function we used before ($\chi(s) = 1$ for $s < 1$, and $\chi(s)=0$ for $s > 2$), and $\Lambda$ is a large constant to be chosen. Then $u_1$ is supported on the set where $r/\rho > \Lambda |\delta|^{1/3}$. By Proposition~\ref{prop:Sdelta} the cone $V_\delta$ is the graph of a function $f_0$ over $\delta^{1/3}H$ on the slice $y=1$ with $|f_0| < C|\delta| r^{-2+\kappa}$. In turn on the region $r > \Lambda |\delta|^{1/3}$, the surface $\delta^{1/3} H$ itself is an $O(|\delta| r^{-2})$-graph over $C$. In sum on the slice $y=1$ on the region $r > \Lambda |\delta|^{1/3}$, $V_\delta$ is the graph of a function $r f_1$ over $C$, where $|f_1| < C\Lambda^{-1}$. It follows that if $\Lambda$ is sufficiently large and $\delta$ is small, we can approximate $V_\delta$ using $V_0$ on the support of $u_1$. Using this we think of $u_1$ as being defined on $V_0$. 

      Similarly $u_2$ is supported on the region $r/\rho < \Lambda |\delta|^{1/3}$. Fix a large $B > 0$ and let $y_0\in \mathbf{R}$. Consider the interval $I$ of radius $2B|y_0| |\delta|^{1/3}$ around $y_0$, and rescale the region $\{ r < 2\Lambda |y_0 \delta^{1/3}|, y\in I\}$ in $V_\delta$ by a factor of $|y_0 \delta^{1/3}|^{-1}$. By the construction this rescaled surface can be seen as a perturbation of the region $\{|r'| < 2\Lambda, |y'| < 2B\}$ in the product $H\times \mathbf{R}$, if for given $\Lambda, B$ we choose $\delta$ sufficiently small. We further decompose $u_2 = \sum u_{2,j}$ into pieces supported in regions of this type. 

      Arguing as in the proof of \cite[Proposition 22]{Sz17} we now want to use the invertibility of \eqref{eq:LV0inv} and \eqref{eq:LHR} to invert the operator on the pieces $u_1, u_{2,j}$, and then reassemble them to produce an approximate inverse for $L_{V_\delta}$. The only new feature is that to use \eqref{eq:LV0inv} we first need to write $u_1 = u_{1,0} + u_{1,\perp}$, where $u_{1,\perp} \in C^{0,\alpha, G,\perp}_{-1, \tau-2}$ and $u_{1,0} = u_{1,0}(\rho)\phi \rho^{-1}$ for a function $u_{1,0} \in C^{0,\alpha}_{-1}(0,\infty)$. Since we have ensured that on $V_\delta$ the function $u$ is orthogonal in each cross section to $\phi_\delta$, it is not hard to see that we have
      \[  \Vert u_{1,0}\Vert_{C^{0,\alpha}_{\tau-2, -1}} \leq |\delta|^{\kappa'} \]
      for some $\kappa' > 0$ for sufficiently small $\delta$. This piece can be discarded when building the approximate inverse since it does not affect the estimate \eqref{eq:Puest1} below. Applying the inverses of \eqref{eq:LV0inv} and \eqref{eq:LHR} to the remaining pieces, and then reassembling them using further cutoff functions as in \cite[Proposition 22]{Sz17} we end up with a function $Pu$ on the annular region $U_\delta$, satisfying
      \[ \label{eq:Puest1} \Vert Pu\Vert_{C^{2,\alpha}_{1,\tau}} &\leq C, \\
      \Vert L_{V_\delta}Pu - u\Vert_{C^{0,\alpha}_{-1,\tau-2}} &\leq C|\delta|^{\kappa'}\]
for some $\kappa' > 0$. 

\bigskip
\noindent{\em Step 4.} It remains to deal with the piece $u_0 = u_0(\rho)\phi_\delta \rho^{-1}$ in the decomposition \eqref{eq:utilde1}. Note that for any $f(\rho)$ we have
\[ L_{V_\delta} (\rho f \phi_\delta \rho^{-1}) &= ( (\rho f)'' + 7\rho^{-1}(\rho f)' - 7\rho^{-1} f) \phi_\delta \rho^{-1} + \rho^{-1} f L_{\Sigma_\delta}(\phi_\delta \rho^{-1}) \\
  &= \rho^{-8} (\rho^9 f')' \phi_\delta + \rho^{-1} f h'(\delta) \zeta\rho^{-1}, \]
noting that $\phi_\delta \rho^{-1}$ is homogeneous with degree zero, and $L_{\Sigma_\delta} \phi_\delta = h'(\delta) \zeta$. As with $\phi_\delta$, we are extending $\zeta$ to $V_\delta$ to have degree one. 
Given $u_0(\rho)\in C^{0,\alpha}_{-1}(0,\infty)$, we can find $f \in C^{2,\alpha}$ satisfying
\[ \rho^{-9}(\rho^9 f')' = u_0 \rho^{-1} \]
by the formula
\[ f(\rho) =  \int_1^\rho t^{-9} \int_1^t u_0(s) s^8\,ds. \]
From this we see that $|f(\rho)| \leq C |\ln \rho| \Vert u_0\Vert_{C^{0,\alpha}_{-1}}$. It follows that $Pu_0 = \rho f \phi_\delta\rho^{-1}$ satisfies
  \[ \label{eq:Pu0} \Vert Pu_0 \Vert_{C^{2,\alpha}_{1,\tau}(U_\delta)} &\leq C|\delta|^{-\kappa} \Vert u_0\Vert_{C^{0,\alpha}_{-1,\tau-2}(U_\delta)}, \\
    \Vert L_{V_\delta} Pu_0- u_0 \Vert_{C^{0,\alpha}_{-1,\tau-2}(U_\delta)} &\leq C |\delta|^{-\kappa} |h'(\delta)|. \]
  Recall that $|h'(\delta)| \leq C|\delta|^{\kappa_1}$ for some $\kappa_1 > 0$, and so we choose $\kappa = \min\{\kappa_1/2, \kappa'\}$ for the $\kappa'$ in \eqref{eq:Puest1}. 

  For any $\tilde{u}\in C^{0,\alpha}_{-1,\tau-2}(U_\delta)$ we can now define $P\tilde{u} = Pu_0 + Pu$ in terms of the decomposition \eqref{eq:utilde1}, and the estimates \eqref{eq:Puest1} and \eqref{eq:Pu0} imply
  \[ \Vert P\tilde{u} \Vert_{C^{2,\alpha}_{1,\tau}(U_\delta)} &\leq C|\delta|^{-\kappa} \Vert \tilde{u} \Vert_{C^{0,\alpha}_{-1,\tau-2}(U_\delta)}, \\
    \Vert L_{V_\delta} P\tilde{u} - \tilde{u} \Vert_{C^{0,\alpha}_{-1,\tau-2}(U_\delta)} &\leq C |\delta|^{-\kappa/2}  \Vert \tilde{u} \Vert_{C^{0,\alpha}_{-1,\tau-2}(U_\delta)}. \]
   Once $\delta$ is sufficiently small, the second estimate implies that $(L_{V_\delta}P)$ is invertible with uniformly bounded inverse, and then $P_\delta = P(L_{V_\delta}P)^{-1}$ is a right inverse  of $L_{V_\delta}$ with norm bounded by $C|\delta|^{-\kappa}$. 
  \end{proof}

 Given the invertibility of the linearized operator, the construction of $T_\delta$ as a graph over $V_\delta$ is very similar to what we did in Section~\ref{sec:Sigmadelta}. 
  For a function $f$ over $V_\delta$ let $m_{V_\delta}(f)$  denote the mean curvature of the graph of $f$ and define the nonlinear operator $\mathcal{Q}_{V_\delta}$ by
  \[ m_{V_\delta}(f) = m_{V_\delta}(0) + L_{V_\delta}(f) + \mathcal{Q}_{V_\delta}(f). \]
  This satisfies the following estimate, just like in Proposition~\ref{prop:Qprop1}. 
  \begin{prop} \label{prop:Qest}
    Suppose that $f_1,f_2$ are in $C^{2,\alpha}_{loc}$ on an open subset $U$ of $V_\delta$. There is a $C > 0$ independent of $\delta, U$ such that if $\Vert f_i\Vert_{C^{2,\alpha}_{1,1}} \leq C^{-1}$,  then
    \[ \Vert \mathcal{Q}_{V_\delta}(f_1) - \mathcal{Q}_{V_\delta}(f_2)\Vert_{C^{0,\alpha}_{-1,\tau-2}} \leq C \Big( \Vert f_1\Vert_{C^{2,\alpha}_{1,1}}+ \Vert f_2\Vert_{C^{2,\alpha}_{1,1}}\Big) \Vert f_1 - f_2\Vert_{C^{2,\alpha}_{1,\tau}}. \]
    In particular since $\mathcal{Q}_{V_\delta}(0)=0$ we have
    $\Vert \mathcal{Q}_{V_\delta}(f_1)\Vert_{C^{0,\alpha}_{-1,\tau-2}} \leq C\Vert f_1\Vert_{C^{2,\alpha}_{1,1}} \Vert f_1 \Vert_{C^{2,\alpha}_{1,\tau}}$.
  \end{prop}

  The main existence result is the following. 
\begin{prop}\label{prop:Tdelta1}
  Suppose that $\tau\in (-3,-2)$ is sufficiently close to $-2$, and $\kappa > 0$ is sufficiently small. Let $U_\delta$ denote the annular region $|\ln\rho| < |\delta|^{-\kappa}$. There is an $\epsilon > 0$ with the following property. For sufficiently small $\delta$ there is a function $f$ on $V_\delta$ in the annulus $U_\delta$ such that the surface $T_\delta$ defined as the graph of
  \[ -h(\delta)\Big[ 7^{-1} \phi_\delta \ln\rho + \xi_\delta\Big] + f \]
  over $V_\delta$ is minimal, and $f$ satisfies the estimate $\Vert f \Vert_{C^{2,\alpha}_{1,\tau}} \leq |h(\delta)| |\delta|^\epsilon$. The function $\xi_\delta$ used here was defined in Proposition~\ref{prop:xidefn}, and just like $\phi_\delta$, we extend it as a degree one function to $V_\delta$.

  In particular for $|\ln \Lambda| < |\delta|^{-\kappa}$ the surface $\Lambda T_\delta \cap \{y=1\}$ is the graph of a function $f_1$ over $\delta^{1/3}H$, with $|f_1| < C|\delta| r^{-2+\kappa}$. 
\end{prop}
\begin{proof}
  Let us write
  \[ f_0 = -h(\delta)\Big[7^{-1} \phi_\delta \ln \rho + \xi_\delta \Big]. \]
  Our goal is to find an $f$ over $U_\delta$ satisfying the required estimate and the equation
  \[ m_{V_\delta}(f_0 + f) = 0. \]
  We have
  \[ m_{V_\delta}(f_0+f) = m_{V_\delta}(0) + L_{V_\delta}(f_0+f) + \mathcal{Q}_{V_\delta}(f_0 + f), \]
  so using the right inverse from Proposition~\ref{prop:Ldinv}  it is enough to solve the fixed point problem $\mathcal{N}(f) = f$, where
  \[ \mathcal{N}(f) = - P_\delta \big( m_{V_\delta}(0) + L_{V_\delta}(f_0) + \mathcal{Q}_{V_\delta}(f_0+f) \big).   \]

By construction we have $m_{V_\delta}(0) = h(\delta)\zeta \rho^{-2}$ (recall that we extend $\zeta$ as a degree one function to $V_\delta$), and
\[ L_{V_\delta} f_0 = -h(\delta) \zeta \rho^{-2} + h(\delta) h'(\delta) \zeta \rho^{-2} \ln \rho. \]
It follows that
\[ \mathcal{N}(f) = P_\delta\big( -h(\delta) h'(\delta) \zeta \rho^{-2}\ln \rho - \mathcal{Q}_{V_\delta}(f_0+f)\big). \]
Let us define the set
\[ E = \{ f\in C^{2,\alpha}_{1,\tau}(U_\delta)\,:\, \Vert f\Vert_{C^{2,\alpha}_{1,\tau}}\leq |h(\delta)| |\delta|^\epsilon\}, \]
for $\epsilon > 0$ to be chosen. We will show that if $\epsilon$ is chosen small, then $\mathcal{N}$ is a contraction on $E$ for sufficiently small $\delta$.

Let $f\in E$. On the annulus $U_\delta$ we have
\[ \Vert f_0\Vert_{C^{2,\alpha}_{1,\tau}} \leq C |h(\delta)| |\delta|^{-\kappa}, \]
so
\[ \Vert  f_0 + f\Vert_{C^{2,\alpha}_{1,\tau}} \leq C |h(\delta)| |\delta|^{-\kappa} \]
for a larger $C$. 
By Lemma~\ref{prop:normcomp1} we have
\[ \Vert f_0\Vert_{C^{2,\alpha}_{1,1}} \leq C |h(\delta)| |\delta|^{(\tau-1)/3 -\kappa}. \]
By the properties of $h$ there is an $\epsilon_1  >0$ such that $|h(\delta)| < |\delta|^{1+\epsilon_1}$ and $|h'(\delta)| < |\delta|^{\epsilon_1}$ for sufficiently small $\delta$. It then follows from the estimate for the inverse $P_\delta$ in Proposition~\ref{prop:Ldinv} and the estimate in Proposition~\ref{prop:Qest} that
\[ \Vert\mathcal{N}(f)\Vert_{C^{2,\alpha}_{1,\tau}} &\leq C |\delta|^{-\kappa}( |h(\delta)| |\delta|^{\epsilon_1} |\delta|^{-\kappa} + |h(\delta)|^2|\delta|^{(\tau-1)/3 - 2\kappa}) \\
&\leq |h(\delta)| |\delta|^{\epsilon} \]
for sufficiently small $\delta$ if $\epsilon, \kappa$ are sufficiently small and $\tau$ is sufficiently close to $-2$. Therefore $\mathcal{N}$ maps $E$ into $E$. 

To see that $\mathcal{N}$ is a contraction, note that for $f_1,f_2\in E$ we have, as above, that 
\[ \Vert f_0 + f_i\Vert_{C^{2,\alpha}_{1,1}} \leq C |h(\delta)| |\delta|^{(\tau-1)/3 -\kappa}, \]
and so using Proposition~\ref{prop:Qest} we have
\[ \Vert \mathcal{N}(f_1) - \mathcal{N}(f_2) \Vert_{C^{2,\alpha}_{1,\tau}} &\leq C |\delta|^{-\kappa} |h(\delta)| |\delta|^{(\tau-1)/3 -\kappa} \Vert f_1 - f_2\Vert_{C^{2,\alpha}_{1,\tau}} \\
  &\leq \frac{1}{2}\Vert f_1 - f_2\Vert_{C^{2,\alpha}_{1,\tau}} \]
for small $\delta$, if $\kappa$ is sufficiently small, and $\tau$ is sufficiently close to $-2$.

By choosing $\tau$ sufficiently close to $-2$, and $\kappa > 0$ sufficiently small, we find that in the annulus $B_2\setminus B_{1/2}$ the surface $\Lambda T_\delta$ is the graph of a function $F$ over $V_\delta$ satisfying $|F| \leq C|\delta| r^{-2+\kappa}$, if $|\ln \Lambda| < |\delta|^{-\kappa}$. Using property (ii) in Proposition~\ref{prop:Sdelta}, this implies the final claim in the proposition. 
\end{proof}

\begin{remark}\label{rem:Wcone}
Let $W_\delta$ be the graph of $-h(\delta)\xi_\delta$ over $V_\delta = C(\Sigma_\delta)$. Proposition~\ref{prop:Tdelta1} then implies that for $|\ln \Lambda| < |\delta|^{-\kappa}$, $\Lambda T_\delta$ is the graph of a function $F'$ over $W_\delta$ where $\Vert F'\Vert_{C^{2,\alpha}_{1,-2}(B_1 \setminus B_{1/2})}\leq C(1 + |\ln \Lambda|) |h(\delta)|$ on the annulus $B_1\setminus B_{1/2}$. This estimate (applied for bounded $|\ln\Lambda|$) will be used in Section~\ref{sec:mainargument}, in particular in Lemma~\ref{lem:triangle}. Note that if we view $T_\delta$ as the graph of $F$ over $V_\delta$, instead of over $W_\delta$, then from Proposition~\ref{prop:xidefn} the bound we get is $\Vert F\Vert_{C^{2,\alpha}_{1,\tau}(B_1\setminus B_{1/2})}\leq C|h(\delta)|$ for any $\tau < -2$ (with $C$ depending on $\tau$), which in turn would lead to $\Vert F\Vert_{C^{2,\alpha}_{1,-2}}\leq C |h(\delta)| |\delta|^{-\epsilon}$ for any $\epsilon > 0$, with $C$ depending on $\epsilon$. 
\end{remark}

We will need an estimate for the integral appearing in the monotonicity formula over $T_\delta$. For this we have the following.
\begin{prop}\label{prop:Tm}
  For sufficiently small $\delta$ we have
  \[ \label{eq:m1} \int_{T_\delta \cap(B_{1/2}\setminus B_{1/4})} \frac{|z^{\perp}|^2}{|z|^{10}}  > C^{-1} |h(\delta)|^2, \]
for a uniform $C > 0$, where $z^\perp$ denotes the component of the position vector $z\in T_\delta$ normal to $T_\delta$. 
\end{prop}
\begin{proof}
  We use that $T_\delta$ is the graph of  $F$ over the cone $V_\delta$, given by
  \[ F = -h(\delta)\Big[ 7^{-1}\phi_\delta \ln\rho + \xi_\delta\Big] + f, \]
  and $\Vert f\Vert_{C^{2,\alpha}_{1, \tau}} \leq |h(\delta)| |\delta|^\epsilon$. In particular, choosing $\tau$ sufficiently close to $-2$, we have 
  \[ |r^{-1} F| + |\nabla F| &< C |h(\delta)| r^{-3.1}, \\
    |r^{-1} f| + |\nabla f| &< C |h(\delta)| |\delta|^\epsilon r^{-3.1}. \]
  Let us write $n_V, n_T$ for the normal vectors to $V_\delta, T_\delta$, thinking of both as vector fields along $V_\delta$. We can apply Lemma~\ref{lem:grapharea} after scaling, to find that
  \[ n_T = n_V - \nabla F + O( |h(\delta)|^2 r^{-6.2} ), \]
  while the positions vectors satisfy $z_T = z_V + F n_V$. It follows that
  \[ z_T\cdot n_T &= (z_V + F n_V)\cdot (n_V - \nabla F + O( |h(\delta)|^2 r^{-6.2})) \\
    &= - z_V \cdot \nabla F + F + O(|h(\delta)|^2 r^{-6.2}). \]
  Since $V_\delta$ is a cone, we have $F - z_V \cdot \nabla F = F - \rho \partial_\rho F$, which vanishes on the degree one part of $F$. In particular we have
  \[ F - \rho\partial_{\rho} F &= h(\delta) 7^{-1}\phi_\delta + f - \rho\cdot \nabla f \\
    &= h(\delta) 7^{-1}\phi_\delta  + O( |h(\delta)| |\delta|^\epsilon r^{-3.1}). \]
  Since $|h(\delta)| \leq C |\delta|^{4/3}$, and on $V_\delta$ we have $r > C^{-1} |\delta|^{1/3}$, we have
  \[ \label{eq:zTperp} z_T\cdot n_T = h(\delta) 7^{-1} \phi_\delta + O( |h(\delta)| |\delta|^\epsilon r^{-3.2}). \]
  It follows that
  \[ |z_T^\perp|^2 > C^{-1} |h(\delta)|^2 \phi_\delta^2 - C |h(\delta)|^2 |\delta|^\epsilon r^{-5.2}, \]
  and by integrating we get \eqref{eq:m1} for sufficiently small $\delta$.

\end{proof}

We used the following lemma whose proof is by direct calculation. 
\begin{lemma}\label{lem:grapharea}
  Let $S$ be a hypersurface with second fundamental form $A_S$ satisfying $|A_S|\leq 1$. There are $\epsilon, C > 0$ with the following properties. Suppose that $f$ is a function on $S$ satisfying $|f|, |\nabla f| < \epsilon$, and let $S_f$ denote the graph of $f$ over $S$. The normal vectors $n, n_f$, and area forms $dA, dA_f$ of $S, S_f$ satisfies the following, under the natural identification of $S_f$ with $S$:
  \[  n_f &= n - \nabla f + E_1, \\
    \frac{dA_f}{dA} &= 1 - f m(S) + E_2, \]
  where $|E_1|, |E_2| \leq C\epsilon^2$, and $m(S)$ is the mean curvature of $S$.
\end{lemma}

\section{Non-concentration}\label{sec:nonconc}
The goal of this section is to prove the key non-concentration result, Proposition~\ref{prop:nonconcentration} below, and as a consequence to prove a three annulus lemma, Proposition~\ref{prop:D3annulus}.
We will need to consider minimal surfaces $M$ very close to one of the comparison surfaces $T_\delta$ on different annuli. Instead of varying the annulus, it is convenient to consider scalings $\Lambda T_\delta$ inside a fixed annulus, say $B_1\setminus B_{1/2} \subset \mathbf{R}^8\times \mathbf{R}$. Note that as long as $|\ln \Lambda| < |\delta|^{-\kappa}$, for a slightly smaller $\kappa$ than that in Proposition~\ref{prop:Tdelta1}, the surface $\Lambda T_\delta$ is defined in any fixed annulus for sufficiently small $\delta$.

To define the distance, we first note the following.
\begin{lemma}\label{lem:Phidefn1}
  There are $c_0, C > 0$ with the following properties:
  \begin{itemize}
    \item If $|\lambda| < c_0$, then $(1 + \lambda)H$ is the graph of a function $\Phi_\lambda$ over $H$ satisfying $C^{-1} r^{-2} < \lambda^{-1}\Phi_\lambda< C r^{-2}$ and we have
  \[ \Vert \Phi_\lambda -  \lambda\Phi\Vert_{C^{2,\alpha}_{-2}} \leq C\lambda^2, \]
  where $\Phi$ is the corresponding Jacobi field.
  \item 
    For any $a\in \mathbf{R}$ and $\lambda > 0$, the surface $(a + \lambda)^{1/3}H$ is on the positive side of the graph of $c_0 \min\{  \lambda  r^{-2}, r\}$ over $a^{1/3} H$. 
    \end{itemize}
\end{lemma}
\begin{proof}
  The estimates for $\Phi_\lambda$ follow from the fact that $H$ lies on the positive side of $C$ and is the graph of $r^{-2} + O(r^{-3})$ outside of a compact set. To see the second claim we can consider three cases separately.
  \begin{itemize}
  \item If $\lambda \leq \epsilon |a|$ for a sufficiently small $\epsilon  > 0$, then after scaling by $|a|^{-1/3}$ the claim is equivalent to saying that $(|a|^{-1}a + |a|^{-1}\lambda)^{1/3} H$ is on the positive side of the graph of $c_0 \min\{ |a|^{-1}\lambda r^{-2}, r\}$ over $(|a|^{-1} a)^{1/3} H$. If $\epsilon$ is sufficiently small, this follows from the previous claim, replacing $c_0$ by a smaller constant if necessary.
    \item If $\lambda \geq \epsilon^{-1} |a|$ for sufficiently small $\epsilon$, then we rescale by $\lambda^{-1/3}$. The claim is equivalent to asking for $(\lambda^{-1}a + 1)^{1/3}H$ to be on the positive side of the graph of $c_0 \min\{ r^{-2}, r\}$ over $(\lambda^{-1}a)^{1/3} H$. If $\epsilon$ is sufficiently small, then this holds (reducing the value of $c_0$ if necessary), using that $H$ is the graph of $r^{-2} + O(r^{-3})$ over $C$ outside of a compact set.
    \item For the $\epsilon > 0$ obtained in the previous two cases, suppose that $\epsilon |a| \leq \lambda \leq \epsilon^{-1}|a|$. We have that $(a+\lambda)^{1/3}H$ lies on the positive side of $(a + \epsilon |a|)^{1/3}H$, and applying the first case (setting $\lambda = \epsilon |a|$) we have that $(a + \epsilon |a|)^{1/3}H$ lies on the positive side of the graph of $c_0\min\{ \epsilon |a| r^{-2}, r\}$ over $a^{1/3}H$. Since we have
$\lambda \leq \epsilon^{-1} |a|$,
it then follows that $(a+\lambda)^{1/3}H$ lies on the positive side of the graph of  $\epsilon^2 c_0\min\{ \lambda  r^{-2}, r\}$. Replacing $c_0$ by $\epsilon^2c_0$ we are done. 
  \end{itemize} 
\end{proof}
It also follows by rescaling that $\delta^{1/3}(1 +\lambda) H$ is the graph of $f = \delta^{1/3}\Phi_{\lambda}(\delta^{-1/3}\cdot)$ over $\delta^{1/3}H$, and so $C^{-1}|\delta| r^{-2} < \lambda^{-1}f < C|\delta| r^{-2}$.

We will define the distance $D_{\Lambda T_\delta}(M; U)$ of the surface $M$ from $\Lambda T_\delta$  (and similarly from $V_\delta, W_\delta$) over a set $U\subset B_1\setminus B_{1/2}$. Recall from Proposition~\ref{prop:Tdelta1} that $\Lambda T_\delta$ is the graph of a function $f$ over $(y^3\delta)^{1/3}H$, with $|f| \leq C |\delta| r^{-2+\kappa}$. We first define neighborhoods of such surfaces in each $y$-slice, depending on an additional small parameter $\beta > 0$.  At several points below we will reduce the value of $\beta$, and its final value will be fixed in Proposition~\ref{prop:nonconcentration}. 
\begin{definition}\label{defn:nhood}
  Let $\beta > 0$.   Suppose that $S$ is a hypersurface in the unit ball of $\mathbf{R}^8$ given as that graph of a function $f$ over $\delta^{1/3}H$ with $|f|\leq C |\delta| r^{-2+\kappa}$.  Given $d > 0$, the $d$-neighborhood $N_{\beta, d}(S)$ is defined as follows.
  \begin{itemize}
  \item[(a)] If $d \geq \beta|\delta|$ then $N_{\beta, d}(S)$ is the region bounded between the surfaces $\pm (\beta^{-2} d)^{1/3}H$. 
  \item[(b)] If $d < \beta |\delta|$, then $N_{\beta, d}(S)$ is the region bounded between the graphs of
    \[ \pm \min\{ (\beta |\delta| + d) r^{-2}, d r^{-2.1}\} \]
    over $S$. 
 \end{itemize}
\end{definition}
This definition is motivated by the dichotomy between surfaces that are very close to $S$, which we view as graphs over $S$, and surfaces that are relatively far from $S$, at which scale we do not distinguish between $S$ and the cone $C$. The reason for the somewhat awkward expression in (b) is that when we construct barrier surfaces with negative mean curvature as graphs over $\Lambda T_\delta$ in Proposition~\ref{prop:linearbarrier}, then it helps to allow slightly faster blowup as $r\to 0$ than $r^{-2}$ (see Step 1 in the proof). 
Note that when viewing $S$ as a graph over $\delta^{1/3}H$ the value of $\delta$ is not uniquely defined, but in our applications $S$ will always be $V_\delta, W_\delta$ or $\Lambda T_\delta$, which in the $y$-slice we view as graphs over $\delta^{1/3}y H$. We have the following. 

\begin{lemma}\label{lem:Ndprop}
If $\beta$ is sufficiently small, then in the setting of the previous definition we have $N_{\beta, d_1}(S) \subset N_{\beta, d_2}(S)$ whenever $d_1 < d_2$, as long as $\delta$ is sufficiently small. In addition $\cap_{d > 0} N_{\beta, d}(S) = S$. 
\end{lemma}
\begin{proof}
  To see that $N_{\beta, d_1}(S) \subset N_{\beta, d_2}(S)$, it is enough to show that the region between the graphs of $\pm \min\{ 2\beta|\delta| r^{-2}, \beta|\delta| r^{-2.1}\}$ over $S$ is contained between the surfaces $\pm (\beta^{-2} \beta |\delta|)^{1/3}H$.

  To see this, note that by Lemma~\ref{lem:Phidefn1} the region between $\pm (\beta^{-1}|\delta|)^{1/3}H$ contains the region between the graphs of
  \[ \pm c_0 \min\{ (\beta^{-1}|\delta| - \delta) r^{-2}, r\} \]
  over $\delta^{1/3}H$. On the other hand $S$ is the graph of an $O(|\delta| r^{-2+\kappa})$ function over $\delta^{1/3}H$, and so the region between the graphs of $\pm \min\{ 2\beta|\delta| r^{-2}, \beta|\delta| r^{-2.1}\}$ over $S$ is contained between the graphs of
  \[ \pm (2\beta |\delta| r^{-2} + C |\delta| r^{-2+\kappa}) \]
  over $\delta^{1/3} H$. We therefore just need to show that if $\beta, \delta$ are sufficiently small, then on the unit ball, over $\delta^{1/3}H$, we have
  \[ c_0 \min\{ \frac{1}{2} \beta^{-1}|\delta| r^{-2}, r\} \geq 2\beta |\delta| r^{-2} + C|\delta| r^{-2+\kappa} . \]
  This follows, since on $\delta^{1/3} H$, in the unit ball, we have $C^{-1}|\delta|^{1/3} < r  < 1$ for a uniform $C$. In particular we have $S \subset N_{\beta, d}(S)$ for all $d > 0$, and it follows from this that $\cap_{d > 0} N_{\beta, d}(S) = S$. 
\end{proof}

We now define neighborhoods of $\Lambda T_\delta, V_\delta, W_\delta$, and the distance $D$. Note that the rescalings $\Lambda T_\delta$ are defined in the annulus $B_1\setminus B_{1/2}$ for $|\ln\Lambda| < |\delta|^{-\kappa}$ once $\delta$ is small enough. 
\begin{definition}\label{defn:D}
  We define $N_{\beta, d}(\Lambda T_\delta)$ to be the union of the neighborhoods $N_{\beta, d}(\Lambda T_\delta \cap \{y=y_0, r < 1\})$ for $\frac{1}{2} < y_0 < 1$.
  The neighborhoods of $V_\delta, W_\delta$ are defined in the same way.
  In order to define the distance over a set $U\subset B_1\setminus B_{1/2}$,  let $S' = V_\delta, W_\delta$ or $S'=\Lambda T_\delta$. We define the distance $D_{\beta, S'}(S; U)$ to be the infimum of all $d$ for which $S\cap U\subset N_{\beta, d}(S')$.
\end{definition}

  In this definition we used the annulus $B_1\setminus B_{1/2}$ for convenience, but the definition clearly extends to larger annuli, we just need to take $\delta$ smaller to ensure that $\Lambda T_\delta$ is still defined. From now on we will omit the $\beta$ from the notation, but note that the distance depends on the choice of $\beta$. The final choice will be determined in Proposition~\ref{prop:nonconcentration} below.

We will frequently use the following, analogous to \cite[Lemma 1.4]{Simon94}. Recall that our minimal surfaces $M$ are in a given multiplicity one class $\mathcal{M}$ as in \cite{Simon94}. 
\begin{lemma}\label{lem:smalldgraph}
  There is a $C = C(\beta) > 0$ such that for any $c > 0$ there are $d_0, \delta_0 > 0$ with the following property. Suppose that on the set $U\subset B_1\setminus B_{1/2}$ we have $D_{V_\delta}(M;U) < d_0$ for some $|\delta| < \delta_0$, and $M\in \mathcal{M}$. Then on the subset $U\cap \{r > c\}$ the surface $M$ can be written as the graph of $u$ over $V_\delta$, with $\sup_{\{r > c\}} |r^{2.1}u| \leq CD_{V_\delta}(M; U)$. The same result also holds for $\Lambda T_{\delta}$ with $|\Lambda| < |\delta|^{-\kappa}$. 
\end{lemma}
\begin{proof}
  Suppose that given $c > 0$ we have a sequence $M_i$, and $d_i, \delta_i \to 0$ such that $D_{V_{\delta_i}}(M_i;U) < d_i$. We show that the required conclusion holds for sufficiently large $i$. It follows that on $U$ the $M_i$ converge to $C\times \mathbf{R}$ in the sense of currents, and in particular by Allard's regularity theorem~\cite{All72} we can write $M_i$ as the graph of a function $v_i$ over $C\times\mathbf{R}$ on the set $U\cap \{r > c/2\}$ for large $i$. For sufficiently small $\delta_i$ the $M_i$ will then also be the graph of a function $u_i$ over $V_{\delta_i}$ on the region $U\cap \{r > c\}$.

  If $d_i \geq \beta|\delta_i|$ in terms of Definition~\ref{defn:nhood}, then $M_i$ lies between the surfaces $\pm (\beta^{-2}d_i)^{1/3}H$. Given $c > 0$, this means that on the region $\{r > c\}$, $M_i$ lies between the graphs of $\pm 2\beta^{-2} d_i r^{-2}$ over $C\times \mathbf{R}$, once $i$ is sufficiently large. At the same time, $V_{\delta_i}$ lies between the graphs of $\pm C|\delta_i| r^{-2}$ over $C\times \mathbf{R}$ on the same region, for a uniform $C$. It follows that
  \[ |u_i| \leq C |\delta_i| r^{-2} + 2\beta^{-2} d_i r^{-2} \leq (C \beta^{-1} + 2\beta^{-2}) d_i r^{-2}. \]
  If on the other hand $d_i \leq \beta|\delta_i|$, then by definition we must have $|u_i| \leq d_i r^{-2.1}$, and so the result follows. 
\end{proof}

The following is the key non-concentration result.
\begin{prop}\label{prop:nonconcentration}
 Suppose that $\beta > 0$ is sufficiently small.  There is a constant $C = C(\beta)$, such that given any $\gamma > 0$, if $r_0 < r_0(\beta, \gamma)$, $\delta < \delta(\beta, \gamma, r_0)$, $d < d(\beta, \gamma, r_0)$ then we have the following.

  Suppose that $M\in\mathcal{M}$ is a minimal surface with
  $D_{\Lambda T_\delta}(M; B_1\setminus B_{1/2}) = d$, where $|\ln\Lambda| < |\delta|^{-\kappa}$, that on the region $\{r > r_0\}$ we can write as the graph of a function $u$ over $\Lambda T_\delta$. Then we have
 \[ \label{eq:nonconc}
   D_{\Lambda T_\delta}(M; B_{4/5}\setminus B_{3/5}) \leq C \sup_{(B_1\setminus B_{1/2})\cap \{r > r_0\}} |r^2 u| + C \gamma d. \]
The same conclusion also holds if $M$ is only assumed to be minimal in the region $\{r < r_1\}$ for some $r_1 > r_0$, and with $V_\delta$ replacing $\Lambda T_\delta$. 
\end{prop}
The constant $\beta$ will be fixed after this result, and
the crucial point is that the constant $C$ does not depend on $\gamma$. While we stated the result for specific annuli $A_1 = B_1\setminus B_{1/2}$ and $A_2 = B_{4/5}\setminus B_{3/5}$, it is clear that the result can be applied to any annuli $A_2 \subset A_1$, with appropriate changes in the constants. 

To prove this result we first construct suitable barrier surfaces, in order to apply the maximum principle later. A preliminary step is the following.
\begin{lemma}\label{lem:Fadefn}
  For any $a > -2$ there is a constant $C_a > 0$ and a function $F_a$ on $H$ satisfying the following. Outside of a compact set $F_a = r^a$ and at the same time $L_H F_a > C_a^{-1} r^{a-2}$ on $H$. We extend $F_a$ to $\mathbf{R}^8\setminus \{0\}$ to be homogeneous with degree $a$. Then we have $L_{\lambda H} F_a > C_a^{-1} r^{a-2}$ and $|F_a| < C_a r^a$ on any rescaling of $H$, and $F_a = 1$ on a neighborhood of $C \cap \partial B_1(0)$. 
\end{lemma}
\begin{proof}
  Note first that on the cone $C$ we have $L_C r^a = c_a r^{a-2}$ for some $c_a > 0$ if $a > -2$. Using that $H$ is the graph of an $O(r^{-2})$ function over $C$, we have $L_H r^a = c_a r^{a-2} + O(r^{a-3})$ as $r\to\infty$. In particular $L_H r^{a} > \frac{c_a}{2} r^{a-2} - \eta_a$ for a compactly supported function $\eta_a \geq 0$. The operator $L_H : C^{2,\alpha}_\tau \to C^{0,\alpha}_{\tau-2}$ is invertible for $\tau\in (-3,-2)$, so we can find $f_a = O(r^{-5/2})$ such that $L_H f_a = \eta_a$. Finally let $\chi_R$ be a cutoff function such that $\chi_R(s) = 0$ for $s > 2R$, and $\chi_R(s)=1$ for $s < R$. Define
  \[ F_a = r^{a} + \chi_R(r) f_a. \]
  For large $R$, on the support of $\nabla \chi_R$ we have
$L_H F_a > L_H r^{a} - c r^{-9/2}$ for $c$ depending on $f_a$, but not on $R$. Away from the support of $\nabla\chi_R$ we have $L_H F_a = L_H (r^{a} + f_a)$, if $R$ is so large that $L_H f_a = 0$ wherever $\chi_R \not=1$. It follows that $L_H F_a > \frac{c_a}{2} r^{a-2}$ if $R$ is chosen sufficiently large, and by construction $F_a = r^{a}$ outside of the support of $\chi_R$. 
\end{proof}

\begin{prop}\label{prop:barrier1}
  Let $f: (a,b)\to\mathbf{R}$ be a $C^2$ function. There is a large $C_f > 0$ and small $r_0, \epsilon_0 > 0$ depending on an upper bound for the $C^2$-norm of $f$, with the following property. For $0 < \epsilon < \epsilon_0$ define the surface $X$ to be the graph of $-C_f\epsilon |f(y)|^3 F_0 - \epsilon F_1$ over $\epsilon^{1/3}f(y)^3H$ in the $\{y\}\times \mathbf{R}^8$ slice, for $r < r_0$ and $y\in (a,b)$. Here $F_0, F_1$ are the functions in Lemma~\ref{lem:Fadefn}.

  The surface $X$ has negative mean curvature, except at points $\{y=y_0, r=0\}$ where $f(y_0)=0$. At these points the tangent cone of $X$ is the graph of $-\epsilon r$ over $\mathbf{R}\times C$. In particular, if a minimal hypersurface lies on the negative side of $X$, then it cannot touch $X$ at an interior point. 
\end{prop}
\begin{proof}
  By translation we can assume that $0\in (a,b)$ and we study $X$ in the slice $\{y=0\}$. Consider the region $-R < y < R$, $R/2 < r < 2R$ for small $R$ and scale it up by a factor of $R^{-1}$. The scaled up surface $R^{-1}X$ is the graph of $-R^{-1}C_f\epsilon |f(Ry)|^3 F_0 - \epsilon F_1$ over $R^{-1}\epsilon^{1/3} f(Ry)^3 H$, in the $y$-slice. Let us write
  \[ E &= R^{-1}\epsilon^{1/3} f(Ry)^3, \\
    B &= -R^{-1}C_f\epsilon |f(Ry)|^3F_0 - \epsilon F_1, \]
so that we are taking the graph of $B(r,y)$ over $E(y)\cdot H$. 
After the scaling we have $y\in (-1,1)$ and $r\in (1/2,2)$. We estimate the mean curvature at $y=0, r=1$, in two separate cases. Let $C_0 > 0$ be large, to be chosen, depending on the $C^2$ bound for $f$.
  \begin{itemize}
  \item Suppose that $|f(0)| \geq C_0R$. In this case we will view the surface $R^{-1}X$ as a graph over $E(0)\cdot H$. Note that if $C_0$ is sufficiently large, then on our region we have $E(0)^{-1} E(y) - 1 = O(R |f(0)|^{-1}) = O(C_0^{-1})$. Therefore using Lemma~\ref{lem:Phidefn1} we can view $E(y)\cdot H$ as the graph of the function 
    \[ A = E(0) \Phi_{E(0)^{-1}E(y) - 1 } (E(0)^{-1}\cdot\, ) \]
    over $E(0)\cdot H$, where $\Phi_\epsilon : H \to \mathbf{R}$ is defined so that $(1+\epsilon)H$ is the graph of $\Phi_\epsilon$ over $H$. We have $\Phi_\epsilon(x) = \epsilon\Phi (x) + O(\epsilon^2 |x|^{-2})$, where $\Phi$ is the corresponding Jacobi field on $H$. We thus have
    \[ A = E(0) \Big(E(0)^{-1}E(y) - 1\Big) \Phi( E(0)^{-1}\cdot\,) + O\Big( E(0)^3 (E(0)^{-1}E(y) - 1)^2 \Big). \]
    This means that $R^{-1}X$ in our region, in the $y$-slice, is the graph of $B$ over the graph of $A$ over the surface $E(0)\cdot H$. We have the following estimates for $A, B$:
    \[ A, \nabla^i A &\lesssim |E(0)|^3 R|f(0)|^{-1} = R^{-2} \epsilon |f(0)|^8 \\
      B, \nabla^i B &\lesssim R^{-1}C_f \epsilon |f(0)|^3+ \epsilon. \]
    Here we use the notation $a\lesssim b$ for $|a| \leq Cb$ with a constant $C$ that depends on the $C^2$ norm of $f$. 
    Note also that we must have $R^{-1}\epsilon^{1/3}|f(0)|^3 \lesssim 1$, otherwise there is no point with $y=0, r=1$ in the surface $R^{-1}X$. It follows that $A\lesssim \epsilon^{1/3}$, and $B \lesssim \epsilon^{2/3}C_f$, which we can make small by choosing $\epsilon_0$ small depending on $C_f$.
    
    According to Lemma~\ref{lem:gg} below, up to an error of order $AB$, we can view $R^{-1}X$ as the graph of $A+B$ over $E(0)\cdot H \times \mathbf{R}$. In particular we can estimate the mean curvature as
    \[  m = L_{E(0)\cdot H\times\mathbf{R}}(A+B) + O(A^2 + B^2). \]
    We work at $y=0$, so we have
    \[ L_{E(0)\cdot H \times \mathbf{R}} A &= E''(0) \Phi(E(0)^{-1}\cdot)  + O( E(0)^3 R^2 |f(0)|^{-2})\\
      &\lesssim E(0)^2 E''(0)  + E(0)^3 R^2 |f(0)|^{-2}\\
      &\lesssim R^{-1}\epsilon |f(0)|^7,\]
    using that
    \[ E(0) &\lesssim R^{-1}\epsilon^{1/3} |f(0)|^3 \\
        E''(0) &\lesssim R \epsilon^{1/3} |(f^3)''(0)| \lesssim R \epsilon^{1/3} |f(0)|. \]
      The useful negative terms will come from $LB$, in particular the terms $LF_0$ and $LF_1$. At $y=0$ we have
      \[ L_{E(0)\cdot H\times \mathbf{R}} B &= -R^{-1}C_f \epsilon |f(0)|^3 L_{E(0)\cdot H} F_0
        - R C_f \epsilon (|f|^3)''(0)F_0 - \epsilon L_{E(0)\cdot H} F_1 \\
        &\leq - R^{-1} C_f\epsilon |f(0)|^3 c_1 r^{-2} - \epsilon c_2 r^{-1} + O(R C_ f\epsilon |f(0)|), \\
        &\leq - \frac{1}{2} R^{-1} C_f \epsilon |f(0)|^3 c_1 - \frac{1}{2}\epsilon  c_2 + O(R C_ f\epsilon |f(0)|), \]
      where we used the estimates for $L_H F_a$ in Lemma~\ref{lem:Fadefn}.
      
      Adding the quadratic terms as well, we then have
      \[ \label{eq:mleq1} m \leq -\frac{1}{2} c_1 R^{-1}C_f \epsilon |f(0)|^3 - \frac{1}{2}c_2\epsilon + K,\]     where
    \[ K &\lesssim R^{-1}\epsilon |f(0)|^7 + RC_f \epsilon |f(0)| + R^{-4}\epsilon^2 |f(0)|^{16} + R^{-2} C_f^2 \epsilon^2 |f(0)|^6 + \epsilon^2\\
      &= R^{-1}C_f\epsilon |f(0)|^3 \Big( C_f^{-1} |f(0)|^4 + R^2|f(0)|^{-2} \\
      &\quad + R^{-3}C_f^{-1}\epsilon |f(0)|^{13} + R^{-1}C_f \epsilon |f(0)|^3\Big) + \epsilon^2\\
      &\lesssim R^{-1}C_f\epsilon |f(0)|^3 \Big( C_f^{-1} + C_0^{-2} + C_f \epsilon^{2/3}\Big) + \epsilon^2,\]
    using that $R |f(0)|^{-1} \leq C_0^{-1}$ and $R^{-1}\epsilon^{1/3}|f(0)|^3\lesssim 1$. 
    If we first choose $C_f$ sufficiently large, then for $\epsilon$ sufficiently small and $C_0$ sufficiently large we have $m < 0$ by \eqref{eq:mleq1}. At this point $C_f, C_0$ are fixed. 
  \item Suppose that $|f(0)| \leq C_0R$. Then we have $|E(0)|\lesssim R^2\epsilon^{1/3} \ll 1$ and the derivatives of $E$ are of order $R^2\epsilon^{1/3}$, so in our region $|E(y)| \lesssim R^2\epsilon^{1/3}$. We can view $E(y)\cdot H$ as the graph over $C\times\mathbf{R}$ of the function
    \[ A = E(y)\Psi(E(y)^{-1}\cdot), \]
    where $H$ is the graph of $\Psi$ over $C$ outside of a compact set. Since $\Psi(x) = r^{-2} + O(r^{-3})$, we have
    \[ A = E(y)^3 r^{-2} + O( E(y)^4) \lesssim \epsilon. \]
    The surface $X$ is the graph of
    \[ B= -R^{-1}C_f \epsilon |f(Ry)|^3F_0 - \epsilon F_1 \lesssim \epsilon \]
    over $E(y)\cdot H$ and so, similarly to the previous case, the mean curvature satisfies
    \[ m  = L_{C\times\mathbf{R}}(A+B) + O(A^2 + B^2). \]
    Using that $r^{-2}$ is a Jacobi field on $C$, at $y=0$ we have
    \[ L_{C\times\mathbf{R}}A = (E^3)''(0)r^{-2} + O(|E(0)|^4) \lesssim R^6\epsilon. \]
    We also have
    \[ L_{C\times \mathbf{R}} B = - R C_f \epsilon (|f|^3)''(0)F_0 - R^{-1}C_f \epsilon |f(0)|^3 L_{C} F_0 - \epsilon L_C F_1, \]
    and $L_C F_0 > c_1r^{-2}$, $L_C F_1 > c_1 r^{-1}$ for some $c_1 > 0$. Since  $r\in (1/2,2)$, we get
    \[ L_{C\times\mathbf{R}}B \leq -\frac{1}{2} c_1\epsilon + R\epsilon C_f C, \]
    for $C$ depending on the $C^2$ norm of $f$.
    We already fixed $C_0, C_f$, and $R$ is small, so $A^2 + B^2 \lesssim \epsilon^2$. It follows that 
    \[ m < -\frac{1}{2} c_1\epsilon + C(R^6\epsilon + R\epsilon + \epsilon^2), \]
    for $C$ depending on $C_0, C_f$. If $R, \epsilon$ are small, then we have $m<0$. 
  \end{itemize}
  This shows that the mean curvature of $X$ is negative at all points where $0 < r < r_0$ and it remains to deal with the points $(0,y_0)\in X$. This point can only be in $X$ if $f(y_0)=0$, and we assume $y_0=0$. Let $\lambda > 0$ be small and rescale $X$ by $\lambda^{-1}$. The new surface $\lambda^{-1}X$ is the graph of
  \[ -\lambda^{-1}C_f \epsilon |f(\lambda y)|^3 F_0 -\epsilon F_1 \]
  over $\lambda^{-1} \epsilon^{1/3} f(\lambda y)^3 H$ in the $y$-slice. In the limit $\lambda\to 0$ we get (using that $f$ is a $C^2$ function and $f(0)=0$) the graph of $-\epsilon r$ over $C\times \mathbf{R}$.

  Suppose that a minimal surface $M$ lies on the negative side of $X$, and touches $X$ at a point $x_0$. If $x_0$ is a smooth point of $X$, then this is a contradiction since $X$ has negative mean curvature. Note that in this case $M$ is necessarily smooth at $x_0$ since its tangent cone must be a hyperplane. If $r(x_0)=0$, then the tangent cone of $M$ at $x_0$ would lie of the negative side of the graph of $-\epsilon r$ over the minimal cone $C\times \mathbf{R}$, which is also a contradiction.
\end{proof}

We used the following lemma in the calculation above whose prove we omit. 
\begin{lemma}\label{lem:gg}
  Let $S$ be a surface in $\mathbf{R}^n$, and let $f,g$ be two (small) functions defined on $\mathbf{R}^n$. If the second fundamental form of $S$ and its derivatives are bounded, then the graph of $g$ over the graph of $f$ over $S$ can be written as the graph of $h$ over $S$, where
  \[ h = f + g + O(|f| | \nabla g| + |g| |\nabla f|^2 + |g| |\nabla f| |\nabla g|) .\]
  In particular if $|f|_{C^k} < a$ and $|g|_{C^k} < b$ are sufficiently small (depending on the geometry of $S$), then we have
  \[ |h - (f+g)|_{C^{k-1}} < Cab, \]
  for $C$ depending on the bounds for $S$.  
\end{lemma}

The following result shows that geometrically the surface constructed in Proposition~\ref{prop:barrier1} can be thought of as the surface given in each $y$-slice by $\epsilon^{1/3} f(y)^3 H$, at least in the region $r < r_0$ for sufficiently small $r_0$. 
\begin{prop}\label{prop:Xc}
  In the setting of  Proposition~\ref{prop:barrier1}, given any small $c > 0$, there is an $r_0' < r_0$ depending on the $C^2$-norm of $f$ and on $c$, such that on the region $r < r_0'$ the surface $X$ lies between the surfaces given by $(\epsilon f(y)^9 \pm c\epsilon)^{1/3} H$ in each $y$-slice.
\end{prop}
\begin{proof}
  By Lemma~\ref{lem:Phidefn1} the surface $(\epsilon f(y)^9 + c\epsilon)^{1/3} H$ lies on the positive side of the graph of
\[ c_0 \min\{ c\epsilon r^{-2}, r\} \]
over $\epsilon^{1/3} f(y)^3 H$. It remains to check that on the region $r < r_0'$
\[ \Big|\,C\epsilon |f(y)|^3 F_0 + \epsilon F_1\,\Big| \leq  c_0 \min\{ \epsilon c r^{-2}, r\} \]
if $r_0'$ is chosen small enough. 
This follows using that  
on the surface $\epsilon^{1/3} f(y)^3 H$ we have $r \gtrsim \epsilon^{1/3}|f(y)|^3$, and $F_0 \lesssim r_0'^2 r^{-2}$ if $r < r_0'$. The argument to see that $X$ lies on the positive side of $(\epsilon f(y)^9 - c\epsilon)^{1/3}H$ is similar, applying Lemma~\ref{lem:Phidefn1} to the opposite orientation.
\end{proof}

Choosing suitable functions $f$ we can use this to construct surfaces that are barriers at ``distance'' $\epsilon$ around the $\Lambda T_\delta$, at least when $\epsilon$ is not much smaller than $\delta$ (as long as we restrict to $r < r_0'$ for smaller $r_0'$, we can allow $\epsilon |\delta|^{-1}$ to be smaller). When $\epsilon$ is much smaller than $\delta$ then the construction of $X$ above is too coarse, and we will use graphs over $\Lambda T_\delta$ as barriers, based on the following.
\begin{prop}\label{prop:linearbarrier}
Suppose that $\beta > 0$ is sufficiently small.  Let $g: (a,b) \to (0,\infty)$ be a $C^2$ function, where $(a,b)\subset (1/2,1)$. There are $C > 0$ independent of $g$, and $r_0, \delta_0, \epsilon_0 > 0$ depending on the $C^2$-norms of $g, g^{-1}$, such that for $|\delta| < \delta_0$ and $|\log \Lambda| < |\delta|^{-\kappa}$ we have the following. There is a function $G$ on $\Lambda T_\delta$ in the region $\{r < r_0\}$, $y\in(a,b)$, satisfying
  \begin{enumerate}
    \item $C^{-1} g(y)r^{-2.1} < G < C g(y) r^{-2.1}$, 
    \item The mean curvature of the graph of $\epsilon G$ over $\Lambda T_\delta$ is negative for $\epsilon < \epsilon_0$, on the region defined by the conditions $\{r < r_0, y\in (a,b)\}$ and $\epsilon G < \beta|\delta|r^{-2}$. 
   \end{enumerate}
 \end{prop}
 \begin{proof}
   We break the construction up into several steps. \\ 

  \noindent{\em Step 1.}
  First we argue similarly to Lemma~\ref{lem:Fadefn} to construct a function $F_{-2.1} > 0$ on $H$ satisfying the conditions that $F_{-2.1} = r^{-2.1}$ outside of a ball and $L_H F_{-2.1} < -c_1 r^{-4.1}$ for a $c_1 > 0$. For this note that on the cone $C$ we have $L_C r^{-2.1} = -c_1' r^{-4.1}$ for some $c_1' >0$. It follows that we must have $F_{-2.1} > 0$ on all of $H$, since if $F_{-2.1}$ were negative at a point, then for a suitable $\lambda \geq 0$ we would have $F_{-2.1} \geq -\lambda\Phi$, with equality at some point ($\Phi$ being the positive Jacobi field on $H$). At the contact point we would have $L_H F_{-2.1} \geq 0$, which is a contradiction.

  Using $F_{-2.1}$  and the construction of $\Sigma_\delta$ we can define functions $\tilde{F}_{-2.1}$ on $\Sigma_\delta$ satisfying $C^{-1} r^{-2.1} < \tilde{F}_{-2.1} < C r^{-2.1}$ and $L_{\Sigma_\delta} \tilde{F}_{-2.1} < -c_1 r^{-4.1}$ on the region $\{r < r_0\}$ for sufficiently small $r_0$ (independent of $\delta$). We then extend $\tilde{F}_{-2.1}$ as a degree one function on the cone $V_\delta = C(\Sigma_\delta)$. This will satisfy the same estimates as $\tilde{F}_{-2.1}$ on $\Sigma_\delta$ on the region $y\in (1/2, 1)$. \\

  \noindent{\em Step 2.}
  The surface $\Lambda T_\delta$ is the graph of a function $F$ over $V_\delta$ where $\Vert F\Vert_{C^{2,\alpha}_{1,\tau}} \leq C |\delta|^{-\kappa} |h(\delta)|$. It follows from Lemma~\ref{prop:normcomp1} that if $\kappa$ is sufficiently small and $\tau$ is close to $-2$, then $\Vert f\Vert_{C^{2,\alpha}_{1,1}} \leq |\delta|^\epsilon$ for some $\epsilon > 0$ and sufficiently small $\delta$. This allows us to estimate the difference between the linear operators $L_{V_\delta}$ and $L_{\Lambda T_\delta}$ and we get
  \[ L_{\Lambda T_\delta} \tilde{F}_{-2,1} < -c_1 r^{-4.1} +  C |\delta|^\epsilon r^{-4.1} \leq -\frac{c_1}{2} r^{-4.1}\]
  on the region $\{ r < r_0\}$, after decreasing $r_0$ if necessary. \\

   \noindent{\em Step 3.}
   Let us now consider the function
   \[ G = g(y) \tilde{F}_{-2.1}, \]
   defined on $\Lambda T_\delta$. 
   Note that on $\Lambda T_\delta$ we have $|\nabla y| \leq 1$ and $|\nabla^2 y| \leq C r^{-1}$ for some $C > 0$ (for the latter we use that the regularity scale of $\Lambda T_\delta$ is comparable to $r$ at each point). We have
   \[ L_{\Lambda T_\delta} G &\leq - \frac{c_1}{2} g(y) r^{-4.1} + 2\nabla g\cdot \nabla \tilde{F}_{-2.1} + \Delta g \tilde{F}_{-2.1} \\
     &\leq - \frac{c_1}{2} g(y) r^{-4.1}  + C \Vert g\Vert_{C^2} r^{-3.1} \\
     &\leq -\frac{c_1}{4} g(y) r^{-4.1},
   \]
   if $\{r < r_0\}$ for $r_0$ depending on $\Vert g\Vert_{C^2}$ and a lower bound for $g$. 

   We also have the estimates
   \[ C^{-1} g r^{-2.1} \leq G &\leq C g r^{-2.1}, \\
     |\nabla G| &\leq C g r^{-3.1} + C \Vert g\Vert_{C^2} r^{-2.1} \leq 2C g r^{-3.1}, \\
     |\nabla^2 G| &\leq C g r^{-4.1} + C \Vert g\Vert_{C^2} r^{-3.1} \leq 2C g r^{-4.1}, \]
   if $r < r_0$ for small $r_0$ depending on $\Vert g\Vert_{C^2}$ and a lower bound for $g$. 
   Consider the graph of $\epsilon G$ over $\Lambda T_\delta$. Let us estimate the mean curvature at a point where $\epsilon G < \beta|\delta| r^{-2}$, and $r < r_0$. At such a point we have $\epsilon g r^{-2.1} < C|\beta|\delta r^{-2}$, and so it follows that
   \[ r^{-1} |\epsilon G| + |\nabla \epsilon G| + r|\nabla^2 \epsilon G| \leq C \epsilon gr^{-3.1} \leq C \beta|\delta| r^{-3} \leq C \beta, \]
   for a larger $C$, using that $r \geq c_2 |\delta|^{1/3}$ for a uniform $c_2 > 0$. We find that if $\beta$ is chosen sufficiently small, then the mean curvature of the graph of $\epsilon G$ satisfies
   \[ m( \epsilon G) &\leq L_{\Lambda T_\delta} (\epsilon G) + C \beta \epsilon g r^{-4.1} \\
     &\leq -\frac{c_1}{4} \epsilon g r^{-4.1} + C \beta \epsilon g r^{-4.1} \\
     &\leq - \frac{c_1}{8} \epsilon g r^{-4.1}, \]
   as required. 
 \end{proof}

 We can now give the proof of Proposition~\ref{prop:nonconcentration} using a barrier argument. When we are considering a surface $M$ at a distance from $T_\delta$ that is large compared to $\gamma\beta\delta$, then we can use the barrier surfaces constructed in Proposition~\ref{prop:barrier1}, since at such scales $T_\delta$ can be well approximated by the surface with cross sections $y\delta^{1/3} H$ on a region $\{r < r_0\}$ with $r_0$ depending on $\gamma, \beta$. At scales smaller than this we do not have such a precise picture of $T_\delta$, however if $M$ is closer to $T_\delta$ than this, then it is actually a graph, and we can instead use the graphical barriers constructed in Proposition~\ref{prop:linearbarrier}.
 
\begin{proof}[Proof of Proposition~\ref{prop:nonconcentration}]
  Let us fix $\gamma > 0$. We write $d = D_{\Lambda T_\delta}(M; B_1\setminus B_{1/2})$, and for any $r_0 > 0$ we let $D(r_0) = \sup_{r > r_0} |r^2 u|$, where $M$ is the graph of $u$ over $\Lambda T_\delta$ on the region $\{ r > r_0\}$. Note that for any given $r_0 > 0$, once $d, \delta$ are sufficiently small, $M$ is a graph on this region by Lemma~\ref{lem:smalldgraph}.
Define 
  \[ \overline{d} = \max\{ \beta |\delta|, d + \gamma^{-1}D(r_0)\}. \]
 Using the definition of the distance, if $\delta$ is sufficiently small (depending on $r_0$), we have
  \[ D_{\Lambda T_\delta}(M; (B_1\setminus B_{1/2})\cap \{r \geq r_0\}) \leq D(r_0). \]
  It follows that to prove the required estimate \eqref{eq:nonconc} we only need to estimate the distance on the region $\{r < r_0\}$. We can also restrict ourselves to $y\geq 0$, the corresponding estimates for $y\leq 0$ are completely analogous. 

  We will first use the barrier surfaces constructed in Proposition~\ref{prop:barrier1} to show that if $r_0$ is sufficiently small (depending on $\gamma$ and $\beta$), then 
  \[ D_{\Lambda T_\delta}(M; B_{9/10} \setminus B_{11/20}) \leq C' \gamma \overline{d} \]
  for a constant $C'$ depending on $\beta$, but not on $\gamma$.

  We apply Proposition~\ref{prop:barrier1} to the function
  \[ f(y) = (\delta/\epsilon)^{1/9}y^{1/3} + (y-1/2)^{-1} + (1-y)^{-1}. \]
  We will assume that $\epsilon > \gamma \beta|\delta|$, and that $y\in (1/2 + \gamma, 1-\gamma)$. Under these assumptions we have a $\gamma, \beta$-dependent bound for the $C^2$-norm of $f$, and so according to Proposition~\ref{prop:barrier1} there are $\epsilon_0, r_0 > 0$ depending on $\gamma, \beta$, such that we can build a hypersurface $X_\epsilon$ for $\epsilon \leq \epsilon_0$ with negative mean curvature, modeled on $\epsilon^{1/3}f(y)^3 H$, on the region $U$ defined by $r < r_0$ and $y\in (1/2 + \gamma, 1-\gamma)$. Given any $c > 0$, and replacing $r_0$ by a smaller constant, we can assume that, over $U$, $X_\epsilon$ lies between the surfaces $(\epsilon f(y)^9 \pm c\epsilon)^{1/3} H$ by Proposition~\ref{prop:Xc}. Note that 
  \[ \epsilon f(y)^9 = \Big[ \delta^{1/9} y^{1/3} + \epsilon^{1/9} (y-1/2)^{-1} + \epsilon^{1/9} (1-y)^{-1}\Big]^9. \]
  In particular there is a constant $C$, independent of $\gamma$, such that for sufficiently small $\delta, \epsilon$ the function $f$ satisfies the inequalities
  \[ \epsilon f(y)^9 &\geq \delta y^3 + C^{-1} \epsilon, \text{ for }y\in (1/2+\gamma, 1-\gamma) \\
    \epsilon f(y)^9 &\geq \delta y^3 + C^{-1} \epsilon \gamma^{-1}, \text{ for }y=1/2+\gamma \text{ or } y=1-\gamma, \\
    \epsilon f(y)^9 &\leq \delta y^3 + C \epsilon, \text{ for }y\in (11/20, 9/10). 
  \]
  Choosing a suitably small $c$, and then letting $r_0$ be small, we can then assume that on the region $U$ the surface $X_\epsilon$ lies on the positive side of $(\delta y^3 + C^{-1}\epsilon /2)^{1/3}H$, while in the slices $y=1/2+\gamma$ and $y=1-\gamma$ the surface $X_\epsilon$ lies on the positive side of $(\delta y^3 + C^{-1}\gamma^{-1}\epsilon /2)^{1/3}H$.

\bigskip
\noindent {\bf Claim:} There is a $C'> 0$ depending on $\beta$ (independent of $\gamma$) such that if $r_0$ is sufficiently small (depending on $\beta,\gamma$) we have the following. If $d, |\delta|$ are sufficiently small (depending on $\beta, \gamma, r_0$), then along the boundary of $U$ the minimal surface $M$ lies on the negative side of the surface $X_{C'(\gamma \overline{d} + D(r_0))}$, while on all of $U$ the surface $M$ lies on the negative side of $X_{\epsilon_0}$.

\bigskip
\noindent {\em Proof of Claim:} Since $\overline{d} \geq \beta|\delta|$ and $D_{\Lambda T_\delta}(M; B_1\setminus B_{1/2}) \leq d$, by definition we have that $M$ is on the negative side of $(\beta^{-2} \overline{d})^{1/3}H$. At the same time $X_{\epsilon_0}$ lies on the positive side of $(\delta y^3 + C^{-1}\epsilon_0 / 2)^{1/3}H$ over $U$, which for sufficiently small $\delta$ (depending on $\epsilon_0$, i.e. on $\gamma, \beta$)  is on the positive side of $(C^{-1}\epsilon_0 /4)^{1/3}H$. If $d, \delta$ is sufficiently small, depending on $\beta, \epsilon_0$, then  $\beta^{-2}\overline{d} \leq C^{-1}\epsilon_0/4$, and this implies that $M$ is on the negative side of $X_{\epsilon_0}$.

On the boundary pieces $y=1/2 + \gamma$ and $y=1 - \gamma$ of $U$ we have that $M$ is on the negative side of $(\beta^{-2} \overline{d})^{1/3} H$ as above. At the same time along these boundary pieces $X_{C'(\gamma\overline{d} + D(r_0))}$ is on the positive side of $(\delta y^3 + C^{-1}C'\overline{d}/2)^{1/3}H$. Since $\overline{d} \geq \beta |\delta|$, it follows that for sufficiently large $C'$ (depending on $\beta$), $M$ is on the negative side of $X_{C'(\gamma\overline{d} + D(r_0))}$.

On the boundary piece $\{r=r_0\}$ of $U$, we have that $M$ is the graph of a function $u$ over $\Lambda T_\delta$ with $|u|\leq D(r_0) r^{-2}$, and so $M$ is the graph of $v$ over $(y^3\delta)^{1/3}H$ with
\[ \label{eq:vbound10} |v| \leq D(r_0) r^{-2} + C |\delta| r^{-2+\kappa}, \text{ along }r=r_0.\]
We used here that $\Lambda T_\delta$ is an $O(|\delta| r^{-2+\kappa})$ graph over $(y^3\delta)^{1/3}H$ by Proposition~\ref{prop:Tdelta1}. At the same time $X_{C'(\gamma \overline{d} + D(r_0))}$ lies on the positive side of $(\delta y^3 + C^{-1}C' (\gamma \overline{d} + D(r_0))/2)^{1/3}H$, which by Lemma~\ref{lem:Phidefn1} is on the positive side of the graph of
\[ c_0 \min\{ \frac{1}{2} C^{-1}C' (\gamma \overline{d} + D(r_0)) r^{-2}, r\} \]
over $(y^3\delta)^{1/3}H$. Recall that we are only interested in the set where $\{r = r_0\}$. For fixed $C', r_0 > 0$, once $d, \delta$ are sufficiently small (depending on $\gamma$), the minimum is achieved by $\frac{1}{2} c_0 C^{-1}C' (\gamma \overline{d} + D(r_0)) r^{-2}$. Comparing this with \eqref{eq:vbound10}, we just need to pick $C'$ sufficiently large so that
\[ D(r_0) r_0^{-2} + C|\delta|r_0^{-2+\kappa}  < \frac{1}{2} c_0 C^{-1} C' (\gamma \overline{d} +  D(r_0)) r_0^{-2}. \]
Using that $\overline{d} \geq \beta |\delta|$, this inequality holds for sufficiently large $C'$ independent of  $\gamma$, once  $r_0$ is chosen sufficiently small (depending on $\gamma, \beta$). This completes the proof of the Claim.

\bigskip
Interpolating between $\epsilon = \epsilon_0$ and $\epsilon = C'(\gamma\overline{d} + D(r_0))$ we find that $M$ lies on the negative side of all the corresponding $X_\epsilon$, since otherwise there would be some value $\epsilon$, such that $M$ touches $X_\epsilon$ at an interior point, which contradicts Proposition~\ref{prop:barrier1}. We now use that on the region $y\in (11/20, 9/10)$, $\{r\leq r_0\}$ the surface $X_{C'(\gamma\overline{d} + D(r_0))}$, and so also $M$, lies on the negative side of
\[ (\delta y^3 + CC'(\gamma \overline{d} + D(r_0))/2)^{1/3}H. \]
We can repeat the same argument, reversing orientations, to find that on the same region $M$ is also on the positive side of
\[ (\delta y^3 - CC'(\gamma \overline{d} + D(r_0))/2)^{1/3}H. \]
Let us write $A = CC'(\gamma\overline{d} + D(r_0))/2$, so $M$ is between the surfaces $(\delta y^3 \pm A)^{1/3}H$. 
We claim that this implies that
\[ \label{eq:Dbound10} D_{\Lambda T_\delta}\Big(M, (B_{9/10}\setminus B_{11/20})\cap \{r < r_0\}\Big) < 4A, \]
if $\beta$ is sufficiently small (independent of $\gamma$) and $r_0$ is sufficiently small (depending on $\beta,\gamma$). To see this, by definition we need to look at two cases:
\begin{itemize}
\item If $4A \geq \beta|\delta|$, then \eqref{eq:Dbound10} follows if on the relevant region $M$ is between the surfaces $(\pm \beta^{-2} 4A)^{1/3}H$. This follows since in this case $|\delta y^3 \pm A| \leq \beta^{-2}4 A$ for sufficiently small $\beta$. 
\item If $4A < \beta|\delta|$, and $\beta$ is sufficiently small, then the fact that $M$ lies between $(\delta y^3 \pm A)^{1/3}H$ implies that $M$ is the graph of a function $u$ over $\Lambda T_\delta$ with
  \[ |u| \leq 2A r^{-2} + C|\delta| r^{-2+\kappa}. \]
  By definition, for the bound \eqref{eq:Dbound10} we need to ensure that
  \[ \label{eq:b11} 2Ar^{-2} + C|\delta| r^{-2+\kappa} < \min\{ (\beta |\delta| + 4A)r^{-2}, 4A r^{-2.1}\}. \]
  Note that $A \geq \gamma \overline{d} \geq \gamma \beta |\delta|$, so
  \[ \min\{(\beta |\delta| + 4A)r^{-2}, 4A r^{-2.1}\} \geq 2A r^{-2} + 2\gamma\beta |\delta| r^{-2}. \]
  On the region $\{r < r_0\}$ for sufficiently small $r_0$ (depending on $\gamma,\beta$) we then have \eqref{eq:b11} as required. 
\end{itemize}

If $\overline{d} = d + \gamma^{-1}D(r_0)$, then \eqref{eq:Dbound10} implies the required estimate \eqref{eq:nonconc}. It remains to deal with the case $\overline{d} > d + \gamma^{-1} D(r_0)$, i.e. $\overline{d} = \beta |\delta|$. It follows that $\overline{d} + \gamma^{-1}D(r_0) < 2\beta|\delta|$, and so the constant $A$ above satisfies
\[ A \leq CC' \gamma \beta|\delta|. \]
We still have that on the region $B_{9/10}\setminus B_{11/20}$, where $\{r < r_0\}$, $M$ is between the surfaces $(\delta y^3 \pm A)^{1/3}H$. Without loss of generality we can assume that $\gamma < (4CC'C_{5.1})^{-1}$, where $C_{5.1}$ is the constant appearing in Lemma~\ref{lem:Phidefn1}, since $C, C', C_{5.1}$ are independent of $\gamma$. It follows from this that $M$ is between $(\delta y^3 \pm \frac{1}{4}C_{5.1}^{-1}\beta |\delta|)^{1/3}H$ on the region $U'$ defined by $\{r < r_0\}\cap (B_{9/10}\setminus B_{11/20})$. If $\beta$ is sufficiently small, and we replace $r_0$ with a smaller constant (depending on $\beta$), then it follows that $M$ is the graph of a function $u$ over $\Lambda T_\delta$ on $U'$, with
\[ \label{eq:ubound12} |u| \leq \frac{1}{2} \beta |\delta| r^{-2}. \]
We now use the graphical barrier surfaces constructed in Proposition~\ref{prop:linearbarrier}, with the function $g$ given by
\[ g(y) = \left(y - \frac{11}{20}\right)^{-1} + \left( \frac{9}{10} - y\right)^{-1}. \]
We work on the interval $y\in (\frac{11}{20} + \gamma, \frac{9}{10} - \gamma)$, where we have $C^2$ bounds for $g, g^{-1}$ depending on $\gamma$. Further decreasing $r_0, \epsilon_0$ from before, we can use the graph of $\epsilon G$ as a barrier surface wherever $\epsilon G < \beta|\delta|r^{-2}$. Let us define
\[ \tilde{G}_\epsilon = \min\{ \beta |\delta| r^{-2}, \epsilon G\}. \]
For a constant $C_3$ (depending on the constant $C$ in Proposition~\ref{prop:linearbarrier} and a lower bound for $g$), the surface $M$ is on the negative side of $\tilde{G}_{C_3 \beta|\delta|}$ on $U'$. At the same time we claim that on the boundary of $U'$ the surface $M$ is on the negative side of $\tilde{G}_{C_3(D(r_0) + \gamma d)}$. To see this, we examine the two kinds of boundary components of $U'$:
\begin{itemize}
\item On the set $\{r=r_0\}$ we have that $M$ is the graph of $u$ over $\Lambda T_\delta$ with $|u| \leq D(r_0)r^{-2}$, so $M$ is on the negative side of the graph of  $\tilde{G}_{C_3 D(r_0)}$ for suitable $C_3$. 
\item On the boundary components where $y=\frac{11}{20} + \gamma$ or $y=\frac{9}{10} - \gamma$, we use that $D_{\Lambda T_\delta}(M; B_1\setminus B_{1/2}) = d$, and we have $d < \beta|\delta|$, so $M$ lies between the graphs of
  \[ \pm \min\{ (\beta |\delta| + d) r^{-2}, dr^{-2.1}\}. \]
  In particular $M$ lies on the negative side of the graph of $dr^{-2.1}$, i.e. $M$ lies on the negative side of the graph of $\tilde{G}_{C_3 \gamma d}$ for suitable $C_3$. Note that on these boundary components $g(y) > \gamma^{-1}$. 
\end{itemize}

Letting $\epsilon$ vary between $C_3\beta |\delta|$ and $C_3(D(r_0) + \gamma d)$ we find that $M$ is on the negative side of the graphs of the corresponding $\tilde{G}_\epsilon$, since otherwise for some value $\epsilon = \epsilon_1$ the graph of $\tilde{G}_{\epsilon_1}$ would touch $M$ from the positive side. By \eqref{eq:ubound12} this contact point must happen where $\tilde{G}_{\epsilon_1} < \beta|\delta|r^{-2}$, and by Proposition~\ref{prop:linearbarrier} at these points the graph of $\tilde{G}_{\epsilon_1}$ has negative mean curvature. This is a contradiction. 

It follows that $M$ lies on the negative side of the graph of $\tilde{G}_{C_3(D(r_0) + \gamma d)}$ over $\Lambda T_\delta$ on the region $U'$. Restricting to $y\in (3/5, 4/5)$ we have a uniform upper bound for $g$ (independent of $\gamma$), and so we find that on this region $M$ lies on the negative side of the graph of $C_4(D(r_0) + \gamma d) r^{-2.1}$ over $\Lambda T_\delta$. Because of \eqref{eq:ubound12} we have that $M$ is on the negative side of the graph of
\[ \min\{ \beta|\delta| r^{-2},  C_4(D(r_0) + \gamma d) r^{-2.1}\}, \]
which implies that
\[ D_{\Lambda T_\delta}(M; (B_{4/5}\setminus B_{3/5})\cap \{ r < r_0\}) \leq C_4(D(r_0) + \gamma d), \]
as required. 

It is clear from the arguments that we only require $M$ to be minimal in a neighborhood $\{r < r_1\}$, and in addition the result holds with $\Lambda T_\delta$ replaced by $V_\delta$. 
\end{proof}

From this point on the number $\beta > 0$ in the definition of the distance $D_{\Lambda T_\delta}$ is fixed so that \eqref{eq:nonconc} holds. Next we derive a three annulus lemma for the distance $D$ for minimal surfaces close to one of our minimal comparison surfaces $\Lambda T_\delta$, from the $L^2$ three annulus lemma on the cone $C\times\mathbf{R}$ given in Lemma~\ref{lem:L23annulus}. To ease the notation, unless otherwise specified, we will always measure the distance in the annulus $B_1\setminus B_{\rho_0}$, where $\rho_0$ is as in the $L^2$ three annulus lemma.  The number $L > 0$ given by the following result will remain fixed afterwards. 
\begin{prop}\label{prop:D3annulus}
  There is an $L > 0$ such that if $|\delta|, d$ are sufficiently small, then we have the following. Let $T=\Lambda T_\delta$, where $|\ln \Lambda| < |\delta|^{-\kappa}$. Suppose that $D_T(M) < d$ for $M\in \mathcal{M}$, and $\alpha\in (\alpha_1,\alpha_2)$ for the $\alpha_i$ appearing in Lemma~\ref{lem:L23annulus}. We have
  \begin{itemize}
    \item[(i)] If
  $D_{LT}(LM) \geq L^\alpha D_T(M)$, then $D_{L^2T}(L^2M) \geq L^\alpha D_{LT}(LM)$. 
\item[(ii)] If $D_{L^{-1}T}(L^{-1}M) \geq L^\alpha D_T(M)$, then $D_{L^{-2}T}(L^{-2}M) \geq L^\alpha D_{L^{-1}T}(L^{-1}M)$.
  \end{itemize}
\end{prop}
\begin{proof}
  We prove (i), since the other statement is completely analogous. Suppose for contradiction that we have $\delta_i \to 0, d_i \to 0$, and surfaces $M_i\in\mathcal{M}$ satisfying $D_{LT_i}(LM_i)\to 0$, violating the statement. Let us write $T_i = \Lambda_i T_{\delta_i}$. So
  \[ \label{eq:a1} D_{LT_i}(LM_i) \geq L^\alpha D_{T_i}(M_i), \text{ and } D_{L^2T_i}(L^2M_i) < L^\alpha D_{LT_i}(LM_i). \]
We write $d_i = D_{LT_i}(LM_i)$.

  \bigskip
\noindent{\em Step 1.} We first show that there is a constant $D$ such that for sufficiently large $i$ and any $\lambda\in [1, L^2]$ we have
\[ D_{\lambda T_i}(\lambda M_i) < Dd_i, \]
i.e. $M_i$ is within distance $Dd_i$ from $T_i$ on the annulus $B_1\setminus B_{L^{-2}\rho_0}$. 
To see this suppose that there is no such $D$, and let us denote by $\lambda_i$ the value of $\lambda$ maximizing $D_{\lambda T_i}(\lambda M_i)$ for each $i$. Let us also write $D_i = D_{\lambda_i T_i}(\lambda_i M_i)$, so that up to choosing a subsequence we have $D_i / d_i \to \infty$. In addition we have $D_i\to 0$, since by assumption $M_i$ and $T_i$ both converge to $C\times \mathbf{R}$ on the annulus $\rho\in[L^{-2}\rho_0, 1]$ as $i\to\infty$. It follows from Lemma~\ref{lem:smalldgraph} that for a sequence $r_i\to 0$ we can write $M_i$ as graphs of $u_i$ over $T_i$ on the region $\{r > r_i\}$ in this annulus, and we have an estimate $|r^{2.1} u_i| < C D_i$. The rescaled functions $D_i^{-1} u_i$ then converge, after choosing a subsequence, to a Jacobi field $U$ on $C\times \mathbf{R}$ (on the same annulus), satisfying $|r^{2.1} U| \leq C$. At the same time by \eqref{eq:a1} on the annuli $\rho\in [\rho_0, 1]$ and $\rho\in [L^{-2}\rho_0, L^{-2}]$ we have better estimates $|r^{2.1}u_i| < C L^{-\alpha} d_i$ and $|r^{2.1} u_i| < C L^{-2+\alpha} d_i$ respectively, so the rescaled limit $U$ vanishes on these two annuli. It follows that $U$ is identically zero. Using the non-concentration result with a sufficiently small $\gamma $, this contradicts $D_i = D_{\lambda_i T_i}(\lambda_i M)$ for sufficiently large $i$. More precisely, by the assumption that $\lambda_i$ maximizes $D_{\lambda T_i}(\lambda M_i)$ we have
\[ D_{\lambda_i T_i}(\lambda_i M_i; B_{3/2}\setminus B_{\rho_0 /2}) < CD_i \]
for a uniform $C$. Note that we may modify the $\lambda_i$ slightly if necessary to ensure that $\lambda_i M_i$ is defined on this slightly larger annulus, since by \eqref{eq:a1} the maximal distance does not occur at the ends of the annulus $B_1\setminus B_{L^{-2}\rho_0}$. Proposition~\ref{prop:nonconcentration} (applied to the annuli $B_{3/2}\setminus B_{\rho_0/2}$ and $B_1\setminus B_{\rho_0}$) implies that given $\gamma > 0$ there is an $r_0 > 0$ such that for sufficiently large $i$ we have
\[ D_i = D_{\lambda_i T_i}(\lambda_i M; B_1\setminus B_{\rho_0}) \leq C \sup_{\{r > r_0\}}|r^2u_i| + \gamma CD_i,\]
Since the $D_i^{-1} u_i$ converge to zero uniformly away from the singular set, this is a contradiction for sufficiently large $i$ if $\gamma$ is chosen small enough. 

\bigskip
\noindent{\em Step 2.} As above we still write $M_i$ as the graph of $u_i$ over $T_i$ on the set where $r > r_i$ for a sequence $r_i\to 0$. Using Step 1 we have $|r^{2.1} u_i| < D d_i$ on this set (note that $D$ may depend on $L$). Choosing a subsequence, the rescaled functions $d_i^{-1}u_i$ converge to a Jacobi field $U$ with $|r^{2.1} U| \leq D$. As above this limit cannot be identically zero, since then the non-concentration result would contradict $d_i = D_{LT_i}(LM_i)$ for sufficiently large $i$.

By \eqref{eq:a1} we have the estimates
\[ | r^{2.1} U | \leq C L^{-\alpha}, \text{ on } B_1\setminus B_{\rho_0}, \\
  L^2 |r^{2.1} U| \leq  C L^\alpha, \text{ on } B_{L^{-2}}\setminus B_{L^{-2}\rho_0}. \]
Let us assume that $L$ is of the form $L = \rho_0^{-k}$ for some $k > 0$. 
In the notation of Lemma~\ref{lem:L23annulus} this implies (using $\rho=\rho_0$) that
\[ \Vert U \Vert_0 \leq C' \rho_0^{k\alpha}, \quad \Vert U\Vert_{2k} \leq C' \rho_0^{-k\alpha}, \]
for a constant $C'$ (independent of $L$).

Let $\xi = -1, 0$, or 1. 
Suppose that $\Vert U\Vert_{k+\xi} \geq \rho_0^{-(k+\xi)\alpha_1} \Vert U\Vert_0$. Then for some $i\leq k$ we have $\Vert U\Vert_{i+1} \geq \rho_0^{-\alpha_1}\Vert U\Vert_i$ and so by Lemma~\ref{lem:L23annulus} $\Vert U\Vert_{j+1} \geq \rho_0^{-\alpha_2} \Vert U\Vert_j$ for all $j \geq k+1$. We would then have $\Vert U\Vert_{2k} \geq \rho_0^{-(k-1)\alpha_2} \Vert U\Vert_{k+\xi}$, i.e.
\[ \Vert U\Vert_{k+\xi} \leq C' \rho_0^{-k\alpha + (k-1)\alpha_2}.  \]
If instead $\Vert U\Vert_{k+\xi} \leq \rho_0^{-(k+\xi)\alpha_1}\Vert U\Vert_0$, then
\[ \Vert U\Vert_{k+\xi} \leq C' \rho_0^{k\alpha - (k+\xi)\alpha_1}, \]
Either way, for any $\epsilon > 0$ we have $\Vert U\Vert_{k+\xi} \leq \epsilon$, if $k$ is sufficiently large (depending on $\epsilon, \alpha$).

Using the $L^\infty$ estimate, Lemma~\ref{lem:L2Linfty}, we then have
\[ \rho_0^{-k} |r^2 U| \leq C \epsilon, \text{ on the annulus } B_{\rho_0^{-(k-1)}/2} \setminus B_{2\rho_0^{-(k+2)}}, \]
for a uniform constant $C$. For any $\gamma > 0$ we can apply the non-concentration result (choosing $r_0$ appropriately and $i$ sufficiently large), and we obtain
\[ d_i = D_{LT_i}(LM_i) \leq Cd_i \epsilon  + \gamma D d_i. \]
First we choose $\epsilon < 1/4 C^{-1}$. This determines a large choice of $k$, and therefore $L$ (note that the constant $D$ may depend on $L$). We can then choose $\gamma < 1/4 D^{-1}$, so that in sum we get $d_i < \frac{1}{2} d_i$ for sufficiently large $i$, giving a contradiction. 
\end{proof}

\section{The main argument}\label{sec:mainargument}
In this section we will give the proof of our main result, Theorem~\ref{thm:main}. We first need some preliminary definitions and results. 
Let us denote by $\mathcal{T}$ the set of all comparison surfaces $T_{\delta}$ and their rotations, i.e.
\[ \mathcal{T} = \{ RT_\delta\,:\, |\delta| < \delta_0, \text{ and } R\text{ is a rotation of }\mathbf{R}^9\}. \]
Similarly we let
\[ \mathcal{W} = \{ RW_\delta\,:\, |\delta| < \delta_0, \text{ and } R\text{ is a rotation of }\mathbf{R}^9\}, \]
where $W_\delta$ are the graphs of $-h(\delta)\xi_\delta$ over $V_\delta = C(\Sigma_\delta)$ as in Remark~\ref{rem:Wcone}. 
Here $\delta_0 > 0$ is a fixed constant so that the $T_\delta, W_\delta$ are defined for $|\delta| < \delta_0$. We let $a_1 < 4/3 < a_2$ be constants so that for sufficiently small $\delta$ we have
\[ |\delta|^{a_2} \leq |h(\delta)| \leq |\delta|^{a_1}. \]
We extend our notion of distance to such rotated reference surfaces by letting $D_{RW_\delta}(M) = D_{W_\delta}(R^{-1}M)$ and $D_{RT_\delta}(M) = D_{T_\delta}(R^{-1}M)$. We will need the following results analogous to the triangle inequality for our distance.
\begin{lemma}\label{lem:triangle}
  There is a constant $C>0$ satisfying the following. Suppose that $|\delta|, |\delta'|, |R - \mathrm{Id}|, D_{W_\delta}(M) < C^{-1}$. Then
  \[ \label{eq:tr1} D_{RW_{\delta'}}(M) \leq C( D_{W_\delta}(M) + |\delta - \delta'| + |R - \mathrm{Id}|). \]
  The same inequality holds with $\Lambda T_\delta, \Lambda T_{\delta'}$ instead of $W_\delta, W_{\delta'}$. For relating the distance from $W_\delta$ to the distance from $T_\delta$ we have
  \[ \label{eq:tr2} D_{\Lambda T_\delta}(M) \leq C\Big( D_{W_\delta}(M) + (1 +|\ln \Lambda|) |h(\delta)|\Big) \]
 where $|\ln \Lambda| < |\delta|^{-\kappa}$. The same holds with $T_\delta$ and $W_\delta$ interchanged. 
\end{lemma}
\begin{proof}
  To prove \eqref{eq:tr1}, let $d = D_{W_\delta}(M)$, so that by definition $M \subset N_{2d}(W_\delta)$. Recall that in each $y$-slice $W_\delta$ is an $O(|\delta| r^{-2+\kappa})$-graph over $\delta^{1/3}y H$, and our convention is that $D_{W_\delta}$ is measuring the distance on the annulus $B_1 \setminus B_{\rho_0}$. Let us define
  \[ K = d + |\delta - \delta'| + |R-\mathrm{Id}|. \]
We need to show that for a sufficiently large $C$, if $|\delta|, K < C^{-1}$, then
$R^{-1}M \subset N_{CK}(W_{\delta'})$, 
using that $R^{-1} M \subset R^{-1} N_{2d}(W_{\delta})$. In other words our goal is to show that
\[ \label{eq:N10}
  R^{-1} N_{2d}(W_\delta) \subset N_{CK}(W_{\delta'}). \]
Similarly to the proof of Lemma~\ref{lem:Phidefn1} we split into three cases depending on whether $K$ is much smaller than $|\delta'|$, much larger, or in between.
\begin{itemize}
\item Suppose $K \leq \epsilon |\delta'|$ for a small $\epsilon > 0$. It then follows that $d \ll |\delta|$, and so $N_{2d}(W_{\delta})$ is between the graphs of
  \[ \pm \min\{ (\beta|\delta| + 2d) r^{-2}, 2d r^{-2.1}\} \]
  over $W_\delta$. Since also $|R-\mathrm{Id}| \ll |\delta|$, and rotations are generated by linear growth Jacobi fields on $C\times \mathbf{R}$ that are bounded along the ray $\{0\} \times \mathbf{R}$ (i.e. have milder singularities than $r^{-2}$), it follows that if $\epsilon$ and $K$ are sufficiently small, then $R^{-1}N_{2d} (W_{\delta})$ is contained between the graphs of
  \[ \pm \min\{ (\beta |\delta| + C_1K)r^{-2}, C_1K r^{-2.1}\} \]
  over $W_{\delta}$, for a suitable $C_1$ independent of the other choices. Since $|\delta - \delta'| \leq \epsilon |\delta'|$ we have that $W_{\delta'}$ is bounded between the graphs of $\pm C_2 K r^{-2}$ over $W_\delta$, and so $R^{-1}N_{2d}(W_\delta)$ is bounded between the graphs of
  \[ \pm \min\{ (\beta |\delta'| + C'K)r^{-2}, C' K r^{-2.1}\} \]
  over $W_{\delta'}$ for suitable $C'$. If $\epsilon$ is sufficiently small, this implies $R^{-1}N_{2d}(W_\delta) \subset N_{C'K}(W_{\delta'})$.
\item Suppose that $K\geq \epsilon^{-1} |\delta'|$ for small $\epsilon > 0$. Since $\delta' \ll K$, it follows that for any $C' > 1$ the neighborhood $N_{C'K}(W_{\delta'})$ contains the region between the surfaces $\pm (\beta^{-2} C'K)^{1/3}H$ in each $y$-slice. At the same time $N_{2d}(W_\delta)\subset N_{2d + |\delta|}(W_\delta)$, and by assumption
  \[ 2d + |\delta| \leq 2K + |\delta'| \leq 3K, \]
  so $N_{2d}(W_\delta) \subset N_{3K}(W_\delta)$.  Using that $|\delta|\leq 3K$, this region is contained between the surfaces $\pm (\beta^{-2} 3K)^{1/3}H$. Since $|R-\mathrm{Id}| \leq K$, we then also have that $R^{-1}N_{3K}(W_\delta)$ is contained between the surfaces $\pm (\beta^{-2} C_4K)^{1/3} H$, for suitable $C_4$. If we increase the $C'$ from the previous case so that also $C' > C_4$, then this implies $R^{-1}N_{2d}(W_\delta) \subset N_{C'K}(W_{\delta'})$. 
  
\item Suppose that for the $\epsilon$ found above we have $\epsilon |\delta'| \leq K \leq \epsilon^{-1}|\delta'|$. Then applying the previous case
  \[ R^{-1}N_{2d}(W_\delta) &\subset R^{-1}N_{2d + \epsilon^{-1}|\delta'|}(W_\delta) \\
    &\subset N_{C'(K+\epsilon^{-1}|\delta'|)}(W_{\delta'}) \\
    &\subset N_{(C'+\epsilon^{-2})K}(W_{\delta'}). 
  \]
  Setting $C= C' +\epsilon^{-2}$ we get \eqref{eq:N10}. 
\end{itemize}

For \eqref{eq:tr2} we can argue similarly, letting $K = d + (1+|\ln \Lambda|)|h(\delta)|$, with $d = D_{W_\delta}(M)$. We need to show that
\[ \label{eq:N11} N_{2d}(W_\delta) \subset N_{CK}(\Lambda T_\delta) \]
for sufficiently large $C$
\begin{itemize} 
\item When $K\leq \epsilon |\delta|$ for sufficiently small $\epsilon$, then $N_{2d}(W_\delta)$ is contained between the graphs of
  \[ \pm \min\{ (\beta |\delta| + 2K) r^{-2}, 2K r^{-2.1}\} \]
  over $W_\delta$. By Proposition~\ref{prop:Tdelta1} and Remark~\ref{rem:Wcone} we have that $\Lambda T_\delta$ is the graph of a function $f$ over $V_\delta$, with $\Vert f\Vert_{C^{2,\alpha}_{1,-2}} \lesssim (1 + |\ln \Lambda|)|h(\delta)|$. Therefore $N_{2d}(W_\delta)$ is contained between the graphs of
  \[ \pm \min \{ (\beta |\delta| + C'K)r^{-2}, C'Kr^{-2.1}\} \]
over $\Lambda T_\delta$ for some $C'$. If $\epsilon$ is sufficiently small, this implies that $N_{2d}(W_\delta) \subset N_{C'K}(\Lambda T_\delta)$.  
\item When $K \geq \epsilon^{-1} |\delta|$ for small $\epsilon$, then $N_{CK}(\Lambda T_\delta)$ is the region between $\pm (\beta^{-2} CK)^{1/3}  H$. We also have $N_{2d}(W_\delta) \subset N_{2d+|\delta|}(W_\delta)$ and so $N_{2d}(W_\delta) \subset N_{3K}(W_\delta)$. The latter neighborhood in turn is the region between $\pm (\beta^{-2}3 K)^{1/3} H$. Therefore $N_{2d}(W_\delta) \subset N_{C'K}(\Lambda T_\delta)$ for $C' > 3$. 
 \item $\epsilon |\delta| \leq K \leq \epsilon^{-1} |\delta|$, then we can argue as above, setting $C=C'+\epsilon^{-2}$ for the $C'$ that works in the two previous cases. 
\end{itemize}
\end{proof}

For a surface $M$ defined in the annulus $B_1\setminus B_{\rho_0}$ let us write
\[ \mathcal{A}(M) = \mathrm{Area}(M \cap B_{1/2}) - \mathrm{Area}( (C\times \mathbf{R}) \cap B_{1/2}) \]
for the area excess. By monotonicity if $M\in\mathcal{M}$ has the same density as $C\times\mathbf{R}$ at the origin, then $\mathcal{A}(M) \geq 0$ and also $\mathcal{A}(LM) \leq \mathcal{A}(M)$ for any $L > 1$. Given $B > 0$ we define
\[ \label{eq:EBdefn} E_B(M) = \inf\{ D_{W}(M) + D_{W}(L^BM)\,:\, W\in \mathcal{W}\}. \]
The purpose of this will be to compare $M$ to $L^BM$ using that the $W$ are cones. Note that it is not clear how to define a distance similar to $D$ between $M$ and $L^BM$ directly, since we do not have precise enough information about $M$ near the singular set of the tangent cone. However $E_B(M)$ can be used to control other weaker types of distances between $M$ and $L^BM$, such as the flat distance $d_{\mathcal{F}}(M, L^BM)$ over the annulus $B_1\setminus B_{\rho_0}$. 
\begin{lemma}\label{lem:DFbound}
  Suppose that $M\in \mathcal{M}$ is a minimal surface in the ball $B_2$, with $\mathrm{Area}(M) < 2\mathrm{Area}(B_2\cap (C\times\mathbf{R}))$. There is a constant $C > 0$ independent of $M$  such that over the annulus $B_1\setminus B_{\rho_0}$ we have $d_{\mathcal{F}}(M, L^BM)  \leq C E_B(M)$. 
\end{lemma}
\begin{proof}
  Suppose that $W \in \mathcal{W}$. Up to a rotation we can assume that $W = W_\delta$ for some $|\delta| < \delta_0$. We will show that for a uniform constant $C$ we have $d_{\mathcal{F}}(W_\delta, M) \leq C D_{W_\delta}(M)$ over the annulus $B_1\setminus B_{\rho_0}$.  The required result will follow easily from this and the corresponding bound for $L^BM$, since then
  \[ d_{\mathcal{F}}(M, L^BM) &\leq d_{\mathcal{F}}(W_\delta, M) + d_{\mathcal{F}}(W_\delta, L^BM) \\
    &\leq C(D_{W}(M) + D_{W}(L^BM)), \]
  and we can minimize over all $W\in\mathcal{V}$.

  Let us write $d = D_{W_\delta}(M)$. Then $M\subset N_{2d}(W_\delta)$, and from the definition of the distance it is not hard to see that for some $C_1 > 0$ (independent of $\delta$), on the set where $r > C_1d^{1/3}$, the surface $M$ is a graph over $W_\delta$ of a function $u$ such that
  \[ |u| \leq \min\{ (\beta |\delta| + C_1d) r^{-2}, C_1d r^{-2.1}\} \leq C_1 d r^{-2.1}. \]
Since we have a bound for the densities of $W_\delta, M$ at each point we also know that, for any $s > 0$, on the region where $s/2 < r < s$ the areas of $W_\delta, M$ are both bounded by $C_2 s^7$ for some $C_2 > 0$. Using this we have
  \[ d_{\mathcal{F}}(W_\delta, M) &\leq  C(d^{1/3})^7 + \int_{C_1d^{1/3}}^1 C_1ds^{-2.1} C_2s^6\,ds  \\
    &\leq Cd, \]
  as required. 
\end{proof}

Recall that the distance between $W_\delta$ and $T_\delta$ is of order $|h(\delta)|$. When the distance from $M$ to $T_\delta$ is much smaller than this, then we can obtain more information by comparing $M$ to the minimal surface $T_\delta$ rather than to the (non-minimal) cone $W_\delta$. 
The following result shows that in this case we can control the change in area excess of $M$ in terms of the corresponding change for $T_\delta$. 
\begin{prop}\label{prop:areadrop}
  Suppose that $M\in \mathcal{M}$ is area minimizing in $B_1$. There are $C, \epsilon, \delta_0 > 0$ with the following property. Suppose that $D_{T_\delta}(M) < \epsilon |h(\delta)|$, where $|\delta| < \delta_0$. Then
  \[ \label{eq:Adrop1} \mathcal{A}(M) - \mathcal{A}(2M) > C^{-1} |h(\delta)|^2. \]
  It follows that for sufficiently small $\theta > 0$ we have
  \[ \label{eq:Adrop2} \mathcal{A}(M)^\theta - \mathcal{A}(2M)^\theta > C^{-1} |h(\delta)|. \]
\end{prop}
\begin{proof}
  Using the monotonicity formula (see \cite{SimonGMT}), we have
  \[ \mathcal{A}(M) - \mathcal{A}(2M) = \int_{M\cap (B_{1/2}\setminus B_{1/4})} \frac{ |z^\perp|^2}{|z|^{10}}. \]
  From Proposition~\ref{prop:Tm} we have an estimate for the corresponding integral over $T_\delta$:
  \[ \label{eq:Tm1} \int_{T_\delta \cap (B_{1/2}\setminus B_{1/4})} \frac{ |z^\perp|^2}{|z|^{10}} > C^{-1} |h(\delta)|^2. \]
  To compare the two integrals, let us work on $T_\delta$, and write $z_M, n_M, dA_M$ (resp. $z_T, n_T, dA_T$) for the position vector, normal vector and area form of $M$ (resp. $T_\delta$). 

 Since $|h(\delta)| < |\delta|^{a'}$ with $a' > 1$, for sufficiently small $\delta$ we have $D_{T_\delta}(M) < \epsilon |h(\delta)| < \beta |\delta|$. It follows that  $M$ is the graph of a function $f$ over $T_\delta$ on the annulus $B_1\setminus B_{\rho_0}$ which satisfies
 \[ \label{eq:fb20} |f| \leq \min\{ (\beta|\delta| + \epsilon |h(\delta)|) r^{-2}, \epsilon |h(\delta)| r^{-2.1} \} \leq \epsilon |h(\delta)|r^{-2.1}. \]
 Since the regularity scale of $T_\delta$ at each point in the annulus is of order $r$, and $M$ is minimal, we deduce that on a smaller annulus we have $|\nabla f| < Cr^{-3.1}\epsilon |h(\delta)|$. We have $z_M = z_T + f n_T$, and using Lemma~\ref{lem:grapharea} we have
 \[ n_M = n_T - \nabla f + O( \epsilon^2 |h(\delta)|^2 r^{-6.2}). \]
 It follows that
 \[ z_M \cdot n_M &= z_T \cdot n_T + f - z_T\cdot\nabla f + O( \epsilon^2 |h(\delta)|^2 r^{-6.2}) \\
   &= z_T \cdot n_T + O( \epsilon |h(\delta)| r^{-3.2}), \]
 where we used that $|h(\delta)| r^{-3} \leq C$ on $T_\delta$. From \eqref{eq:zTperp} we also have  $z_T\cdot n_T = O(|h(\delta)| r^{-2})$, and so the previous formula implies
 \[ z_M\cdot n_M = O( |h(\delta)| r^{-3.2}). \]
 It follows that
 \[ |z_M^\perp|^2 - |z_T^\perp|^2 = O( \epsilon |h(\delta)|^2 r^{-6.4}). \]
 We also have $|z_M|^{-10} - |z_T|^{-10} = O(\epsilon |h(\delta)| r^{-2.1})$ and $\frac{dA_M}{dA_T} - 1 = O(\epsilon |h(\delta)| r^{-2.1})$. Combining these estimates we have
 \[\int_{M\cap (B_{1/2}\setminus B_{1/4})} \frac{ |z^\perp|^2}{|z|^{10}} -
   \int_{T_\delta \cap (B_{1/2}\setminus B_{1/4})} \frac{ |z^\perp|^2}{|z|^{10}} = O(\epsilon |h(\delta)|^2). \]
Together with \eqref{eq:Tm1}, we obtain \eqref{eq:Adrop1} once $\epsilon$ is sufficiently small.

To see Inequality~\eqref{eq:Adrop2} we first use that $M$ is area minimizing to bound $\mathrm{Area}(M\cap B_{1/2})$ in terms of the cone over $M\cap \partial B_{1/2}$:
\[ \mathrm{Area}(M\cap B_{1/2}) \leq \frac{1}{16} \mathrm{Area}(M\cap \partial B_{1/2}). \]
To estimate the area of $M\cap \partial B_{1/2}$ note that, by the construction of $T_\delta$ together with \eqref{eq:fb20}, this cross section is the graph of a function $F$ over $\frac{1}{2}\Sigma_\delta$, satisfying
\[ |r^{-1}F| + |\nabla F| \leq C |h(\delta)| r^{-3.1}. \]
Using Lemma~\ref{lem:grapharea} and \eqref{eq:ASbound} we then have
\[ \mathrm{Area}(M\cap \partial B_{1/2}) &\leq \mathrm{Area}((C\times\mathbf{R})\cap \partial B_{1/2}) + C|\delta||h(\delta)| + O(|h(\delta)|^2). \]
It follows that
\[ \mathrm{Area}(M\cap B_{1/2}) \leq \mathrm{Area}((C\times\mathbf{R}) \cap B_{1/2}) + C|\delta| |h(\delta)|, \]
and so
\[ \mathcal{A}(2M) \leq \mathcal{A}(M) \leq C|\delta| |h(\delta)|. \]
It now follows from the mean value theorem that for $\theta\in (0,1)$
  \[ \mathcal{A}(M)^\theta - \mathcal{A}(2M)^\theta &\geq \theta \mathcal{A}(M)^{\theta-1}(\mathcal{A}(M) - \mathcal{A}(2M)) \\
    &\geq c_\theta |\delta|^{\theta - 1} |h(\delta)|^{\theta-1} |h(\delta)|^2 \\
    &\geq c_\theta |\delta|^{\theta-1}|\delta|^{\alpha_2\theta} |h(\delta)|,  \]
for a $c_\theta > 0$ depending on $\theta$. The required result 
follows since $\theta-1 + \alpha_2 \theta  < 0$ for sufficiently small $\theta > 0$.
\end{proof}

We will need the following extension lemma.
\begin{lemma}\label{lem:extension}
Given any $\epsilon, K > 0$ we have $c  > 0$ with the following property. Suppose that $M\in\mathcal{M}$ is a minimal surface in $B_1$ with density equal to that of $C\times \mathbf{R}$ at the origin, and such that $D_{\Lambda T_\delta}(M) < c$ with $|\delta| < c, |\ln \Lambda| < K$. Then $D_{K\Lambda T_\delta}(KM) < \epsilon$. 
\end{lemma}
\begin{proof}
  We can argue by contradiction. Suppose that $M_i\in\mathcal{M}$ is a sequence of minimal surfaces with density equal to that of $C\times\mathbf{R}$ at the origin, and such that $D_{\Lambda_i T_{\delta_i}}(M_i) \to 0$ for a sequence $|\ln \Lambda_i| < K$ and $\delta_i\to 0$. After choosing a subsequence we can assume that the $M_i$ converge to a minimal cone $M_\infty$. Since $\Lambda_i T_{\delta_i} \to C\times\mathbf{R}$, we must have $M_\infty = C\times\mathbf{R}$. Since we also have $K\Lambda_i T_{\delta_i} \to C\times\mathbf{R}$ and $KM_i\to C\times\mathbf{R}$, it follows that $D_{K\Lambda T_{\delta_i}}(KM_i) < \epsilon$ for sufficiently large $i$. 
\end{proof}

We also have the following, characterizing minimal cones near $C\times\mathbf{R}$. In particular this result says that $C\times \mathbf{R}$ is not integrable, since there is no family of minimal cones near $C\times\mathbf{R}$ modeled on the Jacobi field $\phi$. 
\begin{prop}\label{lem:nearbycones}
  There is an $\epsilon > 0$ with the following property. Suppose that $M$ is a minimal cone in $\mathbf{R}^9$, such that on the unit ball the Hausdorff distance $d_H(M, C\times\mathbf{R}) < \epsilon$. Then $M$ is a rotation of $V_0 = C\times\mathbf{R}$. 
\end{prop}
\begin{proof}
  We will argue by contradiction.
  Let $M_i$ be a sequence of minimal cones converging to $C\times\mathbf{R}$ on the unit ball, such that the $M_i$ are not rotations of $V_0$. For any $i$, if
  \[ \inf\{ D_{RV_\delta}(M_i)\,:\, R\text{ is a rotation and }|\delta|\leq \delta_0 \} = 0, \]
then we must have $M_i = RV_{\delta}$ for some rotation $R$ and $\delta$, but the $V_\delta$ are not minimal for $\delta\not=0$, so $M_i= RV_0$. Therefore we can assume that the infimum above is positive for all $i$.  We choose a sequence of rotations $R_i$ and $\delta_i\to 0$ such that
  \[ \label{eq:Didefn} D_i := D_{R_iV_{\delta_i}}(M_i) < 2\inf\{ D_{RV_\delta}(M_i)\}. \]
  We have $D_i\to 0$ and by applying rotations to the $M_i$ we can assume that $R_i=\mathrm{Id}$.  Note also that since the mean curvature of $V_{\delta_i}$ is $h(\delta_i)\zeta \rho^{-2}$, we must have $D_i > C^{-1} |h(\delta_i)|$ for a uniform constant $C$. We can write $M_i$ as the graph of $u_i$ over $V_{\delta_i}$ on the annulus $B_1\setminus B_{\rho_0}$, where $r > r_i$ for some $r_i\to 0$. Choosing a subsequence we can assume that $D_i^{-1}u_i \to U$ locally uniformly, where $|r^{2.1} U|\leq C$ and $U$ satisfies the equation
  \[ L_{C\times \mathbf{R}}U + a \zeta\rho^{-1} = 0 \]
  for a constant $a\in\mathbf{R}$. Indeed, away from the support of $\zeta \rho^{-1}$ both $V_{\delta_i}$ and $M_i$ are minimal, so in the limit $U$ is a Jacobi field. On the support of $\zeta \rho^{-1}$ the functions $u_i$ satisfy non-linear equations of the form
  \[ \label{eq:Ueq1}
    h(\delta_i)\zeta \rho^{-2} + L_{V_{\delta_i}} u_i + Q_i(u_i) =0, \]
  where $Q_i$ collects the higher order terms in the mean curvature operator for a graph over $V_{\delta_i}$. Dividing through by $D_i$ we have
  \[ D_i^{-1} h(\delta_i) \zeta \rho^{-2} + L_{V_{\delta_i}}(D_i^{-1}u_i) = O(D_i).\]
  Using that $|D_i^{-1} h(\delta_i) | < C$, we have uniform derivative bounds for $D_i^{-1}u_i$ and can pass to a limit $U$ along a subsequence, satisfying \eqref{eq:Ueq1}. 
In addition since the $M_i$ and the $V_{\delta_i}$ are cones, the function $U$ is homogeneous of degree one. 

 Using Lemma~\ref{lem:zetaJacobi} we can write $U = f + \lambda(y^3r^{-2}-y)$, where $f$ is the Jacobi field corresponding to a rotation and $\lambda\in \mathbf{R}$. Let $R_i'$ denote the rotation corresponding to the Jacobi field $D_if$, and let $\delta_i' = \delta_i + D_i\lambda$. We can also view the $M_i$ as graphs of functions $u_i'$ over $R_i'V_{\delta_i'}$, over sets $r > r_i'$. By Lemma~\ref{lem:triangle} we still have $D_{R_i'V_{\delta_i'}}(M_i) \leq CD_i$ and in addition by construction we have $D_i^{-1}u_i' \to 0$. We can apply the non-concentration estimate Proposition~\ref{prop:nonconcentration} since the $M_i$ are minimal near the singular set (note also that since $M_i$ is a cone the estimate is much simpler in this case). It follows that for any $\epsilon > 0$ we have $D_{R_i'V_{\delta_i'}}(M_i; B_{4/5}\setminus B_{3/5}) < \epsilon D_i$ for sufficiently large $i$. Since $M_i$ and $R_i'V_{\delta_i'}$ are cones, this contradicts the definition of $D_i$, i.e. Equation~\eqref{eq:Didefn}. 
\end{proof}

The main result leading to uniqueness of the tangent cone is the following.
\begin{prop}\label{prop:EBdecay}
  There are $\theta, C , B > 0$ with the following property. Let $M\in \mathcal{M}$ be a minimal surface in $B_1$, with density equal to that of the cone $C\times\mathbf{R}$ at the origin. If the Hausdorff distance from $M$ to $C\times \mathbf{R}$ on $B_1$ is sufficiently small, then one of the following holds for the quantity $E_B$ defined in \eqref{eq:EBdefn}:
  \begin{itemize}
  \item[(i)] $E_B(L^BM) \leq \frac{1}{2} E_B(M)$. 
  \item[(ii)] $E_B(L^BM) \leq C\Big(\mathcal{A}(L^BM)^\theta - \mathcal{A}(2L^BM)^\theta\Big)$.
  \end{itemize}
\end{prop}
\begin{proof}
  Suppose that we have a sequence $M_i \to C\times\mathbf{R}$ on $B_1$ in the Hausdorff sense where each $M_i\in\mathcal{M}$ has the same density as $C\times \mathbf{R}$ at the origin. Let $B > 0$ be a large integer. We will show that if $B$ is chosen sufficiently large, then there are constants $C, \theta$ so that along a subsequence of the $M_i$ either (i) or (ii) will hold.

  \bigskip
  \noindent{\em Step 1.} Let $\alpha_1', \alpha_2'$ be such that $\alpha_1 < \alpha_1' < \alpha_2' < \alpha_2$, where $\alpha_1,\alpha_2$ are the constants from Proposition~\ref{prop:D3annulus}. We first show that if $B$ is large, then along a subsequence we can find $\delta_i \to 0$ and rotations $R_i$ converging to the identity such that one of the following two possibilities holds:
  \begin{itemize}
  \item[(a)] $D_{L^{2B} R_i T_{\delta_i}}(L^{2B}  M_i) \geq L^{\alpha_1' B} D_{L^B R_i T_{\delta_i}}(L^B  M_i)$,
    \item[(b)] $D_{ R_i T_{\delta_i}}(M_i) \geq L^{\alpha_2' B} D_{L^B R_i T_{\delta_i}}(L^B  M_i)$. 
    \end{itemize}
    This essentially follows from the fact that the family of surfaces $RT_\delta$ accounts for all the linear growth Jacobi fields on $C\times\mathbf{R}$, and so by comparing $M_i$ to the correct surface of this type, we are in the setting of the 3-annulus lemma for a Jacobi field with no degree one component. 
    
    In more detail, let us first choose sequences $R_i, \delta_i$ such that
    \[ D_{L^BR_iT_{\delta_i}}(L^BM_i) < 2\inf\{ D_{L^BT}(L^BM_i)\,:\, T = RT_\delta \}, \]
    where the infimum is over all rotations $R$ and $\delta$ for which $T_\delta$ is defined. I.e. $R_iT_{\delta_i}$ is approximately a best fit surface to $M_i$ among the family of comparison surfaces $RT_\delta$ on the annulus $B_{L^{-B}}\setminus B_{L^{-B}\rho_0}$. 
    By applying rotations to the $M_i$ we can assume that $R_i = \mathrm{Id}$. Suppose that (a) and (b) both fail for the sequence $\delta_i$ and $R_i = \mathrm{Id}$.  Let us write $T_i = T_{\delta_i}$. Letting $d_i = D_{L^B T_i}(L^BM_i)$, we have $d_i\to 0$ and since (a) and (b) fail, we have
    \[ \label{eq:b1} D_{L^{2B} T_i}(L^{2B} M_i) &< L^{\alpha_1'B} d_i, \\
      D_{T_i}(M_i) &< L^{\alpha_2'B} d_i. \]
    On the annulus $B_1\setminus B_{L^{-2B}\rho_0}$ we can write $M_i$ as the graph of $u_i$ over $T_i$ on the set where $r > r_i$, for a sequence $r_i\to 0$. Arguing as in Step 1 of the proof of Proposition~\ref{prop:D3annulus} there is a constant $D$ (which depends on $B$) such that $|r^{2.1} u_i| < D d_i$ on the set where $r > r_i$. Choosing a subsequence we can assume that $d_i^{-1} u_i$ converges to a Jacobi field $U$ on  $C\times \mathbf{R}$, defined in $B_1\setminus B_{L^{-2B}\rho_0}$, such that $|r^{2.1} U| \leq D$. The inequalities \eqref{eq:b1} imply that in the notation of Lemma~\ref{lem:L23annulus}, with $L = \rho_0^{-k}$ as in Proposition~\ref{prop:D3annulus}, we have
    \[ \label{eq:b2}\Vert U\Vert_0 \leq C' L^{\alpha_2' B}, \quad \Vert U\Vert_{2Bk} \leq C' L^{\alpha_1' B} \]
    for a constant $C'$ independent of $B$.

    Let $U_1$ be the degree one homogeneous Jacobi field, such that $U-U_1$ has no degree one piece. By Lemma~\ref{lem:zetaJacobi} we can write $U_1 = f + \lambda(y^3 r^{-2} - y)$, where $f$ is the Jacobi field corresponding to a rotation. Let $\delta_i' = \delta_i + d_i\lambda$ and let $R_i'$  denote the rotation corresponding to the Jacobi field $d_i f$. Let $T_i' = R_i' T_{\delta_i'}$. By construction the $M_i$ are graphs of new functions $u_i'$ over $T_i'$ over the set where $r > r_i'$, for a sequence $r_i' \to 0$. In addition by Lemma~\ref{lem:triangle} we have a bound $D_{\Lambda T_i'}(\Lambda M_i) < D' d_i$
    for any $\Lambda \in [L^{-2B}, 1]$ and the functions $d_i^{-1} u_i'$ converge to $U' = U-U_1$.

    Since $U_1$ has degree one, the function $U'$ also satisfies the inequalities \eqref{eq:b2}, for a possibly larger constant $C'$ (still independent of $B$). As $U'$ has no degree one component, Lemma~\ref{lem:L23annulus} implies that either we have
    \[ \Vert U'\Vert_{2Bk} \geq \rho_0^{-(Bk-1)\alpha_2} \Vert U'\Vert_{Bk+1}, \]
    or
    \[ \Vert U'\Vert_{0} \geq \rho_0^{-(Bk+1)\alpha_2}\Vert U'\Vert_{Bk+1}. \]
    Either way, we have $\Vert U'\Vert \leq C' L^{-B(\alpha_2 -\alpha_2')}$, using also that $\alpha_1' < \alpha_2'$. Lemma~\ref{lem:L2Linfty} then implies that
    \[ L^B \sup_{B_{L^{-B+1}}\setminus B_{L^{-B-1}\rho_0}} | r^2 U' | \leq C' L^{(\alpha_2' - \alpha_2)B}, \]
    for a larger $C'$, still independent of $B$. 
    After rescaling this estimate to the annulus $B_L\setminus B_{L^{-1}\rho_0}$, we can apply Proposition~\ref{prop:nonconcentration}, keeping in mind that we have the bound $D_{\Lambda T_i'}(\Lambda M_i) < D'd_i$. This implies that for any $\gamma > 0$, once $i$ is sufficiently large, we will have
    \[ D_{L^B T_i'}(L^B M_i) < C' L^{(\alpha_2'-\alpha_2)B}d_i + \gamma D'd_i. \]
    Here $D'$ depends on $B$, but given $B$ we can choose $\gamma$ small so that for sufficiently large $i$ we end up with
    \[ D_{L^B T_i'}(L^B M_i) < 2C' L^{(\alpha_2'-\alpha_2)B}d_i, \]
    with $C'$ independent of $B$. By our choice of $T_i$ we have
    \[ D_{L^B T_i'} (L^B M_i) > \frac{1}{2} D_{L^B T_i}(L^B M_i) = \frac{d_i}{2}, \]
    and so $d_i < 4 C' L^{(\alpha_2' - \alpha_2)B} d_i$. 
Since $\alpha_2' < \alpha_2$, this is a contradiction if $B$ is chosen to be sufficiently large. We now assume that $B$ is chosen large enough, although we may need to choose $B$ even larger below. Replacing the original sequence with a subsequence we will also assume that either condition (a) or (b) holds for all $i$. Finally, we can apply rotations to the sequence so that we can assume $R_i=\mathrm{Id}$. 

    \bigskip
    \noindent{\em Step 2.} Let us suppose that condition (a) holds for all $i$, so
    \[ D_{L^{2B} T_{\delta_i}}(L^{2B}M_i) \geq L^{\alpha_1' B} D_{L^B T_{\delta_i}}(L^B M_i), \]
    for a sequence $\delta_i\to 0$. We want to apply Proposition~\ref{prop:D3annulus} repeatedly to estimate $D_{L^{kB} T_{\delta_i}}(L^{kB}M_i)$ for $k=2,3,\ldots$. We can do this as long as $D_{L^{kB}T_{\delta_i}}(L^{kB}M_i)$ remains sufficiently small and in addition we also need $kB \leq |\delta_i|^{-\kappa}$ for the scaled surfaces $L^{kB} T_{\delta_i}$ to be defined. More precisely, we use a $\kappa$ that is slightly smaller than the $\kappa_{5.11}$ appearing in Proposition~\ref{prop:D3annulus}. Then for any given $B,L$, if $kB\leq |\delta_i|^{-\kappa}$, it follows that $|\ln (L^{(k+2)B})| < |\delta_i|^{-\kappa_{5.11}}$ for sufficiently large $i$, and so $L^{(k+2)B} T_{\delta_i}$ is defined, and Proposition~\ref{prop:D3annulus} can be applied. 
    
 Using the extension result Lemma~\ref{lem:extension} we can choose $d_0(B), d_1(B) > 0$ depending on $B$ such that if
    $D_{L^{kB}T_{\delta_i}}(L^{kB}M_i) < d_0(B)$ then $D_{L^{(k+2)B}T_{\delta_i}}(L^{(k+2)B}M_i) < d_1(B)$. We choose $d_0(B), d_1(B)$ small enough for Proposition~\ref{prop:D3annulus} to be applied to $T=L^{(k+\xi)B}T_{\delta_i}$ and $M=L^{(k+\xi)B}M_i$ with $\xi=0,1,2$, for sufficiently large $i$, as long as also $kB\leq |\delta_i|^{-\kappa}$. There are two possibilities.
    \begin{itemize}
    \item[\bf (a1)] Suppose that as long as $kB < |\delta_i|^{-\kappa}$ we have $D_{L^{kB}T_{\delta_i}}(L^{kB}M_i) < d_0(B)$ for all large $i$. Then we can apply the three-annulus lemma $B^{-1}|\delta_i|^{-\kappa}$ times to find that
      \[ \label{eq:D2} D_{L^{B}T_{\delta_i}}(L^B M_i), D_{L^{2B}T_{\delta_i}}(L^{2B}M_i) < L^{- B^{-1}|\delta_i|^{-\kappa}\alpha_1' B} d_0(B). \]
      Recall that $|h(\delta)| > |\delta|^{a_2}$. For any $\epsilon > 0$, it follows that for sufficiently large $i$ (depending on $\epsilon, B$) we have  $D_{L^{B}T_{\delta_i}}(L^B M_i) < \epsilon |h(\delta_i)|$. From Proposition~\ref{prop:areadrop} we get
      \[ \mathcal{A}(L^BM_i)^\theta - \mathcal{A}(2L^{B}M_i)^\theta > C^{-1} |h(\delta_i)|. \]
      In addition from Equations \eqref{eq:D2} and \eqref{eq:tr2} we have
      \[ D_{W_{\delta_i}}(L^BM_i), D_{W_{\delta_i}}(L^{2B}M_i) < CB|h(\delta_i)|. \]
      Combining these inequalities we get
      \[ \label{eq:EBb1} E_B(L^BM) < CB\Big(\mathcal{A}(L^BM_i)^\theta - \mathcal{A}(2L^{B}M_i)^\theta\Big). \]
    \item[\bf (a2)] Along a subsequence there are $k_i$ with $k_i < B^{-1}|\delta_i|^{-\kappa}$
      such that
      \[ D_{L^{k_iB}T_{\delta_i}}(L^{k_iB}M_i) \geq d_0(B), \]
      and $k_i$ is the smallest such choice for each $i$. By the three-annulus result we still have
      \[\label{eq:m2} D_{L^{(k_i+1)B}T_{\delta_i}}(L^{(k_i+1)B}M_i) \geq L^{\alpha_1'B} D_{L^{k_iB}T_{\delta_i}}(L^{k_iB}M_i) \geq L^{\alpha_1' B}d_0(B). \]
      By area monotonicity, after choosing a subsequence we can assume that $L^{(k_i-1)B}M_i, L^{k_iB}M_i, L^{(k_i+1)B}M_i$ all converge to a minimal cone $M_\infty$. Using that $D_{L^{(k_i-1)B}T_{\delta_i}}(L^{(k_i-1)B}M_i) \leq d_0(B)$, and that $L^{(k_i-1)B}T_{\delta_i}$ converges to $V_0 = C\times\mathbf{R}$, Proposition~\ref{lem:nearbycones} implies that $M_\infty = RV_0$ for a rotation $R$ and in addition $|R-\mathrm{Id}|  < Cd_0(B)$ for a $C$ independent of $B$.

      Using the fact that $L^{(k_i+1)B}T_{\delta_i} \to V_0$ and $L^{(k_i+1)B} M_i \to RV_0$, we deduce using Lemma~\ref{lem:triangle} that
      \[ D_{L^{(k_i+1)B}T_{\delta_i}}(L^{(k_i+1)B} M_i) &\leq C\Big(D_{W_{\delta_i}}(L^{(k_i+1)B}M_i) + (1 + k_iB) |h(\delta_i)|\Big) \\
        &\leq C\Big( C(D_{RV_0}(L^{(k_i+1)B}M_i) + |\delta_i| + |R - \mathrm{Id}|) \\&\qquad\qquad
        + (1 +k_iB)|h(\delta_i)|\Big). \]
      As $i\to \infty$ we have $D_{RV_0}(L^{(k_i+1)B}M_i), |\delta_i| \to 0$, and also since $k_i < B^{-1}|\delta_i|^{-\kappa}$ we have $k_i |h(\delta_i)| \to 0$. It follows that
      \[ D_{L^{(k_i+1)B}T_{\delta_i}}(L^{(k_i+1)B} M_i) &\leq C( 1+ d_0(B)) \]
      for sufficiently large $i$, for $C$ independent of $B$. Comparing this with \eqref{eq:m2} we have
      \[ C(1 + d_0(B)) \geq  L^{\alpha_1' B}d_0(B), \]
      which is a contradiction if $B$ is chosen sufficiently large. 
    \end{itemize}
    \bigskip
    \noindent{\em Step 3.} Finally we suppose that condition (b) holds for all $i$, while condition (a) fails. We assume that $R_i = \mathrm{Id}$ and we let $d_i = D_{L^BT_{\delta_i}}(L^B M_i)$. Again there are two possibilities.
    \begin{itemize}
    \item[\bf (b1)] Let $\epsilon > 0$ be the constant from Proposition~\ref{prop:areadrop} and suppose that $d_i < \epsilon |h(\delta_i)|$. Then arguing similarly to case (a1) above, we find that
      \[ E_B(L^BM) < C_B (\mathcal{A}(L^B M_i)^\theta - \mathcal{A}(2L^B M_i)^\theta), \]
      for a constant $C_B$ depending on $B$.
    \item[\bf (b2)] We have $d_i \geq \epsilon |h(\delta_i)|$, and in addition
      \[ \label{eq:db4} D_{T_{\delta_i}}(M_i) &\geq L^{\alpha_2' B} d_i, \\
        D_{L^{2B}T_{\delta_i}}(L^{2B}M_i) &\leq L^{\alpha_1' B} d_i. \]
      We can estimate $E_B(L^BM)$ from above, using Lemma~\ref{lem:triangle}:
      \[\label{eq:EBest20} E_B(L^BM) &\leq D_{W_{\delta_i}}(L^BM_i) + D_{W_{\delta_i}}(L^{2B}M_i) \\
        &\leq C(d_i + CB|h(\delta_i)|) + C(L^{\alpha_1'B}d_i + CB|h(\delta_i)|) \\
        &\leq Cd_i(1 + B\epsilon^{-1}) (1 + L^{\alpha_1'B}). \]
      We claim that for sufficiently large $B$ we have $E_B(M_i) \geq 2E_B(L^B M_i)$, once $i$ is large enough. Suppose that this were not the case, i.e. $E_B(M_i) < 2E_B(L^BM_i)$ for all $i$. By definition there are rotations $R_i$, and $\delta_i'$, such that
      \[ \label{eq:d3} D_{R_iW_{\delta_i'}}(M_i) + D_{R_i W_{\delta_i'}}(L^B M_i) < 2E_B(L^BM_i). \]
      In particular $L^B M_i$ is contained in the neighborhood $N_{2E_B(L^BM_i)}(R_i W_{\delta_i'})$, while by the definition of $d_i$ it is also contained in the neighborhood $N_{2d_i}(L^B T_{\delta_i})$. Recall that on a fixed region away from the singular ray, say $\{r > 1/10\}$, these neighborhoods are uniformly equivalent to neighborhoods defined in terms of the Hausdorff distance. In addition on such a neighborhood the Hausdorff distance from $L^B T_{\delta_i}$ to $W_{\delta_i}$ is at most $CB |h(\delta_i)|$. It follows from this that we must have (after multiplying $R_i$ by an element of the symmetry group $O(4)\times O(4)$)
      \[ \label{eq:Rd3} |R_i - \mathrm{Id}| + |\delta_i - \delta_i'| \leq C(E_B(L^BM_i) + d_i + B|h(\delta_i)|) \leq Cd_iB L^{\alpha_1' B}, \]
      where we absorbed various factors (including $\epsilon^{-1}$) into $C$ using \eqref{eq:EBest20} and $|h(\delta_i)|\leq \epsilon^{-1}d_i$, and we assumed that $B\epsilon^{-1}, L^{\alpha_1' B} > 1$.

      We now use the other part of \eqref{eq:d3}, i.e. $D_{R_iW_{\delta_i'}}(M_i) < 2E_B(L^BM_i)$. This implies, using also \eqref{eq:Rd3} and Lemma~\ref{lem:triangle}, that for large enough $i$
      \[ D_{W_{\delta_i}}(M_i) &\leq  C( E_B(L^BM_i) + d_i B L^{\alpha_1' B}) \\
        &\leq Cd_i B L^{\alpha_1' B}, \]
      increasing $C$ further. 
      Then by Lemma~\ref{lem:triangle}
      \[ D_{T_{\delta_i}}(M_i) &\leq Cd_i BL^{\alpha_1'B} + CB|h(\delta_i)|\\
        &\leq Cd_i BL^{\alpha_1'B}, \]
      using $|h(\delta_i)| \leq \epsilon^{-1}d_i$ again, absorbing further factors into $C$. 
      At the same time by \eqref{eq:db4} we have $D_{T_{\delta_i}}(M_i) \geq L^{\alpha_2'B} d_i$, so we obtain $L^{\alpha_2'B} \leq CBL^{\alpha_1'B}$ for sufficiently large $i$. Since $\alpha_2' > \alpha_1'$, this is a contradiction if $B$ is sufficiently large. 
    \end{itemize}
This completes the proof of the Proposition. 
\end{proof}

The proposition implies our main result, Theorem~\ref{thm:main}, which we restate here. 
\begin{thm}\label{thm:main2}
 Let $M\in\mathcal{M}$ be a minimal surface in $B_1$, with $C\times \mathbf{R}$ as a multiplicity one tangent cone at the origin. Then the tangent cone at the origin is unique. 
\end{thm}
\begin{proof}
  We apply Proposition~\ref{prop:EBdecay} repeatedly. Let us suppose that $L^{kB}M$ is sufficiently close to $C\times \mathbf{R}$ on $B_1$ to apply the proposition for $k=0,1,\ldots, N$. Define $e_k = E_B(L^{kB}M)$ and let $k_1,k_2, \ldots, k_m$ denote the values of $k$ for which alternative (ii) holds, so that alternative (i) holds for the remaining values. We then have
  \[ e_0 + e_1 + \ldots + e_N &= (e_0 + \ldots + e_{k_1-1}) + (e_{k_1} + \ldots + e_{k_2-1}) + \ldots + (e_{k_m} + \ldots + e_N) \\
    &\leq 2e_0 + 2e_{k_1} + \ldots + 2e_{k_m} \\
    &\leq 2e_0 + 2C\sum_{j=1}^m (\mathcal{A}(L^{k_j B}M)^\theta - \mathcal{A}(2L^{k_j B}M)^\theta) \\
    &\leq 2e_0 + 2C(\mathcal{A}(L^{k_1 B}M)^\theta - \mathcal{A}(2L^{k_m B}M)^\theta) \\
    &\leq 2e_0 + 2C\mathcal{A}(M)^\theta\]
  using the monotonicity of the area.

  Using Lemma~\ref{lem:DFbound} we find that
  \[ \label{eq:dFM1}
    d_{\mathcal{F}}(M, L^{(N+1)B}M) &\leq d_{\mathcal{F}}(M, LM) + \ldots d_{\mathcal{F}}(L^{NB}M, L^{NB+1}M) \\
    &\leq C(e_0 + e_1 + \ldots + e_N) \\
    &\leq C(e_0 + \mathcal{A}(M)^\theta). \]
If $M$ is sufficiently close to $C\times\mathbf{R}$ on $B_1$, then this implies that $L^{(N+1)B}M$ is also close enough to $C\times\mathbf{R}$ to keep applying Proposition~\ref{prop:EBdecay}. We can therefore take $N$ above to be arbitrarily large.  Since $C\times \mathbf{R}$ is a tangent cone, we know that for a sequence $N_i\to \infty$ we have $d_{\mathcal{F}}(L^{N_iB}M, C\times\mathbf{R}) \to 0$ and by applying the argument above to $L^{N_iB}M$ instead of $M$, we then find that $d_{\mathcal{F}}(L^{kB}M, C\times\mathbf{R})\to 0$ as $k\to\infty$. Therefore $C\times\mathbf{R}$ is the unique tangent cone at the origin.
\end{proof}

\section{The case of $C\times \mathbf{R}^k$}\label{sec:CRk}
In this section we reprove Simon's uniqueness result, Theorem~\ref{thm:CRk}, for tangent cones of the form $C\times \mathbf{R}^k$, where $C \subset \mathbf{R}^n$ satisfies the following conditions (see Simon~\cite{Simon94})

\bigskip
\noindent{\bf Conditions (\ddag)}
\begin{itemize}
\item[(a)] $C$ is strictly stable and strictly minimizing in the sense of Hardt-Simon~\cite{HS85}.
\item[(b)] All homogeneous degree 0 and degree 1 Jacobi fields on $C$ are generated by rotations and translations in $\mathbf{R}^n$.
\item[(c)] There are no homogeneous Jacobi fields on $C$ with degree $d \in \mathbf{Z} \cap\left( \frac{3-n}{2}, 0\right)$.
\end{itemize}

\bigskip
The cones $C(S^2\times S^4)$ and $C(S^3\times S^3)$ in $\mathbf{R}^8$ do not satisfy this, because they have the Jacobi field $r^{-2}$. However many other minimal cones do, such as $C(S^p\times S^q)$ with $p+q > 6$, as discussed in \cite{Simon94}. 

The strategy
to  prove Theorem~\ref{thm:CRk} is very similar to what we used for $C(S^3\times S^3)\times\mathbf{R}$, but there are significant simplifications. The cones $C\times\mathbf{R}^k$ we consider here are integrable, and in fact all degree one Jacobi fields arise from rotations in $\mathbf{R}^{n+k}$. As a consequence we do not need to construct surfaces similar to $T_\delta$, which also means that it is enough to show non-concentration relative to the cones $C\times\mathbf{R}^k$. This simplifies both the definition of the distance and the construction of barrier surfaces.

According to Hardt-Simon~\cite{HS85}, there are smooth minimal hypersurfaces $H_{\pm}$ in the two connected components of $\mathbf{R}^n\setminus C$, asymptotic to the graphs of $\pm r^\mu\phi(\omega)$ over $C$.  Here $\phi$ is the eigenfunction of $-L_\Sigma$ on the link $\Sigma$ with the smallest eigenvalue, and $r^\mu\phi(\omega)$ is the corresponding Jacobi field on $C$. We have $\mu\in \left(\frac{3-n}{2}, 0\right)$. The role of Condition (\ddag c) is the following, analogous to Lemma~\ref{lem:zetaJacobi}. 
\begin{prop}
  Let $u$ be a homogeneous degree one function on $C\times\mathbf{R}^k$ such that $r^{-\mu} u$ is locally bounded away from the origin, and $L_{C\times\mathbf{R}^k} u =0$. Then $u$ corresponds to a rotation in $\mathbf{R}^n\times\mathbf{R}^k$. More precisely, writing $z_i=x_i$ for $i\leq n$ and $z_i = y_{i-n}$ for $n < i \leq n+k$, we have $u(z) = Az \cdot \nu(z)$ for $A\in \mathfrak{so}(n+k)$ and $\nu(z)$ the unit normal to $C\times\mathbf{R}^k$. 
\end{prop}
\begin{proof}
  We follow Simon~\cite[Appendix 1]{Simon94}. We have
  \[ u = \sum_j r^{\mu_j} \sum_{k, \mathbf{l} \geq 0} a^j_{k,\mathbf{l}} r^{2k} y^{\mathbf{l}} \phi_j(\omega). \]
  Here $\phi_j$ is the $j^{\mathrm{th}}$ eigenfunction of $-L_{\Sigma}$ on the link of $C$, and $r^{\mu_i}\phi_j(\omega)$ are the corresponding homogeneous Jacobi fields on $C$. We have $\mu_0=\mu$. In addition $\mathbf{l} \in (\mathbf{Z}_{\geq 0})^k$ denotes a multiindex, and we write $y^{\mathbf{l}} = y_1^{l_1}\cdots y_k^{l_k}$, $|\mathbf{l}| = l_1 + \ldots + l_k$. Since $u$ has degree one, we must have $\mu_j + 2k + |\mathbf{l}| = 1$. Since $2k, |\mathbf{l}|\geq 0$ are integers, by Condition (\ddag c) the only possibilities are
  \begin{itemize}
  \item[(i)] $\mu_j=0, k=0, |\mathbf{l}|=1$: the corresponding functions $u$ are spanned by $\nu_j(z) y_i$ where $j=1,\ldots, n$ and $i=1, \ldots, k$ (note that $\nu_j=0$ for $j > n$). 
  \item[(ii)]  $\mu_j=1, k=0, |\mathbf{l}|=0$: the corresponding functions $u$ are spanned by $ \nu_j x_i - \nu_i x_j$, for $i,j=1,\ldots, n$.
  \end{itemize}
  The functions $u$ spanned by these can be written as $u = Az \cdot \nu(z)$, for $A\in \mathfrak{so}(n+k)$. Note that since $\nu_j =0$ for $j > n$, the coefficients $A_{ij}$ for $i > n$ do not affect $u$.  
\end{proof}

For $\lambda\in \mathbf{R}$ we will let $\lambda H = |\lambda| H_{\pm}$, using $H_+$ or $H_-$ depending on the sign of $\lambda$. Note that $\lambda H$ is asymptotic to the graph of $\lambda^{1-\mu} r^\mu$. We let $w=1-\mu$, so in our earlier setting of $C=C(S^3\times S^3)$ we have $\mu=-2, w=3$. As before we will write $(x, y)\in \mathbf{R}^n\times \mathbf{R}^k$ and $r = |x|, \rho = (|x|^2 + |y|^2)^{1/2}$. For a function $f(y)$ we will denote by $f(y)H \subset\mathbf{R}^n\times \mathbf{R}^k$ the hypersurface given for each $y\in \mathbf{R}^k$ by $f(y)H$ in the slice $\mathbf{R}^n\times \{y\}$. We will also write $V=C\times\mathbf{R}^k$. 

\begin{definition}
  For $d > 0$ let us define the neighborhood $N_d(V)$ to be the region between the surfaces $\pm d^{1/w} H$. For an open set $U\subset \mathbf{R}^n\times\mathbf{R}^k$ and a hypersurface $M$
  we define the distance $D_V(M;U)$ to be the infimum of all $d >0$ such that $M\cap U\subset N_d(V)$. 
\end{definition}

In analogy with Lemma~\ref{lem:Fadefn} we can define functions $F_a$ on $\pm H$ for $a > \mu$, satisfying  $F_a = r^a$ for sufficiently large $r$ and $L_{\pm H} F_a > C_a^{-1} r^{a-2}$. We extend $F_a$ to $\mathbf{R}^n\setminus \{0\}$ to be homogeneous of degree $a$. Then on all scalings of $H$ we have $L_{\lambda H} F_a > C_a^{-1}r^{a-2}$ and $|F_a| \leq C_ar^{a-2}$. Using this $F_a$ we construct barrier surfaces analogous to those in Proposition~\ref{prop:barrier1}. The construction is simpler since our surfaces will remain on one side of the cone $C$, i.e. the function $f$ is positive. In particular we do not have to deal with the case $|f(0)| \leq C_0R$ that appears in the proof of Proposition~\ref{prop:barrier1}. 

\begin{prop}\label{prop:barrier2}
  Let $f: \Omega \to (0,\infty)$ be a $C^2$ function on $\Omega\subset \mathbf{R}^k$. There is a large $C_f > 0$ and small $r_0, \epsilon_0 > 0$ depending on upper bounds for the $C^2$-norms of $f, f^{-1}$ satisfying the following. For $\epsilon \in (0,\epsilon_0)$ define the surface $X$ to be the graph of $-C_f \epsilon f(y) F_{\mu+2}$ over $\epsilon^{1/w} f(y) H$, in the slices $\mathbf{R}^n\times \{y\}$ for $y\in \Omega$, and $r < r_0$. Then $X$ has negative mean curvature. 
\end{prop}
\begin{proof}
  The proof is essentially identical to that of Proposition~\ref{prop:barrier1}. We work around the point $y=0$, in the region defined $-R < y < R$ and $R/2 < r < 2R$, for $R < r_0$. The scaled up surface $R^{-1}X$ is the graph of the function $B$ over $E(y)\cdot H$, where
  \[ E(y) &= R^{-1} \epsilon^{1/w} f(Ry) , \\
    B(x, y) &= -R^{\mu+1}C_f \epsilon f(Ry) F_{\mu+2}. \]
  By taking $r_0$ small (depending on the lower bound for $f$), we can arrange that $|f(0)| \geq C_0R$ for any large $C_0$, so we are in the setting analogous to the first case studied in Proposition~\ref{prop:barrier1}. In the same way as in that proof, we find that $R^{-1}X$ can be viewed as the graph of $B$ over the graph of $A$ over the surface $E(0)\cdot H \subset \mathbf{R}^n\times \mathbf{R}^k$, where
  \[ A = E(0) \Phi_{E(0)^{-1}E(y) - 1} (E(0)^{-1}\cdot). \]
  Here $\Phi_t$ is the function on $H$ such that $(1+t)H$ is the graph of $\Phi_t$ over $H$ for small $t$. In terms of the corresponding Jacobi field $\Phi$ we have $\Phi_t(x) = t\Phi(x) + O(t^2 |x|^{\mu})$ and $\Phi(x) = O(|x|^\mu)$. Since $E(0)^{-1} E(y) = f(0)^{-1} f(Ry)$, we have that $E(0)^{-1}E(y) - 1 = O(R)$, and so
  \[ A \lesssim (R^{-1} \epsilon^{1/w} f(0))^w R = R^\mu \epsilon f(0)^w. \]
  The same estimate holds for the derivatives of $A$, and so 
  \[ A, \nabla^i A &\lesssim \epsilon R^\mu \\
    B, \nabla^i B &\lesssim \epsilon C_f R^{\mu+1}, \]
  where $a \lesssim b$ means $|a|\leq Cb$ for $C$ depending on the $C^2$ norms of $f, f^{-1}$. The mean curvature of $X$ is
  \[ m = L_{E(0)\cdot H \times \mathbf{R}^k} (A+B) + O(A^2 + B^2). \]
  At $y=0$ we have
  \[ L_{E(0)\cdot H \times \mathbf{R}^k}A &= \Delta E(0) \Phi(E(0)^{-1}\cdot) + O( R^{1+\mu}\epsilon) \\
    &\lesssim R^{1+\mu}\epsilon, \]
  using that $\Delta E(0) \lesssim R\epsilon^{1/w}$ and $\Phi(E(0)^{-1}\cdot) \lesssim R^\mu \epsilon^{-\mu/w}$.
  As for $LB$ we have
  \[ L_{E(0)\cdot H \times \mathbf{R}^k}B &\leq -C^{-1} R^{\mu+1}C_f \epsilon f(0) + C R^{\mu+3} C_f \epsilon \Delta f(0), \]
  for some $C > 0$. It follows that
  \[ m &\leq -C^{-1} R^{\mu+1}C_f \epsilon f(0) + CR^{\mu+3}C_f \epsilon + CR^{1+\mu}\epsilon + \epsilon^2R^{2\mu} + \epsilon^2C_f^2 R^{2\mu + 2} \\
    &\leq  R^{\mu+1}\epsilon C_f\Big[ -C^{-1}f(0) + CR^2 + CC_f^{-1} + C_f^{-1}\epsilon R^{\mu-1} + \epsilon C_f R^{\mu+1}\Big]. \]
  Note that a point on $X$ with $y=0, r\in (R/2, 2R)$ can only exist if $R \geq C^{-1} \epsilon^{1/w}f(0)$ for a uniform $C > 0$. In particular $R^{\mu-1}\epsilon \lesssim 1$. 
 We first choose $C_f$ large (depending on the lower bound for $f$), to ensure that $CC_f^{-1} + C_f^{-1}\epsilon R^{\mu-1} < \frac{1}{4} C^{-1} f(0)$. Then we choose $R$ sufficiently small (depending on $C_f$), to ensure $C_f\epsilon R^{\mu-1}R^2 < \frac{1}{4} C^{-1}f(0)$. It follows that we will then have $m < 0$ as required. 
\end{proof}

Given these barrier surfaces, we can prove a non-concentration result analogous to Proposition~\ref{prop:nonconcentration}. The statement is identical, except we only need to consider the distance to the cone $V$. 
\begin{prop}\label{prop:nonconc2}
  There is a $C > 0$, such that given any $\gamma > 0$, there are $r_0=r_0(\gamma) > 0, d_0=d(\gamma, r_0)$ satisfying the following. Suppose that $M\in\mathcal{M}$ is a minimal surface with $D_V(M; B_1\setminus B_{1/2}) = d < d_0$, that on the region $\{ r > r_0\}$ in the annulus $B_1\setminus B_{1/2}$ can be written as the graph of a function $u$ over $V$. Then
  \[ \label{eq:nc2} D_V (M; B_{4/5}\setminus B_{3/5}) \leq C\sup_{(B_1\setminus B_{1/2})\cap \{r > r_0\}} |r^{-\mu} u| + C\gamma d. \]
\end{prop}
\begin{proof}
  The proof is essentially the same as the first part of the proof of Proposition~\ref{prop:nonconcentration}. Fix $\gamma > 0$. Let $\Omega_\gamma\subset\mathbf{R}^k$ denote the annulus $\frac{1}{2} +\gamma < |y| < 1-\gamma$. Define the function
  \[ f(y) = (|y| - 1/2)^{-1} + (1-|y|)^{-1} \]
  for $y\in \Omega_\gamma$. We then have $C^2$ bounds for $f$ on $\Omega_\gamma$, depending on $\gamma$. By Proposition~\ref{prop:barrier2} we have the surfaces $X_\epsilon$ defined in the region $\{r < r_0\} \times \Omega_\gamma$, for sufficiently small $r_0$. By a result similar to Proposition~\ref{prop:Xc}, given a small $c > 0$ and replacing $r_0$ by a smaller constant, we can assume that $X_\epsilon$ lies between the surfaces $(\epsilon f(y)^w \pm c\epsilon)^{1/w}H$. We can choose $c$ small enough (independent of $\gamma$) so that $\frac{1}{2} f(y)^w  + c < f(y)^w  < 2f(y)^w - c$, and $X_\epsilon$ lies between $(\frac{1}{2}f(y)^w \epsilon)^{1/w} H$ and $(2f(y)^w\epsilon)^{1/w} H$. In particular if we denote by $U_\gamma$ the region $\{ r < r_0\} \times \Omega_\gamma$, then we have (using $w > 1$):
  \begin{itemize}
  \item[(i)] On the boundary pieces $\{r < r_0\} \times \partial \Omega_{\gamma}$ the surface $X_\epsilon$ lies on the positive side of  $(\frac{1}{2}\gamma^{-1} \epsilon)^{1/w}H$.
  \item[(ii)] On all of $U_\gamma$ the surface $X_\epsilon$ lies on the positive side of $(C^{-1} \epsilon)^{1/w}H$ for $C$ depending only on $f$.
  \item[(iii)] On the smaller region $\{ r < r_0\} \times \{3/5 < |y| < 4/5\}$ the surface $X_\epsilon$ lies on the negative side of $(C \epsilon)^{1/w}H$ for $C$ depending on $f$.  
 \end{itemize}

  Let us define $D(r_0) = \sup_{\{r > r_0\}} |r^{-\mu}u|$, and 
  \[ \epsilon_0 &= C' D(r_0) + C' \gamma d, \\
    \epsilon_1 &= C' d. \]
  For a suitable $C'$. If $C'\geq 2$ then we have the following. Since $(\frac{1}{2}\gamma^{-1}\epsilon_0)^{1/w} \geq d^{1/w}$, on $\{r < r_0\} \times \partial\Omega_\gamma$ the surface $M$ lies on the negative side of $X_{\epsilon_0}$. For fixed $r_0$, if $d$ is sufficiently small, and if $M$ is the graph of $u$ on the region $\{r > r_0\}$, then $|u| \leq D(r_0) r^\mu$. In particular along $\{r=r_0\}$ the surface $M$ is on the negative side of $(CD(r_0))^{1/w}H$ for a fixed $C>0$ once $D(r_0)$ is sufficiently small. It follows that on $\partial\{ r < r_0\} \times \Omega_\gamma$ the surface $M$ lies on the negative side of $X_{\epsilon_0}$ if $C'$ is sufficiently large. At the same time, $M$ lies on the negative side of $X_{\epsilon_1}$ on all of $\Omega_\gamma$. Considering the family of surfaces $X_{\epsilon}$ ranging between $\epsilon_0$ and $\epsilon_1$ we find that $M$ must lie on the negative side of $X_{\epsilon_0}$ as well.

  By property (iii) above, on the region $\{r < r_0\} \times \{3/5 < |y| < 4/5\}$ the surface $X_{\epsilon_0}$ lies on the negative side of $(C\epsilon_0)^{1/w}H$, and therefore so does $M$. Reversing the orientations we find that on the same region $M$ lies on the positive side of $-(C\epsilon_0)^{1/w}H$. If $d$ is small enough, then $M$ lies between $\pm (C\epsilon_0)^{1/w}H$ on the region $\{r > r_0\}$ as well, and so $D_V(M; B_{4/5}\setminus B_{3/5}) \leq C\epsilon_0$, which is the required bound \eqref{eq:nc2}. 
\end{proof}

The three annulus Lemma~\ref{lem:L23annulus} holds for the cone $C\times\mathbf{R}^k$, and we obtain corresponding $\alpha_1 < \alpha_2$ and $\rho_0$. The $L^2$ to $L^\infty$ estimate (Lemma~\ref{lem:L2Linfty}) also holds, in the form
\[ \sup_{B_{1/2}(0)} |r^{-\mu} u| \leq C \Vert u\Vert_{L^2(B_1)}, \]
for any Jacobi field $u$ on $C\times\mathbf{R}^k$ such that $r^{-\mu}u \in L^\infty(B_1)$.  Proposition~\ref{prop:nonconc2} together with these two results implies the following, with essentially the same proof as Proposition~\ref{prop:D3annulus}. As before, we use the notation $D_V(M) = D_V(M; B_1\setminus B_{\rho_0})$. 
\begin{prop}\label{prop:D3ann2}
  There is an $L > 0$ such that for sufficiently small $d > 0$ we have the following. Suppose that $M\in\mathcal{M}$ is a minimal hypersurface in $B_1$. Suppose that $D_V(M) < d$ and $\alpha\in (\alpha_1, \alpha_2)$. Then
  \begin{itemize}
  \item[(i)] If $D_V(LM) \geq L^\alpha D_V(M)$, then $D_V(L^2M) \geq L^\alpha D_V(LM)$.
  \item[(ii)] If $D_V(LM)\geq L^\alpha D_V(L^2M)$, then $D_V(M) \geq L^\alpha D_V(LM)$.
  \end{itemize}
  \end{prop}

  To prove the uniqueness, we let $\mathcal{V}$ denote the cone $V$ and all of its rotations. In analogy with Proposition~\ref{lem:nearbycones}, all minimal cones sufficiently close to $V$ must be in the family $\mathcal{V}$. For $B > 0$ we define
  \[ E_B(M) = \inf\{ D_V(M) + D_V(L^BM) \,:\, V\in\mathcal{V}\}, \]
as in \eqref{eq:EBdefn}. 
  
  The uniqueness of the tangent cone is implied by the following decay estimate, whose proof is essentially identical to that of Proposition~\ref{prop:EBdecay}, except case (ii) of that proposition does not appear here. 
  \begin{prop}\label{prop:decay10}
    There is a $B > 0$ with the following property. Suppose that $M\in\mathcal{M}$ is a minimal surface in $B_1$ with density equal to that of the cone $C\times \mathbf{R}^k$ at the origin. If the Hausdorff distance from $M$ to $C\times\mathbf{R}^k$ on $B_1$ is sufficiently small, then we have
    \[ E_B(L^BM) \leq \frac{1}{2} E_B(M). \]
  \end{prop}
  \begin{proof}
    We can follow the proof of Proposition~\ref{prop:EBdecay} closely, but the situation now is much simpler. Cases (a1) and (b1) that appeared in the earlier proof can not happen here, since our comparison surfaces $RV \in \mathcal{V}$ are all cones (we can think of setting $\delta=0$ in the earlier proof). In particular option (ii) in the statement of Proposition~\ref{prop:EBdecay} does not happen. 
  \end{proof}

  The uniqueness of the tangent cone follows as before, although in this case polynomial convergence to the tangent cone follows.
  \begin{proof}[Proof of Theorem~\ref{thm:CRk}]
    We can use the same argument as in the proof of Theorem~\ref{thm:main2} to prove that $L^{kB}M \to C\times \mathbf{R}^k$ on the unit ball. In addition, since case (ii) of Proposition~\ref{prop:EBdecay} does not appear in Proposition~\ref{prop:decay10}, the inequality analogous to \eqref{eq:dFM1} implies (replacing $M$ by $L^{kB}M$ and letting $N\to\infty$)
    \[ d_{\mathcal{F}}(L^{kB}M, C\times\mathbf{R}^k) \leq C E_B(L^{kB}M) < C2^{-k} E_B(M), \]
    for all $k > 0$ if $M$ is already sufficiently close to the cone. This is enough to conclude the polynomial convergence. 
  \end{proof}

  \begin{remark}
    It is natural to ask whether Theorem~\ref{thm:main} can be extended to the case $C_S\times\mathbf{R}^k$ for $k > 1$, where $C_S=C(S^3\times S^3)$. In this case there are many more degree one Jacobi fields, namely for any homogeneous degree 3 polynomial $P(y)$ we have the Jacobi field $\phi_P = P(y) r^{-2} - \frac{1}{6}\Delta P(y)$. The method of proof of Theorem~\ref{thm:main} would require us to construct corresponding perturbations of the cone $C_S\times\mathbf{R}^k$, and the fact that the link of $C_S\times\mathbf{R}^k$ has non-isolated singular set for $k > 1$ introduces substantial new difficulties. 
  \end{remark}

\end{document}